\documentclass{amsart}
\usepackage[utf8]{inputenc}    
\usepackage[T1]{fontenc}
\usepackage[english]{babel}     
\usepackage[a4paper, left=1in, right=1in, top=1in, bottom=1in]{geometry}
\usepackage{comment}
\usepackage{indentfirst}
\usepackage{graphicx}           
\usepackage{amsmath,amssymb}    
\usepackage{amsthm}             
\usepackage{mathtools}          
\usepackage{enumitem}           
\usepackage{listings}           
\usepackage{todonotes}          
\usepackage{hyperref}           
\usepackage[cal=boondox]{mathalfa}
\usepackage{relsize}
\usepackage{stmaryrd}
\usepackage{setspace}
\usepackage{chngcntr}

\counterwithin{equation}{subsection}
\counterwithin{figure}{subsection}

\setenumerate[1]{label={(\arabic*)}} 
\setlist[enumerate]{itemsep=0pt}
\setlist[itemize]{itemsep=0pt}

\newcommand\restr[2]{{
  \left.\kern-\nulldelimiterspace 
  #1 
  \right |_{#2} 
  }}

\newtheorem{theorem}{Theorem}[section]
\newtheorem{corollary}{Corollary}[theorem]
\newtheorem{lemma}[theorem]{Lemma}
\newtheorem{proposition}[theorem]{Proposition}

\newtheorem{innercustomthm}{Theorem}
\newenvironment{customthm}[1]
  {\renewcommand\theinnercustomthm{#1}\innercustomthm}
  {\endinnercustomthm}

\theoremstyle{definition}
\newtheorem{definition}[theorem]{Definition}
\newtheorem{defprop}[theorem]{Definition/Proposition}

\newtheorem*{remark}{Remark(s)}

\newcommand\Tau{\mathcal{T}}
\newcommand*\interior[1]{\mathring{#1}}

\DeclarePairedDelimiter\ceil{\lceil}{\rceil}
\DeclarePairedDelimiter\floor{\lfloor}{\rfloor}

\begin{document}

\title{On global rigidity of transversely holomorphic Anosov flows}
\author{Mounib Abouanass}
\date{\today}

\maketitle

\begin{abstract}
In this paper, we study transversely holomorphic flows, i.e. those whose holonomy pseudo-group is given by biholomorphic maps. We prove that for Anosov flows on smooth compact manifolds, the strong unstable (respectively, stable) distribution is integrable to complex manifolds, on which the flow acts holomorphically. Furthermore, assuming its complex dimension to be one, it is uniquely integrable to complex affine one-dimensional manifolds, each moreover affinely diffeomorphic to $\mathbb C$, on which the flow acts affinely. In this case, the weak stable (respectively, unstable) foliation is transversely holomorphic, and even transversely projective if the flow is assumed to be topologically transitive.
By combining these facts in low dimensions, our main result is as follows: if a transversely holomorphic Anosov flow on a smooth compact five-dimensional manifold is topologically transitive, then it is either $C^\infty$ orbit equivalent to the suspension of a hyperbolic automorphism of a complex torus, or, up to finite covers, $C^\infty$-orbit equivalent to the geodesic flow of a compact hyperbolic manifold.
\end{abstract}

\tableofcontents

\section{Introduction}
Smooth flows on smooth manifolds have long been of interest to both mathematicians and physicists. Their approach can be either dynamical (by considering the one-parameter subgroup of smooth diffeomorphisms $(\varphi^t)_{t \in \mathbb R}$, using various tools coming from dynamical systems and ergodic theory), or geometrical (by considering the the partition of the phase space into orbits, i.e. the orbit foliation).

A classical approach in continuous dynamics involves assuming transverse structures for the flow.
For example, Brunella and Ghys (see \cite{brunella_umbilical_1995}, \cite{brunella_transversely_1996}, \cite{ghys_transversely_1996}) have studied transversely holomorphic flows on smooth three-manifolds, that is smooth flows whose orbit foliation $\Phi$ can be given by a smooth foliated atlas whose transition maps are holomorphic.
They achieved a complete classification using advanced topological and analytical techniques, see also \cite{carriere_flots_1984}.
However, their techniques fail for higher dimensions due to several reasons: a priori non-conformity, non-equivalence of harmonic functions and holomorphic maps, and the lack classification of compact complex manifolds of complex dimension three, among others.

On the other hand, Anosov flows exhibit rich dynamical properties. 
The fundamental structural stability theorem (see \cite{katok_introduction_1995}) shows that a small $C^1$-perturbation of such a flow results in another Anosov flow $C^1$-orbit equivalent to the initial one, meaning there exists a diffeomorphism sending orbits to orbits without considering parametrization. However, from a dynamical perspective, they are not identical and many different examples of such flows can be find in the literature (see \cite{barbot_caracterisation_1995}). 

For instance, algebraic Anosov flows (that is finite covers of suspensions of Anosov diffeomorphisms or geodesic flows on hyperbolic manifolds) have been studied by Ghys. Their geometric properties are closely linked with the structure of their weak stable and unstable foliations: for example, Plante showed in \cite{plante_anosov_1981} that if the weak stable or weak unstable foliation is transversally affine, then the flow is $C^\infty$-orbit equivalent to the suspension of an Anosov diffeomorphism. 

Ghys proved classification results for smooth compact three-dimensional manifolds, initially by assuming regularity of the strong stable and unstable foliations (which depend on parametrization) in \cite{ghys_flots_1987}, and later by assuming regularity of the weak stable/unstable foliations, along with a volume-preserving hypothesis (which both does not depend on parametrization) in \cite{ghys_deformations_1992}. 
This latter assumption is strong and essential: he constructed examples of "exotic" (i.e. non algebraic) topologically transitive flows which are not volume preserving.

In \cite{ghys_rigidite_1993}, he showed that if the weak stable foliation is $C^{1,1}$, the strong unstable leaves can be uniquely equipped with complete real affine structures in such a way that the the weak stable foliation is transversally projective. By the classification of \cite{barbot_caracterisation_1995}, he concluded without the volume-preserving assumption. Ghys extended these ideas to the low-dimensional complex case (see \cite{ghys_holomorphic_1995}), for holomorphic Anosov diffeomorphisms and flows, achieving a classification without additional assumptions, by rigidity of holomorphicity.

Under certain geometric or regularity assumptions (see \cite{fang_rigidity_2007}), Fang managed to classify classes of uniformly quasi-conformal flows, and in particular transversally holomorphic transversally symplectic flows on smooth compact five-dimensional manifolds.

By combining the rigid properties of holomorphic maps with the geometric richness of Anosov flows, one can hope to achieve classification results for transversely holomorphic Anosov flows, that is Anosov flows whose orbit foliation is transversely holomorphic, in low-dimensional cases, such as five-dimensional manifolds, without use of strong hypothesizes such as the volume-preservation.
The main results of this paper are the following.

\begin{customthm}{A}
Let $(\varphi^t)_{t\in \mathbb R}$ a transversely holomorphic Anosov flow on a smooth compact manifold $M$. Suppose the strong unstable distribution $E^{u}$ is of complex dimension $1$.\\
Then there exists a unique family of complex affine structures on the strong unstable leaves $(\mathcal{F}^{u}_x)_{x\in M}$ such that: for every $x \in M$ and every $t \in \mathbb R$, the map 
    $\restr{\varphi^t}{\mathcal{F}^{u}_{x}}: \mathcal{F}^{u}_{x}\to \mathcal{F}^{u}_{\varphi^t(x)}$
    is affine.\\
Moreover, each of these complex affine structures is complete.
\end{customthm}
This result is obtained in subsection \ref{subsec:6.1} as Theorem \ref{thm:compaffstru}. 
By a complex affine structure for a complex manifold $L$, we mean a holomorphic atlas whose change of coordinates are complex affine diffeomorphisms. \\
Similarly, a transverse projective structure for a foliation $\mathcal{F}$ is given by a foliated atlas whose transition maps are complex projective automorphisms. See the beginning of section \ref{sec:6} for more details.\\
As a consequence, we are able to prove in subsection \ref{subsec:6.2} (see Theorem \ref{thm:FsTrProj}):

\begin{customthm}{B}
    Let $(\varphi^t)_{t\in \mathbb R}$ a transversely holomorphic Anosov flow on a smooth compact manifold $M$. Suppose the strong unstable distribution $E^{u}$ is of complex dimension $1$. Suppose moreover that $(\varphi^t)_{t\in \mathbb R}$ is transitive.\\
    Then the weak stable foliation $\mathcal{F}^s$ is transversely projective.
\end{customthm}
Using these results in dimension five allows us to complete the classification in section \ref{sec:7} (see Theorem \ref{thm:classif}):
\begin{customthm}{C}
    Let $(\varphi^t)_{t \in \mathbb R}$ a transversely holomorphic Anosov flow on a smooth compact manifold $M$ of dimension five. Suppose $(\varphi^t)_{t \in \mathbb R}$ topologically transitive.\\
    Then $(\varphi^t)$ is either $C^\infty$-orbit equivalent to the suspension of a hyperbolic automorphism of a complex torus, or, up to finite covers, $C^\infty$-orbit equivalent to the geodesic flow on the unit tangent bundle of a compact hyperbolic three-dimensional manifold.
\end{customthm}
These results are analogous to that of Ghys in the continuous real case (see \cite{ghys_deformations_1992}, \cite{ghys_rigidite_1993}) and the holomorphic case (\cite{ghys_holomorphic_1995}).

In section \ref{sec:5} we consider an arbitrary transversely holomorphic Anosov flow on a smooth compact manifold (of any dimension). We prove that the strong unstable foliation and the strong stable foliation are continuous foliations with holomorphic leaves on which the action of the flow is holomorphic.\\
Thanks to theses properties, we show that the weak stable foliation $\mathcal{F}^s$ is transversely holomorphic under the assumption that the strong stable distribution $E^{u}$ is of complex dimension one.\\
In section \ref{sec:6}, the fundamental result is that there exists, in that case, a unique family of complete complex affine structures on the strong unstable leaves, holomorphically compatible with the ones introduced in section \ref{sec:5}, on which the action of the flow is affine, and the weak stable foliation is transversely projective with respect to these affine structures if the flow is topologically transitive.\\
Eventually in section \ref{sec:7}, we prove our main result: a transversely holomorphic transitive Anosov flow on a smooth compact five-dimensional manifold is either $C^\infty$-orbit equivalent to the suspension of a hyperbolic automorphism of a complex torus, or, up to finite covers, $C^\infty$-orbit equivalent to the geodesic flow of a compact hyperbolic manifold.

\section{Reminders and definitions}

We only consider smooth ($C^\infty$) manifolds without boundary.
\subsection{Foliations}
We start by recalling some notions coming from the theory of foliations (see \cite{camacho_geometric_1985}, \cite{moerdijk_introduction_2003}, \cite{haefliger_groupoides_1984}, \cite{lee_manifolds_2009} and \cite{candel_foliations_2000}). 
\begin{defprop}
\label{defprop:2.1}
    Let $M$ a smooth manifold of dimension $n\in \mathbb N$. Let $r \geq 1$ (or $r=\infty$) and $k\leq n$.\\
    A \emph{$C^r$ foliation} $\mathcal{F}$ of dimension $k$ on $M$ (or codimension $n-k$) is defined by one of the following equivalent assertions:
    \begin{enumerate}
        \item a $C^r$-atlas $(U_i, \psi_i)_i$ on $M$ which is maximal with respect to the following properties:
            \begin{enumerate}
                \item For all $i$, $\psi_i(U_i)=U_i^1 \times U_i^2$, where $U_i^1$ and $U_i^2$ are connected open subsets of $\mathbb R^k$ and $\mathbb R^{n-k}$ respectively ;
                \item For all $i,j$, there exist $C^r$ maps $f_{ij}$ and $h_{ij}$ such that
                \[\forall(x,y) \in \psi_i(U_i \cap U_j) \subset \mathbb R^k \times \mathbb R^{n-k}, \; \psi_i \circ \psi_j^{-1}(x,y)=(f_{ij}(x,y), h_{ij}(y)).\]
            \end{enumerate}
        \item a maximal atlas $(U_i, s_i)_i$, where each $U_i$ is an open subset of $M$ and $s_i:U_i \to \mathbb R^{n-k}$ is a $C^r$ submersion, satisfying:
            \begin{enumerate}
                \item $\bigcup_i U_i = M$ ;
                \item For all $i,j$, there exists a $C^r$ diffeomorphism
                \[\gamma_{ij}: s_i(U_i\cap U_j) \to s_j(U_i\cap U_j) \text{  such that  } \restr{s_i}{U_i\cap U_j} = \restr{\gamma_{ij}\circ s_j}{U_i\cap U_j}. \]
            \end{enumerate}
        \item a partition of $M$ into a family of disjoint connected $C^r$ immersed submanifolds $(L_\alpha)_\alpha$ of dimension $k$  such that for every $x \in M$, there is a $C^r$ chart $(U, \psi)$ at $x$, of the form $\psi: U \to U^1 \times U^2$ where $U^1$ and $U^2$ are connected open subsets of $\mathbb R^k$ and $\mathbb R^{n-k}$ respectively, satisfying: for each $L_\alpha$, for each connected component $(U \cap L_\alpha)_\beta$ of $U\cap L_\alpha$, there exists $c_{\alpha,\beta} \in \mathbb R^{n-k}$ such that 
        \[ \psi ((U\cap L_\alpha)_\beta) = U^1 \times \{c_{\alpha, \beta}\}.\]
     \end{enumerate}
    The maps $h_{ij}$ and $\gamma_{ij}$ are called \textit{transition maps}.\\
    We will call (for each equivalent assertion $(1)$, $(2)$ and $(3)$):
    \begin{itemize}
        \item a \emph{plaque}:
            \begin{enumerate}
                \item a set of the form $\psi_i^{-1}(U_i \times \{c\})$, for $c \in U_i^2$ ;
                \item a connected component of a fiber of $s_i$ ;
                \item a connected component $(U\cap L_\alpha)_\beta$ of $U\cap L_\alpha$.
            \end{enumerate}
        \item a \emph{leaf} of the foliation an equivalence class for the following equivalence relation: two points $x$ and $y$ are equivalent if and only if there exist a sequence of foliation charts $U_1, \ldots, U_k$ and a sequence of points $x=p_0, p_1, \ldots, p_{k-1},p_k=y$ such that, for $i\in\llbracket 1,k\rrbracket$, $p_{i-1}$ and $p_{i}$ lie on the same plaque of $U_i$.
    \end{itemize}
    \end{defprop}
    We will represent abusively a foliation by the set of its leaves $\mathcal{F}$ and will note $\mathcal{F}_x$ the leaf of the foliation containing $x\in M$. 
    \begin{remark}
    \label{rem:2.1}
    \leavevmode
    \begin{itemize}
        \item A $C^r$ foliation $\mathcal{F}$ gives rise to a $C^{r-1}$ subbundle of $TM$, noted $T\mathcal{F}$, called the \textit{tangent bundle} to the foliation $\mathcal{F}$, whose fibers are the tangent spaces to the leaves. We also define the \textit{normal bundle} $\nu:= TM \diagup T\mathcal{F}$ of the foliation. It is a vector bundle whose transition functions are given by the differential of the transition maps $h_{ij}$.
        \item We can define in the same way a \emph{holomorphic} foliation on a complex manifold by replacing "$C^r$" in the above definition with "holomorphic" and $\mathbb R$ with $\mathbb C$.
    \end{itemize}
    \end{remark}

From the first and third formulations we can define a more general notion of foliation adapted to our purposes: 
\begin{definition}
    Let $M$ a smooth manifold of dimension $n\in \mathbb N$. Let $r \geq 1$ (or $r=\infty$) and $k\leq n$.\\
    A \emph{(topological) foliation} with $C^r$ leaves, $\mathcal{F}$, of dimension $k$ (or codimension $n-k$) on $M$ is defined by one of the following equivalent assertions:
    \begin{enumerate}
        \item a $C^0$-atlas $(U_i, \psi_i)_i$ on $M$ which is maximal with respect to the following properties:
            \begin{enumerate}
                \item For all $i$, $\psi_i(U_i)=U_i^1 \times U_i^2$, where $U_i^1$ and $U_i^2$ are connected open subsets of $\mathbb R^k$ and $\mathbb R^{n-k}$ respectively ;
                \item For all $i,j$, there exist $C^0$ maps $f_{ij}$ and $h_{ij}$ such that
                \[\forall(x,y) \in \psi_j(U_i \cap U_j) \subset \mathbb R^k \times \mathbb R^{n-k}, \; \psi_i \circ \psi_j^{-1}(x,y)=(f_{ij}(x,y), h_{ij}(y))\]
                and for every such $y$, the map $x\mapsto f_{ij}(x,y)$ is $C^r$.
            \end{enumerate}
        \item a partition of $M$ into a family of disjoint connected $C^r$ immersed submanifolds $(L_\alpha)_\alpha$ of dimension $k$  such that for every $x \in M$, there is a $C^0$ chart $(U, \psi)$ at $x$, of the form $\psi: U \to U^1 \times U^2$ where $U^1$ and $U^2$ are connected open subsets of $\mathbb R^k$ and $\mathbb R^{n-k}$ respectively, satisfying: for each $L_\alpha$, for each connected component $(U \cap L_\alpha)_\beta$ of $U\cap L_\alpha$, there exists $c_{\alpha,\beta} \in \mathbb R^{n-k}$ such that 
        \[ \psi ((U\cap L_\alpha)_\beta) = U^1 \times \{c_{\alpha, \beta}\}\]
        and the map $\operatorname{pr}_1\circ \psi|_{(U \cap L_\alpha)_\beta}:(U \cap L_\alpha)_\beta\to U^1$ is a $C^r$ chart for $L$.
     \end{enumerate}
     We call such an atlas a \emph{foliated atlas} adapted to $\mathcal{F}$. 
\end{definition}

\begin{remark}
    \leavevmode
        \begin{itemize}
        \item For every $y$, the map $x\mapsto f_{ij}(x,y)$ in the previous definition is automatically a $C^r$ diffeomorphism.  
        \item We can still define the tangent bundle $T\mathcal{F}$ to the foliation $\mathcal{F}$ if it is a foliation with $C^r$ leaves, $r\geq 1$.
        It is a vector bundle over $M$ whose transition maps are given by the differential along the first variable of the maps $f_{ij}$, ie by $d_1f_{ij}$.
        \item In case $T\mathcal{F}$ is a continuous subbundle of $TM$, the foliation $\mathcal{F}$ is said to be \emph{integral} (see \cite{wilkinson_dynamical_2008} and \cite{pugh_holder_1997} for a discussion about integrability of continuous subbundles of $TM$).
        In order to facilitate the statement of some results, we will call integral foliations \emph{$C^0$ foliations}.
        Note that every $C^r$ foliation, $r\geq 1$, is integral, and by definition an integral foliation has $C^1$ leaves.
        \item We can define in the same way (if $k=2k'$ is even) a \emph{foliation with holomorphic leaves} by replacing "$C^r$" in the above definition with "holomorphic".
        \end{itemize}
\end{remark}

A particular example of a $C^\infty$ foliation on a smooth compact manifold is that of a smooth flow $(\varphi^t)$ induced by a nowhere vanishing vector field $X$. 
In that case, the orbits of the flow define the leaves of a foliation which we usually denote by $\Phi$. Here are two fundamental examples:

\begin{itemize}
    \item \textit{Suspension of a diffeomorphism}: Let $f:F \to F$ a $C^r$-diffeomorphism on a smooth connected manifold $F$. \\
The group $\mathbb Z$ acts freely and properly discontinuously on $\mathbb R \times F$ by:
\[\text{For } k\in \mathbb Z \text{ and } (t,x)\in \mathbb R \times F, \quad k \cdot(t,x) = (t+k, f^{-k}(x)),\]
and maps leaves of the smooth foliation $\mathcal{F}=(\mathbb R \times \{x\})_{x \in F}$ to leaves.\\
The foliation $\mathcal{F}$ gives rise to a $C^r$-foliation on the \emph{suspension manifold of $f$} $F_f:= (\mathbb R \times F) \diagup \mathbb Z$, called the \emph{suspension of the diffeomorphism $f$}. In fact, this foliation is induced by the "vertical" flow, called the \emph{flow suspension of the diffeomorphism $f$}, $\varphi^t: \left \{\begin{array}{l}
        F_f\quad \to \quad F_f \\
        \overline{(s,x)} \; \mapsto \; \overline{(s+t,x)}
    \end{array} \right. . $
    \item \textit{Geodesic flow}: Let $(M,g)$ a Riemannian manifold. Let $\nabla$ the Levi-Civita connection of $(M,g)$. A \emph{geodesic} on $M$ is a curve $\gamma$ satisfying $\nabla_{\dot{\gamma}} \dot{\gamma} = 0$. It has necessarily constant speed.\\
By the fundamental theorem of ordinary differential equations, for every $v=(p, v_p)\in TM$, there exists a unique maximal geodesic $\gamma_v$ defined on a neighborhood of $0 \in \mathbb R$ such that $\gamma_v(0)=p$ and $\dot{\gamma_v}(0) = v_p$. \\
The Riemannian manifold $(M,g)$ is \emph{complete} if any maximal geodesic is defined on the whole $\mathbb R$. In that case, there exists a flow $(g^t)_{t \in \mathbb R}$ in the tangent bundle $TM$ of $M$ called the \emph{geodesic flow}, defined by:
$g^t: \left \{\begin{array}{ccl}
        TM\quad &\to& \quad TM \\
        v \; &\mapsto& \; (\gamma_v(t), \dot{\gamma_v}(t))
    \end{array} \right. . $
By the constant speed property of geodesics, we can still defined the geodesic flow on the unit sphere bundle $S^1M$ of $M$.
\end{itemize}

From now on, we will always consider $C^0$ foliations on a smooth manifold $M$.

\subsection{Holonomy}
    
We now define the fundamental concept of \textit{holonomy} for an integral foliation (see \cite{moerdijk_introduction_2003}, \cite{camacho_geometric_1985}).
In the following, a $C^0$ diffeomorphism means a homeomorphism. 

\begin{definition}
    A \emph{transversal} to the (integral) foliation $\mathcal{F}$ is a smooth submanifold $T$ of $M$ such that for every $x\in T$, $T_xT\oplus T_x\mathcal{F}=T_xM$.
\end{definition}
\begin{defprop}
\label{defprop:2.2}
    Let $ \mathcal{F}$ a $C^r$ foliation on a smooth manifold $M$ ($r \in \mathbb N \cup \{\infty\}$). Let $x,y$ belonging to the same leaf $L$ of $\mathcal{F}$ and $\alpha$ a path from $x$ to $y$ in $L$. Let also $T,S$ two small transversal to $\mathcal{F}$ at $x$ and $y$ respectively.
    \begin{itemize}
        \item If there exists a foliation chart $U$ containing $\alpha([0,1])$, we can define a germ of $C^r$ diffeomorphism $h(\alpha)^{S,T}$ from a small open subset $A$ of $T$ to an open subset of $S$ such that: \begin{enumerate}
            \item $h(\alpha)^{S,T}(x)=y$ ;
            \item For any $x'\in A$, $h(\alpha)^{S,T}(x')$ lies on the same plaque in $U$ as $x'$.
        \end{enumerate}
        Moreover, the germ $h(\alpha)^{S,T}$ does not depend on $U$ nor the path in $L\cap U$ connecting $x$ and $y$.
        \item In the general case, we can choose a sequence of foliation charts $U_1, \ldots, U_k$ such that for all $i$, $\alpha([\frac{i-1}{k}, \frac{i}{k}]) \subset U_i$, and a sequence $(T_i)_i$ of transversal sections of $\mathcal{F}$ at $\alpha(\frac{i}{k})$, with $T_0=T$ and $T_k=S$. Let $\alpha_i$ a path in $L\cap U_i$ from $\alpha(\frac{i-1}{k})$ to $\alpha(\frac{i}{k})$. \\
        We define 
        \[h(\alpha)^{S,T}:=h(\alpha_k)^{T_k, T_{k-1}} \circ \cdots \circ h(\alpha_1)^{T_1, T_0}.
        \]
        Moreover, $h(\alpha)^{S,T}$ depends only on $T, S$ and $\alpha$, and satisfies the same properties $(1), (2)$.
    \end{itemize}
    The germ of $C^r$ diffeomorphism $h(\alpha)^{S,T}$ is called the \emph{$\mathcal{F}$-holonomy} of the path $\alpha$ with respect to the transversals $T$ and $S$.
\end{defprop} 
    We list some basic properties of holonomy maps:
    \begin{enumerate}[label=(\roman*)]
        \item If $\alpha$ is a path in $L$ from $x$ to $y$ and $\beta$ a path in $L$ from $y$ to $z$, and if $T$, $S$ and $R$ are transversal sections of $\mathcal{F}$ at $x$, $y$ and $z$ respectively, then
        \[h(\beta \alpha)^{R,T}=h(\beta)^{R,S} \circ h(\alpha)^{S,T},\]
        where $\beta\alpha$ is the concatenation of the paths $\alpha$ and $\beta$ ;
        \item If $\alpha$ and $\beta$ are homotopic paths in $L$ relatively to endpoints from $x$ to $y$, and if $T$ and $S$ are transversals sections at $x$ and $y$ respectively, then $h(\beta)^{S,T}= h(\alpha)^{S,T}$.
        \item If $\alpha$ is a path in $L$ from $x$ to $y$, and if $T,T'$ and $S,S'$ are pairs of transverse sections at $x$ and $y$, respectively, then \[h(\alpha)^{S',T'}=h(\overline{y})^{S',S} \circ h(\alpha)^{S,T} \circ h(\overline{x})^{T,T'},\] where $\overline{x}$ is the constant path with image $x$.
    \end{enumerate}
    \begin{remark}
    \label{rem:2.2}
    \begin{itemize}
    \leavevmode
        \item The construction of the holonomy map between two points in the same plaque of the same foliation chart $(U, \psi)$ is as follows. Keep the notations of Definition/Proposition \ref{defprop:2.2}. Denote by $n$ the dimension of $M$ and $q$ the codimension of $\mathcal{F}$. 
        $\psi(T) , \psi(S)\subset \mathbb R^{n}$ are vertical graphs of $C^r$ maps from $\mathbb R^q$ to $\mathbb R^{n-q}$.
        By projecting these graphs on $\mathbb R^q$, we obtain two $C^r$ diffeomorphisms $\Psi: T \to \mathbb R^q$ and $\Psi': S \to \mathbb R^q$. Since the leaves of the foliation are the horizontal lines, it comes 
        \[h(\alpha)^{S,T}=\Psi'^{-1} \circ \Psi.\]
        This means that under these coordinates, the holonomy between two points in the same plaque of the same foliation chart is equal to the identity.
        \item The holonomy of a path in a leaf with respect to the natural transversals given by the foliation charts can be expressed in local coordinates as a composition of transition functions $h_{ij}$ (see Definition/Proposition \ref{defprop:2.1}). More precisely, let $\alpha$ a path in $L$ from $x$ to $y$ and a sequence of foliation charts $U_1, \ldots, U_k$ such that for all $i$, $\alpha([\frac{i-1}{k}, \frac{i}{k}]) \subset U_i$ and $U_1$ and $U_k$ are centered at $x$ and $y$ respectively.
        Let $T=\psi_1^{-1}(\{0\} \times U_1^2)$ and $S=\psi_k^{-1}(\{0\} \times U_k^2)$.
        Then there exists a small open subset $V$ of $T$ such that 
        \[h(\alpha)^{S,T}= \restr{\psi_k^{-1} \circ h_{k,k-1} \circ \cdots \circ h_{2,1} \circ \psi_1}{V}.
        \]
    \end{itemize}
    \end{remark}
    
\begin{definition}
    Given a $C^r$ foliation $\mathcal{F}$ on a manifold $M$, the set of all holonomy maps of any path with respect to any transversals to $\mathcal{F}$ defines a pseudo-group called the \emph{holonomy pseudo-group} of the foliation $\mathcal{F}$.\\
    For every leaf $L$ of $\mathcal{F}$, $x\in L$ and small transversal section $T$ at $x$, there is a group homomorphism 
    \[h^T:=h^{T,T}: \pi_1(L,x) \to \text{Diff}^r_x(T), \]
    where $\text{Diff}^r_x(T)$ is the group of $C^r$ diffeomorphisms of $T$ which fix $x$, 
    called the \emph{holonomy homomorphism} of $L$, determined up to conjugation.\\
    The set $h^T(\pi_1(L,x))$ is a group called the \textit{holonomy group} of the leaf $L$, determined also up to conjugation.\\
    We define also the space $\widetilde{L}\diagup \ker(h^T)$, where $\widetilde{L}$ is the universal covering space of $L$. It does not depend on $x \in L$ nor $T$ and is called the \emph{holonomy covering} of $L$.
    \end{definition}

    Even though the holonomy of a $C^0$ foliation $\mathcal{F}$ is only continuous by definition, it can have stronger regularity, such as being $C^r$ or even holomorphic. In the latter case, $\mathcal{F}$ is said to be \emph{transversely holomorphic}.
\subsection{Almost complex structure}
See \cite{kobayashi_foundations_1996} for more details.
\begin{definition}
    Let $M$ a smooth manifold of even dimension $2n\geq2$.\\
    An \emph{almost complex structure} on $M$ is a smooth vector bundle endomorphism $J:TM \to TM$ such that $J^2=-\operatorname{Id}$.
\end{definition}
Every complex manifold, of complex dimension $n$, carries a \emph{canonical} almost complex structure given in local coordinates $(z_1,\cdots, z_n)$, where $z_k=x_k+iy_k$ for $k\in \llbracket 1,n \rrbracket$, by
$$J\left ( \frac{\partial }{\partial x_k}\right) = \frac{\partial }{\partial y_k} \quad\text{ ; } \quad J\left ( \frac{\partial }{\partial y_k}\right) =  -\frac{\partial }{\partial x_k}.$$

\begin{definition}
    An almost complex structure $J$ on a smooth compact manifold $M$ is \emph{integrable} if there exists a complex manifold structure on $M$ compatible with its smooth structure whose canonical induced almost complex structure is $J$.
\end{definition}
Not every almost complex structure is integrable. This can be seen using the following fundamental theorem due to Newlander-Nirenberg:

\begin{theorem}
    An almost complex structure $J$ on a smooth manifold $M$ is integrable if and only if its \emph{Nijenhuis tensor} 
    $$N_J: (X,Y)\mapsto  [JX, JY] - J[X, JY] - J[J X,Y] -[X, Y]$$ 
    vanishes identically.
\end{theorem}

\subsection{Transversely holomorphic foliations}
We all that in mind, we are able to define (and characterize) \emph{transversely holomorphic foliations} (see \cite{gomez-mont_transversal_1980}, \cite{haefliger_deformations_1983}, \cite{duchamp_deformation_1979}, \cite{brunella_transversely_1996}):
\begin{definition}
\label{def:2.4}
    A $C^r$ foliation $\mathcal{F}$ ($r \in \mathbb N \cup \{\infty\}$) on a smooth manifold is \emph{transversely holomorphic} if there exists a $C^r$ atlas of foliated charts compatible with the one defining $\mathcal{F}$ whose transition maps $h_{ij}$ are holomorphic.\\
    If $r\geq 1$, such a structure can be given by a maximal $C^r$ atlas $(U_i,s_i)_i$ of submersions compatible with the one defining $\mathcal{F}$ and whose transition maps $\gamma_{ij}$ are holomorphic.
\end{definition}
\begin{remark}
If $\mathcal{F}$ is a transversely holomorphic $C^r$ foliation, then its normal bundle $\nu=TM/T\mathcal{F}$ is a $C^r$ bundle.
        Also, its restriction to a transversal $T$ to the foliation $\mathcal{F}$ is $C^\infty$ (see Remark \ref{rem:2.1} and the proof of the next result).
\end{remark}
The following is a characterization of transversely holomorphic foliations:
\begin{proposition}
\label{prop:caractrholo}
    Let $M$ a smooth manifold of dimension $k+2q$ $(k \in \mathbb N, q \in \mathbb N^*)$ and $\mathcal{F}$ a $C^r$ foliation ($r \in \mathbb N \cup \{\infty\}$) of dimension $k$ on $M$.\\
    The following assertions are equivalent:
    \begin{enumerate}[label=(\roman*)]
        \item $\mathcal{F}$ is transversely holomorphic;
        \item For any points $x, y$ on the same leaf of $\mathcal{F}$, there exist (hence for all) small transversals to the foliation $T$ and $T'$ at $x$ and $y$ respectively, there exist $C^r$ diffeomorphisms $\Psi: T\to V$ and $\Psi':T'\to V'$, where $V, V'$ are open sets of $\mathbb C^q$, such that, for any $\mathcal{F}$-path $\alpha$ joining $x$ and $y$, 
        \[\Psi' \circ h(\alpha)^{T',T} \circ \Psi^{-1}\] is holomorphic.
    \end{enumerate}
    For $r\geq1$, these assertions are also equivalent to:
    \begin{enumerate}[label=(\roman*)]
        \setcounter{enumi}{2}
        \item There exists an almost complex structure on the normal bundle $ \nu=TM \diagup T\mathcal{F}$ of the foliation (that is a $C^r$ vector bundle endomorphism
        $I$ of $\nu$ satisfying $ I^2=-\operatorname{Id}$) such that
        \begin{enumerate}
            \item For any transversal to the foliation $S$, the almost complex structure $I_S$ induced by $I$ on the tangent bundle $TS \cong \restr{\nu}{S}$ of $S$ is integrable;
            \item $I$ is invariant along $\mathcal{F}$-holonomy maps, that is:\\
            For every $x \in M$, $y \in \mathcal{F}_x$, for every path $\alpha$ in $\mathcal{F}_x$ connecting $x$ to $y$, for every small transversals $T$ and $S$ to $\mathcal{F}$ at $x$ and $y$ respectively,
            $$(h_\alpha^{S,T})^*I_S:= (Dh_\alpha^{S,T})^{-1}\circ I_S\circ Dh_\alpha^{S,T}= I_T .$$
        \end{enumerate}
    \end{enumerate}
\end{proposition}
\begin{proof}
    Suppose $\mathcal{F}$ admits a transversely holomorphic structure. 
    With the notations of the second point of Remark \ref{rem:2.2}, since each $h_{ij}$ is holomorphic, the result is immediate.\\
    Now suppose for any $x,y$ on the same leaf of $\mathcal{F}$, there exists $T,T'$ transversals to the foliation at $x$ and $y$ and $C^r$ diffeomorphisms $\Psi: T\to V$ and $\Psi': T' \to V'$, where $V, V'$ are open sets of $\mathbb C^q$, such that for any $\mathcal{F}$-path $\alpha$ joining $x$ and $y$, $\Psi' \circ h(\alpha)^{T',T} \circ \Psi^{-1}$ is holomorphic.
    Let $S,S'$ any small transversals to the foliation at $x$ and $y$ respectively and denote by 
    \[\Phi = \Psi \circ h(\overline{x})^{T,S}, \quad \Phi' = \Psi' \circ h(\overline{y})^{T',S'}. \]
    It comes that $\Phi' \circ h(\alpha)^{S',S} \circ \Phi^{-1}=\Psi' \circ h(\alpha)^{T',T} \circ \Psi^{-1}$ is holomorphic.\\
    Now suppose the result true for all points and all transversals to the foliation at these points. There exists a covering family $(T_\beta)_\beta$ of small transversals to the foliation (that is, any leaf of the foliation meets at least one transversal of this family) and $C^r$ diffeomorphisms $\Psi_\beta: T_\beta \to V_\beta \subset \mathbb C^q$ such that for any $\mathcal{F}$-path $\alpha$ joining $x \in T_\beta
    $ and $y \in T_{\beta'}$, \[\Psi_{\beta'} \circ h(\alpha)^{T_{\beta'},T_\beta} \circ \Psi_\beta^{-1}\] is holomorphic.
    Let $(U, \psi)$ a $C^r$ foliated chart centered at some $x\in M$, $S$ an element of the covering family which meets $\mathcal{F}_x$ at a point $y$, and $\Psi: S\to V \subset \mathbb C^q$ such a $C^r$ diffeomorphism. Let also $\alpha$ a $\mathcal{F}$-path joining $x$ and $y$. Write $\psi(U)=U^1\times U^2$ and call $T=\psi^{-1}(\{0\} \times U^2)$.
    By restricting $U$ if necessary, the family $(U, \psi^*)$, where
    \[\psi^*:=(\operatorname{Id} \otimes (\Psi \circ h(\alpha)^{S,T} 
    \circ \restr{\psi^{-1}}{U^2})) \circ \psi,\]
    defines a $C^r$ atlas for $\mathcal{F}$ whose transition maps are holomorphic.\\ 
    Suppose $\mathcal{F}$ is a transversely holomorphic foliation and denote by $(U_i, \psi_i)_i$ a foliated atlas whose transition maps $h_{ij}$ are holomorphic. 
    Note 
    \[\psi_i = (t_i^1, \cdots , t_i^k, (x_i^{1}, x_i^{'1}), \cdots , (x_i^{q}, x_i^{'q})).\]
    Recall that the normal bundle is trivialized in $U_i$ by: 
    \[\eta_i:\left \{\begin{array}{l}
        (\pi^\nu)^{-1}(U_i) \quad \to \quad U_i \times \mathbb R^{2q} \\
        v= \sum_{j=1}^q w_j \left[\frac{\partial}{\partial x_i^j} \right] + w'_j \left [\frac{\partial}{\partial x_i^{'j}} \right] \; \mapsto \; (\pi^\nu(v),  \begin{pmatrix}
            w\\
            w'
        \end{pmatrix} )
    \end{array} \right .  .\]
    Let $J:= \begin{pmatrix}
        0 & -I_q \\
        I_q & 0
    \end{pmatrix}$ and define a $C^r$ endomorphism of $\nu|_{U_i}$ by
    \[J_i\left(\eta_i^{-1}(x,  \begin{pmatrix} w\\w'\end{pmatrix}) \right)=
        \eta_i^{-1}(x,  J \begin{pmatrix} w\\w'\end{pmatrix} )
     \]
     for $x\in U_i$ and $\begin{pmatrix} w\\w'\end{pmatrix} \in \mathbb R^{2q}$.
    For $x \in U_i \cap U_j$ and $v \in \mathbb R^{2q}$, $\eta_i\circ \eta_j^{-1}(x,v)=(x,d_{\psi_j(x)}h_{ij}(v))$ (see Remark \ref{rem:2.1}). 
    Since the transition maps $h_{ij}$ are holomorphic, their differential are in $\mathrm{Gl_q}(\mathbb C)$ which is exactly the set of invertible matrices of $\mathrm{M_{2q}}(\mathbb R)$ that commute with $J$. 
    Therefore, for every $v \in  (\pi^\nu)^{-1}(U_i\cap U_j)$, $J_i(v)=J_j(v)$ and we define in this way a $C^r$ endomorphism $I$ of $\nu\to M$. It is immediate that $I^2=-\operatorname{Id}$.\\
    For every transversal $S$ to the foliation, the natural isomorphism $TS \cong \restr{\nu}{S}$ gives rise to an almost complex structure $I_S$ on $S$. 
    Let $x \in M$, $y \in \mathcal{F}_x$, a path $\alpha$ in $\mathcal{F}_x$ connecting $x$ to $y$, $T$ and $S$ small transversals to $\mathcal{F}$ at $x$ and $y$ respectively.
    We want to prove that $(h_\alpha^{S,T})^*I_S =I_T$.
    By definition of the holonomy, we can assume that there exist foliation charts $(U_i, \psi_i)$ and $(U_j, \psi_j)$ centered at $x$ and $y$ respectively such that $\psi_i^{-1}(U_i^1\times \{0\}) \cap\psi_j^{-1}(U_j^1\times \{0\}) \neq \emptyset$.
    Also, by property $(iii)$ of holonomy maps and the first point of Remark \ref{rem:2.2}, we can assume that $T=\psi_i^{-1}(\{0\} \times U_i^2)$ and $S=\psi_j^{-1}(\{0\} \times U_j^2)$. 
    By the second point of Remark \ref{rem:2.2} and the very definition of $I$, we get the result that $I$ is invariant under holonomy maps.\\
    Recall that an almost complex structure $J$ is integrable if and only if its Nijenhuis tensor $N_J(X,Y):= [JX, JY] - J[X, JY] - J[J X,Y] -[X, Y]$ vanishes, and that if $h$ is a diffeomorphism, then $N_{h^*J}(X,Y)=h^*(N_J(h_*(X),h_*(Y))).$  With that in mind, fix a transversal $S$ to the foliation. In order to prove that $I_S$ is integrable, we can assume $S$ as small as we want (since the vanishing of $N_J$ is purely local) and even $S=\psi_i^{-1}(\{0 \} \times U_i^2)$ for some $i$ (since for any diffeomorphism $h$, $(N_{h^*J}=0 \iff N_{J}=0)$ and because if $T$ denotes another small transversal to the foliation and $h$ the holonomy of a path with respect to $S$ and $T$ respectively then $I_S=h^*I_T$ by the previous paragraph). But $I_S$ is integrable by the very definition of $I$.\\
    Conversely, assume there exists an almost complex structure on the normal bundle which is invariant under holonomy and such that the almost complex structure induced on any transversal is integrable. Let $x, y$ two points on the same leaf of $\mathcal{F}$ and fix two transversals $T, S$ of the foliation at $x$ and $y$ respectively. Since $I_T$ and $I_S$ are integrable, we can restrict $T$ and $S$ to smaller sets such that there exist diffeomorphisms $\Psi: T \to V$ and $\Psi' \to V'$, where $V,V'$ are open sets of $\mathbb C$, such that the push-forward of $I_T$ and $I_S$ by $\Psi$ and $\Psi'$ respectively is exactly the multiplication by $i \in \mathbb C$. Now since $I$ is invariant under holonomy, for any $\mathcal{F}$-path $\alpha$ joining $x$ and $y$, $(h(\alpha)^{S,T})^*I_S=I_T$. Therefore, the pull-back of $i \in \text{End}(\mathbb C)$ by $\Psi' \circ h(\alpha)^{S,T} \circ \Psi^{-1}$ is $i$, which exactly means that $\Psi' \circ h(\alpha)^{S,T} \circ \Psi^{-1}$ is holomorphic.
\end{proof}

\subsection{Hyperbolic set, Anosov flows}
We recall some definitions of hyperbolic dynamical systems (see \cite{katok_introduction_1995}, \cite{plante_anosov_1972}, \cite{hirsch_invariant_1977}).
\begin{definition}
\label{def:2.6}
    Let $(\varphi_t)_{t \in \mathbb R}$ a smooth flow on a smooth manifold $M$ generated by a smooth nowhere vanishing vector field $X$.\\
   A $(\varphi_t)$-invariant subset $\Lambda$ of $M$ is called \emph{hyperbolic} if there exist $(d\varphi_t)$-invariant subbundles of $TM$ - $E^{s}$ and $E^{u}$ - a Riemannian metric $\|.\|$ on $M$, real numbers $\lambda < 1 < \mu$ and $C>0$ such that:
    \begin{enumerate}
        \item $TM = \mathbb R X \oplus E^{s} \oplus E^{u}$ ;
        \item 
        \label{eq:ano}
            \begin{align*}
            \forall v_s \in E^{s}, t \geq 0, \quad \|d\varphi_t(v_s)\| &\leq C\lambda^t\|v_s\|,\\
            \forall v_u \in E^{u}, t \geq 0, \quad \|d\varphi_{-t}(v_u)\| &\leq C\mu^{-t}\|v_u\|.            \end{align*}
    \end{enumerate}
    The vector bundles $E^{s}$ and $E^{u}$ are called respectively the \emph{(strong) stable \emph{and} (strong) unstable distributions} of $(\varphi_t)$.\\
    The vector bundles $\mathbb R X \oplus E^{s}$ and $\mathbb R X \oplus E^{u}$ are called respectively the \emph{weak stable \emph{and} weak unstable distributions} of $(\varphi_t)$.\\
    If $M$ is a hyperbolic set for $(\varphi_t)$, then the flow $(\varphi_t)$ is said to be \emph{Anosov}.
\end{definition}
If a compact set is hyperbolic for a given Riemannian metric, then it is hyperbolic for every Riemannian metric. Moreover, by \cite{hirsch_stable_1969}, we can choose a Riemannian metric so that $C=1$, which is what we do from now on. We call such metric an \textit{adapted} metric.\\

As was discussed in \cite{wilkinson_dynamical_2008}, one can define different notions of integrability of a continuous subbundle of $TM$.
We recall some of them.
Let $E$ a continuous subbundle of dimension $k$ of the tangent bundle of a smooth manifold $M$.\\
$E$ is said to be \textit{integrable} if there exists a foliation $\mathcal{F}$ of $M$ by $C^1$ leaves of dimension $k$ which are everywhere tangent to $E$.
$E$ is said to be \textit{uniquely integrable} if it is integrable to a foliation $\mathcal{F}$ and every $C^1$ path of $M$ everywhere tangent to $E$ lies in a single leaf of $\mathcal{F}$.
The notion of unique integrability is different from the fact that there exists a unique foliation tangent to $E$.
If $E$ is uniquely integrable, then there exists a unique foliation whose leaves are everywhere tangent to $E$.\\

Let $(\varphi_t)_{t \in \mathbb R}$ an Anosov flow on a smooth manifold $M$.
Then $E^{s}$ and $E^{u}$ are unique and are characterized by the inequalities \ref{eq:ano}. Also, the unstable and stable distributions are integrable to flow-invariant $C^0$ (even Hölder-continuous) foliations $\mathcal{F}^u$, $\mathcal{F}^s$ respectively (i.e. for $*\in \{s,u\}$ and $t \in \mathbb R$, $\varphi^t(\mathcal{F}^*(x))=\mathcal{F}^*(\varphi^t(x))$), with $C^\infty$ leaves.
The stable and unstable leaves are simply connected.\\
Denote by $d$ the distance function on $M$ induced by our Riemannian metric. For $\delta>0$  and $S$ an immersed submanifold of $M$ containing a point $x\in M$, let $S_{x,\delta}$ the open ball of $S$ (for the metric induced on $S$) centered at $x$ of radius $\delta$. Denote by an index the induced metric on an immersed submanifold.\\
Then $x$ and $y$ belong to the same stable manifold if and only if $d(\varphi^t(x), \varphi^t(y))\xrightarrow[{t\to +\infty}]{} 0$.
Also, if $x$ and $y$ belong to the same stable manifold, then $d_{s}(\varphi^t(x), \varphi^t(y)) \leq \lambda^t d_{s}(x,y)$ for every $t\geq 0$.\\
The same is true for the unstable foliation with negative instead of positive time.\\

We give two fundamental examples of Anosov flows on smooth compact manifolds:
\begin{itemize}
\label{ex:2.7.1}
    \item \textit{Suspension of an Anosov diffeomorphism}: Let $f$ an Anosov diffeomorphism on a smooth compact manifold $F$ (the only real difference with the continuous time case is the absence of the flow direction). Let $E^{u}$ and $E^{s}$ the unstable and stable distributions of $f$ and $\lambda<1<\mu$ the rates of contraction and expansion of $f$ respectively. \\
    The tangent bundle $T(F_f)$ of the suspension manifold $F_f$ is isomorphic to the vector bundle $\mathbb R \frac{\partial}{\partial t} \oplus TF$ over $F_f$, which decomposes therefore as $\mathbb R \frac{\partial}{\partial t} \oplus E^{u} \oplus E^{s}$.\\
    Let $\|.\|$ a Riemannian metric on $F$ and extend it to a Riemannian metric on $F_f$ such that $\mathbb R \frac{\partial}{\partial t} $ and $TF$ are orthogonal with respect to this metric. It comes by calling $(\varphi^t)$ the suspension flow of $f$:
    \begin{align*}
            \forall v_s \in E^{s}, t \geq 0, \quad &\|d\varphi_t(v_s)\| = \|df^{\floor{t}}(v_s)\|\leq C\lambda^{\floor{t}}\|v_s\| \leq \frac{C}{\lambda}\lambda^t\|v_s\|, \\
            \forall v_u \in E^{u}, t \geq 0, \quad &\|d\varphi_{-t}(v_u)\| = \|df^{-\floor{t}}(v_u)\|\leq C\mu^{-\floor{t}}\|v_u\| \leq C\mu\mu^{-t}\|v_u\|.
    \end{align*}
            So the suspension flow of an Anosov diffeomorphism $f$ is an Anosov flow on the (compact) suspension manifold $F_f$, where the constant $C$ of Definition \ref{def:2.6} is now $\max(\frac{C}{\lambda}, C\mu)$ and the rates of contraction and expansion are the same than those of $f$. Remark also that the suspension flow of a diffeomorphism $f$ is Anosov if and only if $f$ is Anosov.
    \item \textit{Geodesic flow on a compact Riemannian manifold of negative curvature}: (See \cite{paternain_geodesic_1999}, \cite{katok_introduction_1995}, \cite{barreira_nonuniform_2007}, \cite{lee_riemannian_1997}, \cite{camacho_geometric_1985}) Let $(M,g)$ a complete Riemannian manifold of constant negative curvature.\\
    For any $v=(p,v_p) \in TM$, there exists an isomorphism
    \[ \psi_v: \left \{\begin{array}{l}
        T_v(TM) \quad \to \quad T_pM \times T_pM \\
        \left[(\gamma,V)\right] \; \mapsto \; (\dot{\gamma}(0), \nabla_{\dot{\gamma}(0)}V)
    \end{array} \right . \]
    and a unique Riemannian metric on $TTM$, called the Sasaki metric of $TM$, such that under these identifications:
    \[\forall(u,v) \in TM \times TM, \quad \|(u,v)\|^2=\|u\|^2+\|v\|^2.\]
    Moreover, each isomorphism $\psi_v$, for $v=(p,v_p)\in S^1M$, restricts to an isomorphism between the tangent bundle at $v$ of the unit sphere bundle, $T_v(S^1M)$, and $T_pM \times \{v_p\}^{\perp}$.\\
    Let $v=(p,v_p) \in TM$ and $\xi \in T_v(TM)$. Note $\psi_v(\xi)=(v_1^{\xi},v_2^{\xi}) \in T_pM \times T_pM$ and denote by $J^{\xi}$ the unique Jacobi field along the geodesic $\gamma_v$ such that $J^{\xi}(0)=v_1^{\xi}$ and $\nabla_{\dot{\gamma}(0)}J^{\xi}=v_2^{\xi}$.\\
    Then, the action of the geodesic flow $(g^t)$ on $T(TM)$ satisfies:
    \[\forall t\in \mathbb R, \quad \psi_{g^t(v)}(dg^t(\xi))=(J^{\xi}(t), \nabla_{\dot{\gamma}(t)}J^{\xi}).\]
    Suppose now $M$ is a compact Riemannian manifold (so it is complete) and consider the action of the geodesic flow on the unit sphere bundle $S^1M$, which is also compact. We can show that for any $v=(p,v_p) \in S^1M$, $T_v(S^1M) \cong T_pM \times \{v_p\}^{\perp}$ decomposes as the direct sum of the three following subspaces of $T_v(S^1M)$:
    \begin{itemize}
        \item $\mathbb R v_p \times \{0\}$, which corresponds to the geodesic flow direction at $v$;
        \item $\left \{(w_p, w_p) \in T_pM \times \{v_p\}^{\perp}\right \}$, which corresponds to the   unstable distribution of $(g^t)$ at $v$ ;
        \item $\left \{(w_p, -w_p) \in T_pM \times \{v_p\}^{\perp}\right \}$, which corresponds to the   stable distribution of $(g^t)$ at $v$,
    \end{itemize}
    so that the geodesic flow on $S^1 M$ is an Anosov flow (see Definition \ref{def:2.6}) with $C=1$, $\lambda=e^{-1}$ and $\mu=e$ when choosing the Sasaki metric on $S^1M$. 
\end{itemize}

\subsection{Quasiconformality}
We recall some notions of quasiconformality (see \cite{ahlfors_quasiconformal_1953} or \cite{vaisala_lectures_1971}). 
Let $k\geq 2$ and $\|\cdot\|$ be the euclidean norm on $\mathbb R^k$.\\
A homeomorphism $h: U \to V$ between two open subsets $U, V$ of $\mathbb{R}^k$ is \emph{conformal} if it is $C^1$ and its differential at every point $x\in U$ is angle-preserving and orientation-preserving.
If $k=2$, then this definition coincides with that of a holomorphic function whose holomorphic derivative is nowhere vanishing.

For such maps $h:U\to V$ which are not $C^1$ we can define a more general notion of \emph{quasiconformality}.\\
We define the \emph{linear dilatation} of $h$ at $x \in U$ to be  
\[
L_h(x) = \limsup_{r \to 0} \frac{\displaystyle\max_{\|y-x\| = r} \|h(y) - h(x)\|}{\displaystyle\min_{\|y-x\| = r} \|h(y) - h(x)\|}.
\]  
If $L_h(x) \leq K$ for every $x \in U$ (where $K\geq1$), then $h$ is said to be \emph{$K$-quasiconformal}.
A homeomorphism $h:U \to V$ is \emph{quasiconformal} if there exists $K\geq1$ such that $h$ is $K$-quasiconformal.\\
We list some properties of quasiconformal maps:
\begin{itemize}
    \item If $h:U\to V$ is a diffeomorphism, then it is $K$-quasiconformal if and only if for every $x \in U$, 
$$\frac{\displaystyle\sup_{\|v\|=1} \|d_xh(v)\|}{\displaystyle\inf_{\|v\|=1} \|d_xh(v)\|}\leq K;$$
    \item If $h:U\to V$ is $K$ quasiconformal and $g:V \to W$ is $K'$ quasiconformal, then $g\circ h: U \to W$ is $KK'$ quasiconformal;
    \item A homeomorphism $h:U\to V$ is $1$-quasiconformal if and only if it is conformal;
    \item If a sequence $(h_n)$ of $K$-quasiconformal maps converges uniformly to a homeomorphism $h$, then $h$ is $K$-quasiconformal.
\end{itemize}

\section{Fundamental examples of transversely holomorphic flows on smooth five-dimensional manifolds}
\label{sec:3}
\subsection{Suspension of a holomorphic diffeomorphism of a complex surface}
Let $f: F\to F$ a holomorphic diffeomorphism of a complex surface $F$. One way to see the suspension flow is transversely holomorphic is by computing the holonomy of any path between two points in the same orbit of the suspension flow $\overline{(s,x)}$ and $\overline{(s+t_0,x)}$ with respect to the transversals $F_s=\pi(\{s\} \times F)$ and $F_{s+t_0}=\pi(\{s+t_0\} \times F)$: in holomorphic local coordinates, it is a power of $f$ ($f^{\floor{t}}$ if $t\geq 0$ and $f^{\ceil{t}}$ if $t<0$).\\
This construction can be generalized to any complex dimension for $F$.
\subsection{Geodesic flow on a hyperbolic three-dimensional manifold}
(See \cite{marden_outer_2007}). Consider the upper half space model $\mathbb H^3:= \{(z,t)\in \mathbb R^3, \; z\in \mathbb C, \, t>0\}$ with the metric 
    $ds = \frac{|dz|^2 + (dt)^2}{t}.$
    We will always say $\mathbb C$ instead of $\mathbb C \times \{0\}$ from now on. 
    Oriented non-parametrized complete geodesics of $(\mathbb H^3, ds)$ (whose set is denoted $\overrightarrow{\mathbb G}$) are exactly the perpendicular lines to $\mathbb C$ and the semi circles perpendicular to $\mathbb C$.
    Therefore, $(\mathbb H^3, ds)$ is \textit{the} complete simply-connected Riemannian three-manifold of constant sectional curvature $-1$.
    Denote by $d$ the distance induced by this metric on $\mathbb H^3$ and fix some $x_0 \in \mathbb H^3$: if an element $x=(z,t) \in \mathbb H^3$ satisfies $x \rightarrow \infty$ or $t \rightarrow 0$, then $d(x,x_0) \rightarrow \infty$, and vice versa. Therefore, we call the \emph{boundary at infinity} of $\mathbb H^3$ the set $\partial \mathbb H^3:= \mathbb C \cup \{\infty\}$ which is non other than the complex projective line, or the Riemann sphere, $\mathbf P^1(\mathbb C)$.
    Remark that $\overrightarrow{\mathbb G}$ has a natural and useful identification by a complex manifold: an element of $\overrightarrow{\mathbb G}$ is naturally identified by its (ordered) endpoints, each of which lies in $\partial \mathbb H^3 = \mathbf P^1(\mathbb C)$ and is different from the other. Therefore, $\overrightarrow{\mathbb G}$ is non other than $\mathbf P^1(\mathbb C) \times \mathbf P^1(\mathbb C) \setminus \Delta$, where $\Delta=\{(x,x), \; x \in \mathbf P^1(\mathbb C)\}$ is the diagonal of $\mathbf P^1(\mathbb C) \times \mathbf P^1(\mathbb C)$. The group of automorphisms of $P^1(\mathbb C)$  is exactly the set of \emph{Möbius transformations} 
    $\left \{ z \to \frac{az+b}{cz+d}, \; a,b,c,d \in \mathbb C, \; ad-bc=1 \right \}$
    which is $\mathrm{PSL}(2,\mathbb C)=\mathrm{SL}(2,\mathbb C) \diagup \{I_2, -I_2\}$.
    In fact, any orientation-preserving isometry of $(\mathbb H^3, ds)$ is the extension of a unique Möbius transformation, that is 
    $\text{Isom}_+(\mathbb H^3,ds) \cong \mathrm{PSL}(2,\mathbb C).$\\
We can define a global submersion $D: S^1\mathbb H^3 \to (\mathbf P^1(\mathbb C) \times \mathbf P^1(\mathbb C)) \setminus \Delta $ which associates to every unit tangent vector $v$ to $\mathbb H^3$ the oriented non-parametrized complete geodesic (seen as an element of $(\mathbf P^1(\mathbb C) \times \mathbf P^1(\mathbb C)) \setminus \Delta $) it belongs to. The orbits of the geodesic flow on $S^1\mathbb H^3$ are precisely the fibers of $D$, so $D: S^1\mathbb H^3 \to (\mathbf P^1(\mathbb C) \times \mathbf P^1(\mathbb C)) \setminus \Delta $ is a trivial fiber bundle with fiber $\mathbb R$.\\
As for the geodesic flow on the unit tangent bundle of a compact hyperbolic $3$-manifold, we begin with some results concerning subgroups $G$ of $\mathrm{PSL}(2,\mathbb C)$ which we will need for our study (see \cite{song_hyperbolic_nodate} for example).\\
$G$ acts properly discontinuously on $\mathbb H^3$ if and only if it is discrete.
If $G$ is discrete, then it acts freely on $\mathbb H^3$ if and only if it is torsion-free, and it is cocompact if and only if it is of virtual cohomological dimension $3$.\\
Therefore, the complete (respectively compact) hyperbolic three-dimensional manifolds are exactly the quotient of $\mathbb H^3$ by a discrete torsion-free (respectively discrete torsion-free and cocompact) subgroup of $\mathrm{PSL}(2, \mathbb C)$. As a result, if $V=\mathbb H^3 \diagup G$ is a complete hyperbolic three-dimensional manifold, where $G$ is a discrete torsion-free subgroup of $\mathrm{PSL}(2, \mathbb C)$, the universal Riemannian covering manifold $\widetilde{V}$ of $V$ is isometric to $\mathbb H^3$. Denote by $\mathbb H^3 \xrightarrow{p} V$ this covering map, which is an isometry. It induces a covering map $S^1\mathbb H^3 \xrightarrow{Tp} S^1V$, where $Tp$ is the tangent map of $p$, which is also an isometry. Therefore, the geodesic flow on $S^1H^3$ is the pull back by $Tp$ of the geodesic flow on $S^1V$ so that the transition functions of the foliated charts for each of these two geodesic flows are the same, which implies that the geodesic flow on $S^1V$ is also transversely holomorphic. 

\section{Sketch of proof of the results}
\label{sec:4}
We will always consider a smooth transversely holomorphic Anosov flow $(\varphi^t)_{t \in \mathbb R}$ on a smooth compact manifold $M$.\\
In section \ref{sec:5}, we first prove that each strong stable and strong unstable leaf is a complex manifold on which the flow acts holomorphically. 
We begin in Proposition \ref{prop:holoinvfol} with an elementary but fundamental result on the holonomy of $\Phi$ with respect to transversals containing small open neighborhoods of a $\varphi^t$-invariant foliation whose tangent distribution does not contain the flow direction.\\
We apply this result to the strong stable and strong unstable foliations (Corollary \ref{cor:holofloFsu}) and use the characterization of $C^\infty$ transversely holomorphic foliations given by Proposition \ref{prop:caractrholo} as well as the characterization of the strong stable and unstable distributions to prove the existence of an integrable almost complex structure on them, which is invariant under the differential of the flow. This is Proposition \ref{prop:complexFsuFss}.
An important corollary which will be fundamental is that the strong stable and unstable foliations are continuous foliations with complex leaves.\\
In Theorem \ref{thm:Fstrholo}, we prove that the weak stable foliation is transversely holomorphic if the strong unstable distribution is of complex dimension one.
The idea is that we can take two points belonging to the same strong stable leaf in order to compute the weak stable holonomy, since we have shown previously that this holonomy is just the action of the flow on strong unstable leaves, which is holomorphic, if the two points belong to the same orbit.
We then follow the idea in \cite{ghys_holomorphic_1995} to prove that it is conformal, thus holomorphic in that case.\\
In section \ref{sec:6}, we prove that the strong unstable leaves are complex affine diffeomorphic to $\mathbb C$ and the flow is affine with respect to these structures (Theorem \ref{thm:compaffstru}), following again the idea in \cite{ghys_holomorphic_1995}.
We use the strong unstable foliation's continuous atlas of complex leaves and the uniform expansion of the flow on the strong unstable distribution in order to apply the contracting mapping principle on a space of continuous section of a jet bundle.
The link between complex affine structures on a holomorphic curve $L$ and jets of holomorphic diffeomorphisms on $L$ is given in Lemma \ref{lem:affstruL}.\\
This allows us to compute the Schwarzian derivative of the weak stable holonomy between two points belonging to the same strong stable leaf, since this holonomy is affine thus projective if the points belong to the same orbit by Theorem \ref{thm:compaffstru}.
We need for that the notion of Markov partition for a flow, as defined in \cite{ratner_markov_1973} (see Theorem \ref{thm:rat} and the discussion before and after it).
By following the idea of \cite{ghys_holomorphic_1995}, we prove that the Schwarzian derivative is equal to zero which implies Theorem \ref{thm:FsTrProj}.\\
In the last section, in case $M$ is of dimension five, we use the preceding result to define the developing maps as well as the holonomy representations (see \cite{thurston_geometry_2022} and section \ref{sec:6} for reminders) associated to the transverse projective structures of the weak stable and weak stable foliations which gives us a developing map $D$ and a holonomy representation $H$ for the transverse $(\mathrm{PSL}(2, \mathbb C) \times \mathrm{PSL}(2, \mathbb C), \mathbf{P}^1(\mathbb C) \times \mathbf{P}^1(\mathbb C))$-structure of the flow.
We prove that $D: \widetilde{M} \to D(\widetilde{M})$ is a smooth fiber bundle with fiber $\mathbb R$ and that the image of $D$ is of two types, which correspond to the suspension and the geodesic flow cases.
\section{Complex structures of strong leaves and transverse structures of weak foliations}
\label{sec:5}
\subsection{The strong leaves are complex}

We state some general results concerning foliations transverse to flows, which are essential for the rest of the article and are proved in \cite{}.

\begin{lemma}
    Let $(\varphi^t)_{t\in \mathbb R}$ a complete smooth flow on a smooth manifold M, induced by a nowhere vanishing vector field $X$.
    Let $\mathcal{F}$ a continuous foliation of $M$ with $C^r$ leaves $(r \geq 1)$, which is invariant under the flow $\varphi^t$ (i.e. for every $t \in \mathbb R$ and $x \in M$, $\varphi^t(\mathcal{F}_x)=\mathcal{F}_{\varphi^t(x)}$), such that $X$ is nowhere tangent to $\mathcal{F}$. \\
    Then there exists a continuous foliation $\mathcal{F}^w$ with $C^r$ leaves tangent to $T\mathcal{F} \oplus \mathbb R X$ whose leaf at $x \in M$ is $\mathcal{F}^w(x)=\bigcup_{t \in \mathbb R}\varphi^t(\mathcal{F}_x)$.
\end{lemma}
This last results allows us to define the \emph{weak stable} foliation $\mathcal{F}^{ws}$ and \emph{weak unstable} foliation $\mathcal{F}^{wu}$. They are $C^0$ foliations with $C^\infty$ leaves.
Since the weak stable foliation and the unstable foliation are transverse, we have the following result (see also
\cite{pugh_o-stability_1970}):
    \begin{proposition}
    \label{prop:locprodstru}
        If $M$ is compact, there exists $\delta_0 >0$ such that for every $\delta\leq \delta_0$, for every $x_0 \in M$,
     \[\left \{\begin{array}{ccc}
        \mathcal{F}^{ws}_{\delta}(x_0) \times \mathcal{F}^{u}_{\delta}(x_0) & \to & M \\
        (y,z) & \mapsto & \mathcal{F}^{u}_{2\delta}(y) \cap \mathcal{F}^{ws}_{2\delta}(z)
    \end{array} \right. \]
    is a well-defined map which is in addition a homeomorphism onto its image. \\
    The same is true by exchanging the roles of the stable and unstable letters.
    \end{proposition}
In fact, this result allows us to define a $C^0$ atlas with $C^\infty$ leaves adapted to $\mathcal{F}^u$ and $\mathcal{F}^{ws}$ simultaneously.
\begin{lemma}
\label{lem:equivmetricweak}
    Let $(\varphi^t)$ a smooth flow on a smooth compact manifold M, induced by a nowhere vanishing vector field $X$.
    Let $\mathcal{F}$ a continuous foliation of $M$ with $C^r$ leaves $(r \geq 1)$, which is invariant under the flow $\varphi^t$ and such that $X$ is nowhere tangent to $\mathcal{F}$. \\
    Then there exist positive constants $C_1<1$, $C_2>1$ and $\delta_0$ such that for every $\delta<\delta_0$ and $x \in M$, 
    $$\bigcup_{z \in  \Phi_{C_1\delta}(x)}\mathcal{F}_{C_1\delta}(z)\subset \mathcal{F}^{w}_{\delta}(x)\subset \bigcup_{z \in  \Phi_{C_2\delta}(x)}\mathcal{F}_{C_2\delta}(z)$$
    and
    $$\bigcup_{z \in \mathcal{F}_{C_1\delta}(x) }\Phi_{C_1\delta}(z)\subset \mathcal{F}^{w}_{\delta}(x)\subset \bigcup_{z \in \mathcal{F}_{C_2\delta}(x) }\Phi_{C_2\delta}(z).$$
\end{lemma}
\begin{lemma}
    Let $(\varphi^t)$ a smooth flow on a smooth compact manifold M, induced by a nowhere vanishing vector field $X$.\\
    Then there exist positive constants $C_1<1$, $C_2>1$ and $\delta_0$ such that for every $\delta<\delta_0$ and $x \in M$, 
    $$\bigcup_{t \in  \left]-C_1\delta, C_1\delta \right[}\varphi^t(x)\subset \Phi_\delta(x)\subset \bigcup_{t \in  \left]-C_2\delta, C_2\delta \right[}\varphi^t(x).$$
\end{lemma}
We now state a general fact concerning the flow-holonomy with respect to transversals containing local leaves of a flow-invariant foliation transverse to the flow.

\begin{proposition}
\label{prop:holoinvfol}
    Let $(\varphi^t)_{t\in \mathbb R}$ a complete smooth flow on a smooth manifold M, induced by a nowhere vanishing vector field $X$.
    Let $\mathcal{F}$ a continuous foliation of $M$ with $C^1$ leaves, which is invariant under the flow $\varphi^t$ (i.e. for every $t \in \mathbb R$ and $x \in M$, $\varphi^t(\mathcal{F}_x)=\mathcal{F}_{\varphi^t(x)}$), such that $X$ is nowhere tangent to $\mathcal{F}$.\\
    For $x \in M$ and $t\in \mathbb{R}$, let $W_{\varphi^t(x)}$ be a small neighborhood of $\varphi^t(x)$ in $\mathcal{F}_{\varphi^t(x)}$ and by $T_{\varphi^t(x)}$ a small transversal to the flow in $\varphi^t(x)$ containing $W_{\varphi^t(x)}$. \\
    Fix $x\in M, t\in \mathbb R$ and let the path 
    $\beta^x_t: \left \{\begin{array}{ccc}
        [0,1] & \to & M \\
        s  &\mapsto  &\varphi^{st}(x)
    \end{array} \right. .$\\
    Then if $T_x$ is taken sufficiently small, the restriction to $W_x$ of the $\Phi$-holonomy of the path $\beta^x_t$ with respect to the transversals $T_{x}$ and $T_{\varphi^t(x)}$ is the restriction to $W_x$ of the map 
        \[\restr{\varphi^t}{\mathcal{F}_x}: \mathcal{F}_x \to \mathcal{F}_{\varphi^t(x)}.\]
\end{proposition}
\begin{proof}
    First assume that $x$ and $\varphi^t(x)$ belong to the same plaque for the orbit foliation $\Phi$. Take $T_x$ sufficiently small inside this foliation chart. 
    Since $\varphi^t(W_x)\subset\mathcal{F}_{\varphi^t(x)}$ and because $\varphi^t$ is continuous on $M$, it comes by definition of holonomy that for every $z \in W_x$,
    \[h(\alpha)^{T_{\varphi^t(x)}, T_x} (z) = \varphi^t(z).\]
    In general, we take intermediate points $x_i=\varphi^{t_i}(x)$ between $x$ and $\varphi^t(x)$ in $\Phi_x$, $0\leq i \leq N$, $N \in \mathbb N^*$, $t_0=0, t_N=t$, such that $\varphi^{t_{i-1}}(x)$ and $\varphi^{t_{i}}(x)$ belong to the same plaque for $\Phi$.
    If $T_{x_i}$ is sufficiently small for every $i \in \llbracket0,N-1\rrbracket$, the above shows that the restriction to $W_{x_{i-1}}$ of the holonomy of $\beta^{x_{i-1}}_{t_i-t_{i-1}}$ between $x_{i-1}$ and $x_i=\varphi^{t_i-t_{i-1}}(x_{i-1})$ with respect to $T_{x_{i-1}}$ and $T_{x_i}$ is precisely the restriction of $\varphi^{t_i-t_{i-1}}$ to $W_{x_{i-1}}$.\\
    If $T_x$ is sufficiently small and $z \in W_x$, then by definition and by the above:
    \[h(\alpha)^{T_{\varphi^t(x)}, T_x} (z)=\varphi^{t-t_{N-1}}(\varphi^{t_{N-1}-t_{N-2}}(\cdots(\varphi^{t_2-t_1}(\varphi^{t_1}(z))))) = \varphi^{t}(z).\] 
\end{proof}\begin{corollary}
\label{cor:holofloFsu}
    Let $(\varphi^t)_{t\in \mathbb R}$ a transversely holomorphic Anosov flow on a smooth compact manifold $M$.\\
    For $x \in M$ and $t\in \mathbb{R}$, denote by $W^{u}_{\varphi^t(x)}$ a small neighborhood of $\varphi^t(x)$ in $\mathcal{F}^{u}_{\varphi^t(x)}$ and by $T_{\varphi^t(x)}$ a small transversal to the flow in $\varphi^t(x)$ containing $W^{u}_{\varphi^t(x)}$. \\
    Fix $x\in M, t\in \mathbb R$ and let the path 
    $\beta^x_t: \left \{\begin{array}{ccc}
        [0,1] & \to & M \\
        s  &\mapsto  &\varphi^{st}(x)
    \end{array} \right. .$\\
    Then:
    \begin{enumerate}
        \item if $T_x$ is taken sufficiently small, the restriction to $W^{u}_x$ of the $\Phi$-holonomy of the path $\beta^x_t$ with respect to the transversals $T_{x}$ and $T_{\varphi^t(x)}$ is the restriction to $W^{u}_x$ of the map 
        \[\restr{\varphi^t}{\mathcal{F}^{u}_x}: \mathcal{F}^{u}_x \to \mathcal{F}^{u}_{\varphi^t(x)}\]
        \item the germ of the $\mathcal{F}^{ws}$-holonomy of the path $\beta^x_t$ with respect to the transversals $W^{u}_x$ and $W^{u}_{\varphi^t(x)}$ is the germ of the map
        \[\restr{\varphi^t}{\mathcal{F}^{u}_x}: \mathcal{F}^{u}_x \to \mathcal{F}^{u}_{\varphi^t(x)}.\]
    \end{enumerate}
    The same statement is true when replacing the   unstable leaves by the   stable leaves and the weak stable foliation by the weak unstable foliation. 
\end{corollary}
\begin{proof}
    The first point is immediate by the previous proposition.
    As for the second point, the $\mathcal{F}^{ws}$-holonomy of the path $\beta^x_t$ with respect to the transversals $W^{u}_x$ and $W^{u}_{\varphi^t(x)}$ in $M$ coincides with the restriction to $W^{u}_x$ of the $\Phi$-holonomy of the path $\beta^x_t$ with respect to the transversals $T^{u}_x$ and $T^{u}_{\varphi^t(x)}$. The same can be said for the weak unstable foliation if we consider stable leaves as transversals.
\end{proof}
\begin{corollary}
\label{cor:diffholo}
    Let $(\varphi^t)_{t\in \mathbb R}$ a transversely holomorphic Anosov flow on a smooth compact manifold $M$. 
    For $x \in M$ and $t\in \mathbb{R}$, denote by $W^{u}_{\varphi^t(x)}$ (respectively $W^{s}_{\varphi^t(x)}$) a small neighborhood of $\varphi^t(x)$ in $\mathcal{F}^{u}_{\varphi^t(x)}$ (respectively $\mathcal{F}^{s}_{\varphi^t(x)}$) and by $S_{\varphi^t(x)}$ a small transversal to the flow in $\varphi^t(x)$ containing $W^{u}_{\varphi^t(x)}$ and $W^{s}_{\varphi^t(x)}$.\\
    Then $d_x(h(\beta^x_t)^{S_{\varphi^t(x)}, S_x})=d_x\varphi^t$.
    \end{corollary}
    \begin{proof}
        Since $T_xS_x=E^{u}_x \oplus E^{s}_x$, the first point of the previous corollary gives the result.
    \end{proof}
Therefore, we can prove the following essential result:
\begin{proposition}
\label{prop:complexFsuFss}
    Let $(\varphi^t)$ a transversely holomorphic Anosov flow on a smooth compact manifold $M$.\\
    Then each leaf of the unstable and stable foliations is an immersed complex manifold.\\
    \noindent Moreover, for any $t\in \mathbb{R}$ and $x\in M$, the well defined maps:
    \[\restr{\varphi^t}{\mathcal{F}^{u}_x}: \mathcal{F}^{u}_x \to \mathcal{F}^{u}_{\varphi^t(x)} \quad \text{and} \quad  \restr{\varphi^t}{\mathcal{F}^{s}_x}: \mathcal{F}^{s}_x \to \mathcal{F}^{s}_{\varphi^t(x)}\]
    are holomorphic with respect to these complex structures.
\end{proposition}

\begin{proof}
    We prove the result for the unstable case, as the stable case is analogous. Let $p$ the complex codimension of $\Phi$.\\
    Fix $x \in M$. For any $t\in \mathbb R$, let $S_{\varphi^t(x)}$ a small transversal to the flow containing a small open neighborhood $W^{u}_{\varphi^t(x)}$ of $\varphi^t(x)$ in $\mathcal{F}^{u}_{\varphi^t(x)}$ and a small open neighborhood $W^{s}_{\varphi^t(x)}$ of $\varphi^t(x)$ in $\mathcal{F}^{s}_{\varphi^t(x)}$.
    Let $\Phi$ the foliation defined by the transversely holomorphic flow $(\varphi^t)_t$. The normal bundle $\nu = TM\diagup \Phi$ is naturally isomorphic (in the $C^0$ category) to $E^{u} \oplus E^{s}$.
    Therefore, the smooth almost complex structure $I$ on $\nu$ invariant by holonomy maps (see Proposition \ref{prop:caractrholo}) gives rise to a continuous almost complex structure, still denoted by $I$, on $E^{u} \oplus E^{s}$ which makes it a continuous vector subbundle of $TM$ with complex fibers.
    Moreover, since $I$ is invariant under $\Phi$-holonomy maps, then the restriction $I^{(t)}$ of $I$ on $TS_{\varphi^t(x)}$ is a smooth almost complex structure on $S_{\varphi^t(x)}$ satisfying 
    \[I^{(t)}\circ d(h(\beta^x_t)^{S_{\varphi^t(x)}, S_x}) = d(h(\beta^x_t)^{S_{\varphi^t(x)}, S_x}) \circ I^{(0)}  \quad (*)\]
    Since each $I^{(t)}$ is integrable by Proposition \ref{prop:caractrholo}, we can restrict each $S_{\varphi^t(x)}$ so that it is diffeomorphic to an open set of $\mathbb C^p$ and $I^{(t)}$ coincides with the multiplication by $i \in \mathbb C$. 
    With that in mind, thanks to Corollary \ref{cor:diffholo}, by taking $S_{x}$ even smaller, it comes for $v\in E^{u}_x$ and $t \in \mathbb R$, that 
    $id_x\varphi^t(v)=d_x\varphi^t(iv).$
    Take a continuous metric  $\|\cdot\|_{(0)}$ on $M$ which is hermitian on $E^{u} \oplus E^{s}$ with respect to $I$.
    It comes, for $v \in E^{u}_x$ and $t \in \mathbb R$:
    \begin{align*} \|d_x\varphi^t(iv)\|_{(0)}=\|id_x\varphi^t(v)\|_{(0)}=\|d_x\varphi^t(v)\|_{(0)} \leq C\lambda^t\|v\|_{(0)}= C\lambda^t\|iv\|_{(0)}.
    \end{align*}
    Now take a smooth Riemannian metric $\|\cdot\|$ on $M$.
    Since $M$ is compact and both metrics are continuous, $\|\cdot\|$ and $\|\cdot\|_{(0)}$ are equivalent on $M$.
    Therefore, by the above and the characterization of the unstable space, it comes $iv\in E^{u}_x$.\\
    To summarize, we have defined on $\mathcal{F}^{u}_x$ a continuous almost complex structure $J^{(x)}$ whose restriction to $W^{u}_y$ for every $y \in \mathcal{F}^{u}_x$ is a smooth almost complex structure, making $J^{(x)}$ a smooth almost complex structure on $\mathcal{F}^{u}_x$.\\
    Now fix $y \in \mathcal{F}^{u}_x$ and let $S_{y}$ a small transversal to the flow containing a small open neighborhood $W^{u}_{y}$ of $y$ in $\mathcal{F}^{u}_{x}$. Denote by $I^{(y)}$ the restriction of $I$ to $S_y$. We know it is integrable by Corollary \ref{cor:holofloFsu}. Therefore, its Nijenhuis tensor $N_{I^{(y)}}$ vanishes everywhere. Since the restriction of $I^{(y)}$ to $W^{u}_y$ is exactly the restriction of $J^{(x)}$ to $W^{u}_y$, it follows that $N_{J^{(x)}}(X,Y)=0$ for any vector fields $X,Y$ tangent to $W^{u}_{y}$. Since $y \in \mathcal{F}^{u}_x$ is arbitrary, it comes $N_{J^{(x)}}=0$ so that $J^{(x)}$ is an integrable almost complex structure on $\mathcal{F}^{u}_x$. Moreover, the equality $(*)$ still holds for $y \in \mathcal{F}^{u}_x$ so that for any $t\in \mathbb R$: 
    \[J^{(\varphi^t(x))}\circ d(\restr{\varphi^t}{\mathcal{F}^{u}_x}) = d(\restr{\varphi^t}{\mathcal{F}^{u}_x}) \circ J^{(x)}.\]
    Since for any $t\in \mathbb R$, $J^{(\varphi^t(x))}$ is integrable, it comes that $d(\restr{\varphi^t}{\mathcal{F}^{u}_x})$ is $\mathbb C$-linear, which means exactly that
    \[\restr{\varphi^t}{\mathcal{F}^{u}_x}: \mathcal{F}^{u}_x \to \mathcal{F}^{u}_{\varphi^t(x)}\]
    is holomorphic with respect to the complex structures induced by $J^{(x)}$ and $J^{(\varphi^t(x))}$.
\end{proof}
\begin{corollary}
\label{cor:conthololeaves}
    Let $(\varphi^t)_{t\in \mathbb R}$ a transversely holomorphic Anosov flow on a smooth compact manifold $M$.\\
    Then the stable and unstable foliations are continuous foliations with holomorphic leaves.
    \end{corollary}
    \begin{proof}
        It comes directly from the local product structure (Proposition \ref{prop:locprodstru}) since we can now use holomorphic diffeomorphisms for local unstable and stable leaves.
    \end{proof}

\begin{remark}
 The unstable $E^{u}$ and   stable $E^{s}$ distributions are thus complex vector bundles over $M$.
\end{remark}

\subsection{Transverse holomorphic structure of the weak foliations}
By following the ideas in \cite{ghys_holomorphic_1995}, we prove:
\begin{theorem}
\label{thm:Fstrholo}
    Let $(\varphi^t)_{t\in \mathbb R}$ a transversely holomorphic Anosov flow on a smooth manifold $M$. Suppose the   unstable distribution $E^{u}$ is of complex dimension $1$.\\
    Then the weak stable foliation $\mathcal{F}^{ws}$ is transversely holomorphic.
\end{theorem}
\begin{proof}
    We will use point $(ii)$ of Proposition \ref{prop:caractrholo} in order to prove that the weak stable foliation $\mathcal{F}^s$ is transversely holomorphic. We will even take small open neighborhoods of   unstable leaves as transversals to $\mathcal{F}^s$. 
    Firstly, take two points $x$ and $y$ on the same weak stable leaf. Take small open neighborhoods $W^{u}_x$ and $W^{u}_y$ of $x$ and $y$ in $\mathcal{F}^{u}_x$ and $\mathcal{F}^{u}_y$ respectively. Assume the distance between $x$ and $y$ is less than $\delta$ so that the holonomy with respect to these transversals does not depend on the path between $x$ and $y$. We note $h$ such holonomy. There exists $y'\in \mathcal{F}^{s}_x$ and $t\in \mathbb R$ such that $y=\varphi^t(y')$. By Lemma \ref{lem:equivmetricweak}, we can take $y$ sufficiently close to $x$ (at a distance $\delta'<\delta$ independent of $x$) such that the distance between $y'$ and $x$ is at most $\delta$. Fix again a small open neighborhood $W^{u}_{y'}$ of $y'$ in $\mathcal{F}^{u}_{y'}$. Then $h$ is the composition of the (path-independent) holonomy between $x$ and $y'$ by the (path-independent) holonomy between $y'$ and $y$. However, this last holonomy is just the germ of $\restr{\varphi^t}{\mathcal{F}^{u}_{y'}}$ (by Corollary \ref{cor:holofloFsu}) which is holomorphic with respect to these transversals (by Proposition \ref{prop:complexFsuFss}).
    Therefore, we can assume that $y\in \mathcal{F}^{s}_x$ and its distance to $x$ is at most $\delta$.\\
    We assume $W^{u}_x$ is compact so that $W^{u}_y=h(W^{u}_x)$ is also compact. 
    Let $k \in \mathbb N$, $W^{k}_x:=\varphi^k(W^{u}_x)$ and $W^{k}_y:=\varphi^k(W^{u}_y)$.
    Recall that the leaves of $\mathcal{F}^{u}$ vary continuously in the $C^\infty$-topology (see \cite{katok_introduction_1995} \cite{hirsch_invariant_1977}).
    Therefore, for $k \in \mathbb N$, there exists a diffeomorphism $ \theta_k: W^k_x \to \mathcal{F}^{u}_{\varphi^k(y)}$ onto its image which satisfies the following properties:
    \begin{enumerate}
        \item The distance between $\theta_k(z')$ and $\varphi^k(h(\varphi^{-k}(z')))$ converges uniformly to $0$ for $z' \in W^{k}_x$ and $k$ goes to $+ \infty$ ;
        \item There is a sequence $(\epsilon_k)_{k \in \mathbb N}$ of positive numbers converging to zero such that for any $v \in TW^k_x$:
        \[(1-\epsilon_k)\|v\| \leq \| d\theta_k(v)\| \leq (1+\epsilon_k) \|v\|.\]
    \end{enumerate}
    Consider now, for $k \in \mathbb N$, $h_k:=\varphi^{-k} \circ \theta_k \circ \varphi^k: W^{u}_x \to \mathcal{F}^{u}_y$ which is a diffeomorphism onto its image. 
    Since each $\varphi^k$ and $\varphi^{-k}$ is holomorphic thus conformal between unstable leaves by Proposition \ref{prop:complexFsuFss}, and since $\theta_k$ is $\frac{1+\epsilon_k}{1-\epsilon_k}$-quasi-conformal, $h_k$ is also $\frac{1+\epsilon_k}{1-\epsilon_k}$-quasi-conformal. In addition,
    $h_k$ converges uniformly to $h$ when $k$ goes to $+\infty$: indeed, let $z \in W^{u}_x$ and $k \in \mathbb N$. Since $\theta_k (\varphi^k(z))$ and $\varphi^k(h(z))$ belong to the same   unstable leaf, it comes 
    \[d_{u}(\varphi^{-k}(\varphi^k(h(z))), \varphi^{-k}(\theta_k (\varphi^k(z)))) \leq C \mu^{-k} d_{u}(\varphi^k(h(z)),\theta_k (\varphi^k(z))). \]
    But then $(1)$ concludes.    
    By a result on quasiconformal mappings (see \cite{ahlfors_quasiconformal_1953}), we deduce that $h$ is $(1+\epsilon)$-quasi-conformal for every $\epsilon>0$, that is $h$ is conformal thus holomorphic.\\
    Now if $x$ and $y$ are two points on the same weak stable leaf, $W^{u}_x$, $W^{u}_y$ are small open neighborhoods of $x$ and $y$ in $\mathcal{F}^{u}_x$ and $\mathcal{F}^{u}_y$ respectively, and if $\alpha$ is any $\mathcal{F}^{ws}$-path joining $x$ and $y$, then as before, we can take intermediate points in the weak stable leaf of $x$ and $y$ whose successive distances are bounded by $\delta$. If we take the same type of transversals at these points, $h(\alpha)^{W^{u}_y,W^{u}_x}$ is the composition of holonomy maps which were just proven to be holomorphic, which gives the result. 
\end{proof}
\begin{corollary}
     Let $(\varphi^t)_{t\in \mathbb R}$ a transversely holomorphic Anosov flow on a smooth manifold $M$. Suppose the   unstable distribution $E^{u}$ is of complex dimension $1$.\\
    Then the weak stable foliation $\mathcal{F}^s$ is $C^\infty$.
\end{corollary}
\begin{proof}
    By the previous Theorem, the weak stable holonomy is holomorphic thus uniformly $C^\infty$ (that is the partial derivatives of the holonomy depend continuously on the second point with respect to which the holonomy is computed), so by Lemma 3.2 of \cite{sadovskaya_uniformly_nodate}, the weak stable foliation is $C^\infty$.
\end{proof}
By exchanging the roles of the   stable and   unstable distributions, the preceding theorem has the following immediate corollary:
\begin{corollary}
\label{cor:4.3.1}
    Let $(\varphi^t)_{t\in \mathbb R}$ a transversely holomorphic Anosov flow on a five dimensional smooth compact manifold $M$.\\
    Then the weak stable and weak unstable foliations are $C^\infty$ transversely holomorphic foliations.
\end{corollary}
\begin{proof}
    Necessarily $\dim_\mathbb C E^{u}=\dim_\mathbb C E^{s}=1$. Since $(\varphi^{-t})_{t \in \mathbb R}$ is a transversely holomorphic Anosov flow whose stable distribution is $E^{u}$ and unstable distribution is $E^{s}$, the previous Theorem and Corollary applied to $(\varphi^{-t})_{t \in \mathbb R}$ gives the result for the weak unstable foliation.
\end{proof}
\begin{corollary}
\label{cor:flotrholocarac}
    Let $(\varphi^t)_{t\in \mathbb R}$ an Anosov flow on a five dimensional smooth compact manifold $M$.\\
    Then its orbit foliation is transversely holomorphic if and only if its weak stable and weak unstable foliations are transversely holomorphic
\end{corollary}
\begin{proof}
    The necessary condition is the previous corollary.
    Conversely, assume that both the weak stable and weak unstable foliations are transversely holomorphic.
    By definition, since both foliations are smooth by Lemma 3.2 of \cite{sadovskaya_uniformly_nodate},
    there exist an atlas of submersions $(U_i, f_i^s)_{i\in I}$ and holomorphic transition maps $\gamma^s_{ij}$ defining the transverse holomorphic structure of the weak stable foliation. There also exist an atlas of submersions $(V_i, f_i^u)_{i\in J}$ and holomorphic transition maps $\gamma^u_{ij}$ defining the transverse holomorphic structure of the weak unstable foliation. 
    By taking a refinement if necessary, we can assume $I=J$ and $V_i=U_i$ for every $i\in I$.
    We define:
    \begin{align*}
        f_i &= (f_i^u, f_i^s)\\
        \gamma_{ij}&=\gamma_{ij}^u \otimes \gamma_{ij}^s.
    \end{align*}
    By Hartog's theorem, $\gamma_{ij}$ is holomorphic because it is holomorphic along each variable. 
    Also, since for every $p\in U_i$
    $$\ker(df_i^s)\cap\ker(df_i^u)=(E^{s}_p \oplus \langle X_p \rangle )\cap (E^{u}_p \oplus \langle X_p \rangle )=\langle X_p \rangle,$$
    $f_i: U_i \to \mathbb C^2$ is a submersion.
    This defines an atlas of submersions $(U_i, f_i)$ onto $\mathbb C^2$ with holomorphic transition maps $\gamma_{ij}$ which is compatible with the one defining the orbit foliation of $(\varphi^t)_{t\in \mathbb R}$, since the latter is uniquely integrable.
    \end{proof}
\section{Complex affine structures of strong leaves and transverse complex projective structures of weak foliations}
\label{sec:6}
We first recall some notions about (transverse) $(G,X)$-structures for (foliated) manifolds (see \cite{thurston_geometry_2022}, \cite{cordero_examples_1990}, \cite{haefliger_groupoides_1984}, \cite{fang_rigidity_2007}, \cite{godbillon_feuilletages_1991}).
\begin{definition}
\label{def:trGXstru}
Let $X$ a smooth connected manifold of dimension $q$ and $G$ a group of diffeomorphisms of $X$ acting analytically on $X$.
A $C^r$-foliation $\mathcal{F}$ of dimension $k \in \mathbb N$ on a smooth manifold of dimension $k+q$ ($q\in \mathbb N^*)$ admits a \emph{transverse $(G,X)$-structure} if there exists an atlas $(U_i, \psi_i)$ of foliated charts compatible with the one defining $\mathcal{F}$ such that:
        \begin{enumerate}
            \item For all $i$, $\psi_i(U_i)=U_i^1 \times U_i^2$, where $U_i^1 \subset \mathbb R^k$ is a connected open ball and $U_i^2$ is the domain of a connected chart of $X$,
            \item For all $i,j$, there exist a map $f_{ij}$ and an element $h_{ij} \in G$, 
                \[\forall(x,y) \in \psi_i(U_i \cap U_j) \subset \mathbb R^k \times X, \; \psi_i \circ \psi_j^{-1}(x,y)=(f_{ij}(x,y), h_{ij}(y)).\]
        \end{enumerate}
        If $r\geq1$, such a structure can be given by a maximal $C^r$ atlas $(U_i,s_i)_i$ of submersions compatible with the one defining $\mathcal{F}$ such that:
        \begin{enumerate}
            \item For all $i$, $s_i:U_i \to X$ is a submersion,
            \item For all $i,j$, there exist an element $\gamma_{ij} \in G$ such that $\restr{s_i}{U_i\cap U_j} = \gamma_{ij}\cdot\restr{s_j}{U_i\cap U_j}.$
        \end{enumerate}
    \end{definition}
    For example, if $X= \mathbf{P}^1(\mathbb C)$ and $G=\mathrm{PSL}(2,\mathbb C)$ is the group of projective automorphisms (homographies) of $\mathbf{P}^1(\mathbb C)$, then a transverse $(G,X)$-structure is exactly a transverse projective structure of complex dimension one. \\
    The following is a characterization of transverse $(G,X)$-structures in terms of holonomy maps.
    \begin{proposition}
    \label{prop:caracGXstru}
    Let $X$ a smooth manifold of dimension $q$ and $G$ a group acting analytically on $X$. Let also a $C^r$-foliation $\mathcal{F}$ of dimension $k \in \mathbb N$ on a smooth manifold of dimension $k+q$ $(q\in \mathbb N^*)$.\\
    Then the following assertions are equivalent:
    \begin{enumerate} 
        \item $\mathcal{F}$ admits a transverse $(G,X)$-structure ;
        \item For any points $x, y$ on the same leaf of $\mathcal{F}$, there exist (hence for all) small transversals to the foliation $T$ and $T'$ at $x$ and $y$ respectively, there exist diffeomorphisms $\Psi: T\to V$ and $\Psi' \to V'$, where $V, V'$ are open sets of $X$, such that for any $\mathcal{F}$-path $\alpha$ joining $x$ and $y$, 
        \[\Psi' \circ h(\alpha)^{T',T} \circ \Psi^{-1}
        \]
        is the restriction of an element of $G$.
    \end{enumerate}
    \end{proposition}
    \begin{proof}
    The proof is analogous to that of transversely holomorphic foliations (see Proposition \ref{prop:caractrholo}).
    \end{proof}
    From a transverse $(G,X)$-structure, we can define two fundamental maps (see \cite{thurston_geometry_2022}, \cite{godbillon_feuilletages_1991}).
    \begin{definition}
    \label{def:devmap}
    Let a foliation $\mathcal{F}$ on a smooth manifold $M$ with a transverse $(G,X)$-structure.\\
    There exists a (non-unique) submersion, from the universal covering space $\widetilde{M}$ of $M$,
    \[ D: \widetilde{M} \to X\]
    called the \emph{developing map} of the transverse $(G,X)$-structure, and a group homomorphism, from the fundamental group $\pi_1(M)$ of $M$,
    \[H: \pi_1(M) \to G\]
    called the \emph{holonomy representation} of $D$ such that:
    \begin{enumerate}
        \item The lifted foliation $\widetilde{\mathcal{F}}$ of $\mathcal{F}$ on $\widetilde{M}$ is defined by $D$; that is the leaves of $\widetilde{\mathcal{F}}$ are precisely the connected components of the fibers of $D$,
        \item $\forall \widetilde{x} \in \widetilde{M}, \; \forall \gamma \in \pi_1(M), \quad D(\gamma \cdot \widetilde{x})= H(\gamma)D(\widetilde{x}).$
    \end{enumerate}
    Moreover, all other developing maps, of the same transverse $(G,X)$-structure, and holonomy representations corresponding to it, are exactly of the form $g\circ D$ and $g\circ H\circ g^{-1}$ respectively, where $g \in G$.\\
    A transverse $(G,X)$-structure for a foliation $\mathcal{F}$ on a manifold $M$ is called \emph{complete} if $\widetilde{M} \xrightarrow{D} X$ is a smooth fiber bundle.
    \end{definition}In case of a $(G,X)$-structure on $M$ (i.e. a transverse $(G,X)$-structure for the foliation by points), $D$ is a local diffeomorphism and the structure is complete precisely if $\widetilde{M} \xrightarrow{D} X$ is a covering map. If $X$ is simply connected, the structure is complete precisely if $D$ is a diffeomorphism.
\subsection{Complex affine structures for the strong leaves}
\label{subsec:6.1}
Here we strengthen the results of Proposition \ref{prop:complexFsuFss} in case the strong unstable distribution is of complex dimension one, by proving that each strong unstable leaf can be uniquely equipped with a complex affine structure invariant by the flow. 
The main result is:

\begin{theorem}
\label{thm:compaffstru}
Let $(\varphi^t)_{t\in \mathbb R}$ a transversely holomorphic Anosov flow on a smooth compact manifold $M$. Suppose the unstable distribution $E^{u}$ is of complex dimension $1$.\\
Then there exists a unique family of complex affine structures on the strong unstable leaves $(\mathcal{F}^{u}_{x})_{x\in M}$, holomorphically compatible with the initial complex structures introduced in Proposition \ref{prop:complexFsuFss} such that: for every $x \in M$ and every $t \in \mathbb R$, the map 
    $\restr{\varphi^t}{\mathcal{F}^{u}_{x}}: \mathcal{F}^{u}_{x} \to \mathcal{F}^{u}_{\varphi^t(x)}$ 
    is affine.\\
Moreover, each of these complex affine structures is complete.
\end{theorem}
We first give some results on jets of diffeomorphisms and how they are linked to complex affine structures. This will lead to an alternative version of Theorem \ref{thm:compaffstru}, in Theorem \ref{thm:AffStruAlternatif}, which we will prove in the next pages.\\

Let $M$ and $N$ two complex manifolds. We note
\[\mathcal{A}_h(M,N)=\{(x, U, f), \, x\in M, \, U \text{ open neighborhood of } x \text{ in } M,\,  f \in \mathcal{H}(U,N)\}.\]
where $\mathcal{H}(U,N)$ denote the set of all holomorphic maps between $U$ and $N$.\\
For $r\in \mathbb N$, we define on $\mathcal{A}_h(M,N)$ an equivalence relation by: $(x, U, f) \sim_r (y, V, g)$ if and only if $x=y, \,  f(x)=g(x)$, there exists (hence for all) $(V, \psi)$ a holomorphic local chart of $M$ centered at $x$, there exists (hence for all) $(W, \theta)$ a holomorphic local chart  of $N$ at $f(x)$ such that the partial derivatives of $\theta\circ f \circ \psi^{-1} \text{ and } \theta \circ g \circ \psi^{-1}$ at $0$ agree up to (and including) the $r$-th order.\\
The equivalence class of $(x, U, f)$ for the equivalence relation $\sim_r$ is noted $\overline{f}^r$ by abuse of notation and is called the \emph{$r$-th jet of the map $f$}. Also, we will use the same notation $\overline{f}^r$ for maps $f$ with different target manifolds, but it will be clear from the context which equivalence relation we refer to.\\

Let $L$ a complex manifold of complex dimension $1$. For $r \in \mathbb N$, we note 
\[
L_r:= \{ \overline{(0, U, f)}^r, \, (0, U, f) \in \mathcal{A}_h(\mathbb{C},L), \, f'(0) \neq 0 \}
\]
the set of $r$-jets of holomorphic diffeomorphisms between open sets of $0 \in \mathbb C^*$ and open sets in $L$.
\begin{lemma}
\label{lem:jetL}
\leavevmode
    \begin{enumerate} 
        \item If $\varphi:L\to L'$ is a holomorphic diffeomorphism, then it induces, for each $r\in \mathbb N$, a homeomorphism 
        \[\varphi_r: \left \{\begin{array}{ccc}
        L_r & \to & L'_r \\
        \overline{f}^r & \mapsto & \overline{\varphi \circ f}^{r}
    \end{array} \right. ;\]
        
        \item $L_0$ is naturally identified with $L$, and $L_1$ is naturally identified with the set of non-zero tangent vectors to L, that is $TL\setminus 0$ ;
        \item If $\varphi:L\to L'$ is a holomorphic diffeomorphism, then $\varphi_1:L_1 \to L'_1$ is identified with the action of the differential of $\varphi$ on non-zero tangent vectors to $L$ ;
        \item There is a chain of fiber bundles:
        \[\cdots \xrightarrow{} L_2 \xrightarrow{\rho_2} L_1 \xrightarrow{\rho_1} L_0\]
        where, for $k \in \mathbb{N^*}$, 
        \[ \rho_k: \left \{\begin{array}{ccc}
        \, L_k & \to & L_{k-1} \\
        \overline{f}^k & \mapsto & \overline{f}^{k-1}
    \end{array} \right.; \]
        \item $L_1  \xrightarrow{\rho_1} L_0$ is a principal $\mathbb C^*$-bundle; the action of $\mathbb C^*$ on $L_1$ is given by:
    \[\left \{\begin{array}{lcl}
        \, \mathbb C^* \times L_1 & \to & L_{1} \\
        (\lambda,\overline{f}^1) & \mapsto & \lambda \cdot \overline{f}^1:= \overline{f(\lambda \cdot \,)}^{1}
    \end{array} \right. \]
    where $f(\lambda \cdot \,) = \left (z \mapsto f(\lambda z) \right ) $ ;
    \item $L_2  \xrightarrow{\rho_1 \circ \rho_2} L_0$ is a principal $\mathbb C^* \times \mathbb C$-bundle; the action of $\mathbb C^* \times \mathbb C$ on $L_2$ is given by:
    \[\left \{\begin{array}{lcl}
        \, \mathbb C^* \times \mathbb C \times L_2 & \to & L_{2} \\
        ((\lambda,\mu),\overline{f}^2) & \mapsto & (\lambda, \mu) \cdot \overline{f}^2:=\overline{f(\frac{\lambda \cdot \,}{1+\mu \, \cdot\,})}^{2}
    \end{array} \right. \]
    where $f(\frac{\lambda \cdot \,}{1+\mu\cdot\,}) = \left (z \mapsto f(\frac{\lambda z }{1+\mu z}) \right) $ ;
    \item $L_2  \xrightarrow{\rho_2} L_1$ is a principal $\mathbb C$-bundle; the action of $\mathbb C$ on $L_2$ is given by:
    \[\left \{\begin{array}{lcl}
        \, \mathbb C \times L_2 & \to & L_2 \\
        (\mu,\overline{f}^2) & \mapsto & (1, \mu) \cdot \overline{f}^2 = \overline{f(\frac{\cdot}{1+\mu\,\cdot\,})}^{2}
    \end{array} \right. \]
    \end{enumerate}
\end{lemma}
\begin{proof}
    $(i)$ comes from the fact that if $(U,\psi)$ is a holomorphic chart for $L'$, then $(\varphi^{-1}(U),\psi \circ \varphi)$ is a holomorphic chart for $L$. \\
    $(ii)$ and $(iii)$ are immediate from the definition of a tangent vector by an equivalence class of curves.\\
    $(iv):$ Let $x \in L$ and $(U, \psi)$ a holomorphic chart for $L$ around $x$. The maps 
    \[\theta: \left \{\begin{array}{lcl}
        \rho_1^{-1}(U)  &\to & U \times \mathbb C^* \\
        (\overline{\gamma}^1)  &\mapsto & (\gamma(0),(\psi \circ \gamma)'(0))
    \end{array} \right. \quad \text{and} \quad \eta: \left \{\begin{array}{lcl}
        \rho_2^{-1}(\rho_1^{-1}(U)) & \to & \rho_1^{-1}(U) \times \mathbb C \\
        (\overline{\gamma}^2) & \mapsto & (\overline{\gamma}^1,(\psi \circ \gamma)''(0))
    \end{array} \right.\]
    are homeomorphisms (by definition of the topology on each $L_r$) whose inverses are respectively
    \[\left \{\begin{array}{l}
        U \times \mathbb C^* \to \rho_1^{-1}(U) \\
        (p,v)  \mapsto  \overline{\psi^{-1}(\psi(p)+vz)}^1
    \end{array} \right. \; \text{and} \;\left \{\begin{array}{l}
        \rho_1^{-1}(U) \times \mathbb C \to \rho_2^{-1}(\rho_1^{-1}(U)) \\
        (\overline{\gamma}^1, w) \quad \mapsto \quad \overline{\psi^{-1}((\psi\circ \gamma)(0)+(\psi\circ \gamma)'(0)z + \frac{w}{2}z^2)}^2
    \end{array} \right. .\]
    They are local trivializations of $L_1 \xrightarrow{\rho_1} L_0$ and $L_2 \xrightarrow{\rho_2} L_1$ respectively.\\
    The proofs of $(v), (vi)$ and $(vii)$ are not difficult.
\end{proof}
As an immediate corollary, by pulling-back the complex structure of $L_1$ and that of $\mathbb C$ thanks to the local trivializations, $L_2$ is a complex manifold. \\
With that in mind, we can link affine structures on a holomorphic curve $L$ with holomorphic $\mathbb C^*$-sections of $L_2 \xrightarrow{\rho_2} L_1$ (see \cite{ghys_holomorphic_1995} Lemma 3.2 p.591):
\begin{lemma}
\label{lem:affstruL}
    Let $L$ a complex manifold of complex dimension one.\\
    There is a one-to-one correspondence between complex affine structures on $L$ and holomorphic sections $\sigma$ of $L_2 \xrightarrow{\rho_2} L_1$ which are equivariant with respect to the $\mathbb C^*$-actions on $L_1$ and $L_2$ defined in Lemma \ref{lem:jetL}, that is holomorphic sections $\sigma$ of $L_2 \xrightarrow{\rho_2} L_1$ satisfying:
    \[\forall \lambda \in \mathbb C^*, \; \forall \overline{f}^1 \in L_1, \quad \sigma(\lambda\cdot \overline{f}^1 ) = (\lambda,0) \cdot\sigma(\overline{f}^1) .\]
\end{lemma}
\begin{lemma}
    Let $\varphi:L\to L'$ a holomorphic diffeomorphism between two complex affine manifolds of complex dimension $1$. Denote by $\sigma$ and $\sigma '$ the holomorphic $\mathbb C^*$-equivariant sections of $L_2 \xrightarrow{\rho_2} L_1$ and $L'_2 \xrightarrow{\rho'_2} L'_1$ respectively corresponding to the given affine structures of $L$ and $L'$.\\
    Then $\varphi:L\to L' $ is affine if and only if $\varphi_2 \circ \sigma = \sigma'\circ \varphi_1.$
\end{lemma}
\begin{proof}
    Assume $\varphi: L \to L'$ affine. Let $\overline{f}^1 \in L_1$ and $(U', \psi')$ a complex affine chart of $L'$ centered at $\varphi(f(0))$ containing $\text{Im}(\varphi \circ f)$. By assumption, $(\varphi^{-1}(U'),\psi' \circ \varphi)$ is a complex affine chart of $L$ centered at $f(0)$. Therefore
    \begin{align*}
        \varphi_2^{-1} (\sigma'(\varphi_1(\overline{f}^1)))&=  \varphi_2^{-1}(\sigma'(\overline{\varphi \circ f}^1))=
        \varphi_2^{-1} \left (\overline{\psi'^{-1}((\psi' \circ\varphi \circ f)(0)+(\psi' \circ \varphi \circ f)'(0)z) )}^2 \right)\\
        &=\overline{(\varphi^{-1}\circ\psi'^{-1})((\psi' \circ\varphi \circ f)(0)+(\psi' \circ \varphi \circ f)'(0)z) )}^2=\sigma(\overline{f}^1).
    \end{align*}
    Conversely, the above equality shows that for any $\overline{f}^1 \in L_1$, if $(U, \psi)$ is a complex affine chart for $L$ at $f(0) \in L$ containing $\text{Im}(f)$, and if $(U', \psi')$ is a complex affine chart for $L'$ at $\varphi(f(0)) \in L'$ containing $\text{Im}(\varphi \circ f)$, then 
    \[\overline{(\varphi^{-1}\circ\psi'^{-1})((\psi' \circ\varphi \circ f)(0)+(\psi' \circ \varphi \circ f)'(0)z) )}^2=\overline{\psi^{-1}(\psi(f(0))+(\psi \circ f)'(0)z)}^2\]
    thus
    $(\psi' \circ \varphi \circ \psi ^{-1})''(0)=0$. Therefore, if $p \in L$, if $(U, \psi)$ is a complex affine chart for $L$ at $p$, and if $(U', \psi')$ is a complex affine chart for $L'$ at $\varphi(p) \in L'$ containing $\varphi(U)$, then $(\psi' \circ \varphi \circ \psi ^{-1})''(0)=0$. So fix such $p$ and affine charts. Assume, by restricting and translating, that $\psi(U)$ is an open ball centered at $0 \in \mathbb C$. Let $z_0 \in \psi(U)$. The map $-z_0 + \psi$ is also an affine chart for $L$ at $p$, whose inverse is $\psi^{-1}(\cdot +z_0)$. As a result, 
    \[(\psi' \circ \varphi \circ (\psi ^{-1}(\cdot +z_0)))''(0)=(\psi' \circ \varphi \circ \psi ^{-1})''(z_0)=0.\]
    Therefore $\psi' \circ \varphi \circ \psi ^{-1}$ is affine and thus $\varphi$ also by definition.
\end{proof}
Let's reconsider our transversely holomorphic Anosov flow $(\varphi^t)_{t\in \mathbb R}$ on $M$ smooth compact manifold.  For $r \in \mathbb N$, we note 
\[
M_r:= \{ \overline{(0, U, f)}^r, \, (0, U, f) \in \mathcal{A}_h(\mathbb{C},L), \, L \text{ leaf of } \mathcal{F}^{u}, \,  f'(0) \neq 0\} = \bigsqcup_{L \text{ leaf of } \mathcal{F}^{u}} L_r 
\]
the set of $r$-jets of holomorphic diffeomorphisms between open sets of $0 \in \mathbb C^*$ and open sets in leaves of the strong stable foliation $\mathcal{F}^{u}$. \\
As an analog of Lemma \ref{lem:jetL}, we have the following lemma:
\begin{lemma}
\label{lem:jetM}
\leavevmode
    \begin{enumerate} 
        \item The flow $(\varphi^t)_{t\in \mathbb R}$ induces, for each $r\in \mathbb N$ and $t \in \mathbb R$, a homeomorphism 
        \[(\varphi^t)_r: \left \{\begin{array}{lcl}
        \, M_r & \to & M_r \\
        \overline{f}^r & \mapsto & \overline{\varphi^t \circ f}^{r}
        \end{array} \right. ;\]
        such that for every $x \in M$,
        \[\restr{(\varphi^t)_r}{(\mathcal{F}^{u}_{x})_r}:(\mathcal{F}^{u}_{x})_r \to (\mathcal{F}^{u}_{\varphi^t(x)})_r\]
        is a well defined homeomorphism, satisfying moreover:
        \begin{align*}
            \forall t,s \in \mathbb R, \quad (\varphi^t)_r \circ (\varphi^s)_r =(\varphi^{t+s})_r \; ;
        \end{align*}
        \item $M_0$ is naturally identified with $M$, and $M_1$ is naturally identified with the set of non-zero tangent vectors to the strong unstable foliation, that is $T\mathcal{F}^{u}\setminus 0$ ;
        \item For every $t\in \mathbb R$, $(\varphi^t)_1:M_1 \to M_1$ is identified with the action of the differential of $\varphi^t$ on non-zero tangent vectors to the strong unstable foliation ;
        \item There is a chain of fiber bundles:
        \[\cdots \xrightarrow{} M_2 \xrightarrow{\rho_2} M_1 \xrightarrow{\rho_1} M_0\]
        where, for $k \in \mathbb{N^*}$, 
        \[ \rho_k: \left \{\begin{array}{lcl}
        \, M_k & \to & M_{k-1} \\
        \overline{f}^k & \mapsto & \overline{f}^{k-1}
    \end{array} \right.; \]
        \item $M_1  \xrightarrow{\rho_1} M_0$ is a principal $\mathbb C^*$-bundle; the action of $\mathbb C^*$ on $M_1$ is given by:
    \[\left \{\begin{array}{lcl}
        \, \mathbb C^* \times M_1 & \to & M_1 \\
        (\lambda,\overline{f}^1) & \mapsto & \overline{f(\lambda \cdot \,)}^{1}
    \end{array} \right. ;\]
    \item $M_2  \xrightarrow{\rho_1 \circ \rho_2} M_0$ is a principal $\mathbb C^* \times \mathbb C$-bundle; the action of $\mathbb C^* \times \mathbb C$ on $M_2$ is given by:
    \[\left \{\begin{array}{lcl}
        \, \mathbb C^* \times \mathbb C \times M_2 & \to & M_2 \\
        ((\lambda,\mu),\overline{f}^2) & \mapsto & \overline{f(\frac{\lambda \cdot \,}{1+\mu \, \cdot\,})}^{2}
    \end{array} \right. ; \]
    \item $M_2  \xrightarrow{\rho_2} M_1$ is a principal $\mathbb C$-bundle; the action of $\mathbb C$ in $M_2$ is given by:
    \[\left \{\begin{array}{lcl}
        \, \mathbb C \times M_2 & \to & M_2 \\
        (\mu,\overline{f}^2) & \mapsto & \overline{f(\frac{\cdot}{1+\mu\,\cdot\,})}^{2}
    \end{array} \right. \]
    \end{enumerate}
\end{lemma}
\begin{proof}
    As before, the proof of $(i), (ii)$ and $(iii)$ is not difficult.\\
    $(iv):$ By Corollary \ref{cor:conthololeaves} , let $(U_i, \psi_i)_i$ a foliated $C^0$ atlas for the unstable foliation $\mathcal{F}^{u}$ such that the restriction of each chart to a unstable leaf is holomorphic. We note $\psi_i(U_i) = U_i^1 \times U_i^2$ where $U_i^1$ and $U_i^2$ are open balls of $\mathbb C$ and $\mathbb R^{n-2}$ respectively.
    If an unstable leaf $L$ is understood from the context, we will still write $\psi_i$ instead of the heavy notation $\text{pr}_1 \circ \restr{\psi_i}{(L\cap U_i)_b}$ to mean the holomorphic diffeomorphism which sends the connected component $(L\cap U_i)_b$ of the leaf $L$ restricted to $U_i$, to $U_i^1 \subset \mathbb C$.
    We can define the maps
    \[\theta_i: \left \{\begin{array}{ccl}
        \rho_1^{-1}(U_i)  & \to & U_i \times \mathbb C^* \\
        (\overline{\gamma}^1) & \mapsto & (\gamma(0),(\psi_i \circ \gamma)'(0))
    \end{array} \right. \; \text{and} \; \eta_i: \left \{\begin{array}{ccl}
        \rho_2^{-1}(\rho_1^{-1}(U_i)) & \to & \rho_1^{-1}(U_i) \times \mathbb C \\
        (\overline{\gamma}^2) & \mapsto & (\overline{\gamma}^1,(\psi_i \circ \gamma)''(0))
    \end{array} \right. .\]
    Let
    \[\Theta_i: \left \{\begin{array}{ccl}
       U_i \times \mathbb C^* & \to & \rho_1^{-1}(U_i)   \\
        (p,v) & \mapsto & \overline{\psi_{i}^{-1}(\psi_{i}(p)+zv)}^1.
    \end{array} \right. \]
    We can check that $\theta_i$ and $\Theta_i$ are inverse of each other. Moreover, by definition of the (disjoint union) topology on each $M_r$, $\theta_i$ and $\Theta_i$ are continuous so we have defined a local trivialization of $L_1 \xrightarrow{\rho_1} L_0$.\\
    Now let $(\overline{\gamma}^1, w) \in \rho_1^{-1}(U_i) \times \mathbb C$.
    It is not difficult to prove that $\overline{\gamma \left(z+ \frac{z^2}{2} \cdot  \frac{w-(\psi_i \circ \gamma)''(0)}{(\psi_i \circ \gamma)'(0)} \right)}^2$
    does not depend on $\alpha \in \overline{\gamma}^1$. 
Therefore, we can define
    \[H_i: \left \{\begin{array}{ccl}
        \rho_1^{-1}(U_i) \times \mathbb C & \to & \rho_2^{-1}(\rho_1^{-1}(U_i))  \\
        (\overline{\gamma}^1,w) & \mapsto &  \overline{\gamma \left(z+ \frac{z^2}{2} \cdot \frac{w-(\psi_i \circ \gamma)''(0)}{(\psi_i \circ \gamma)'(0)} \right)}^2
    \end{array} \right. \]
and check that $\eta_i$ and $H_i$ are inverse of each other. Moreover, by definition of the (disjoint union) topology on each $M_r$, $\eta_i$ and $H_i$ are continuous so we have defined a local trivialization of $L_2 \xrightarrow{\rho_2} L_1$.\\
The proof of $(v), (vi)$ and $(vii)$ is in the same spirit as that of Lemma 6.1 and are left aside.
\end{proof}
With all these lemmas, we understand that we want to prove:
\begin{theorem}
\label{thm:AffStruAlternatif}
    Let $(\varphi^t)_{t\in \mathbb R}$ a transversely holomorphic Anosov flow on a smooth compact manifold $M$. Suppose the unstable distribution $E^{u}$ is of complex dimension $1$.\\
    Then there exists a unique continuous section $\sigma$ of $M_2 \xrightarrow{\rho_2} M_1$ satisfying:
    \begin{enumerate} 
        \item For every leaf $L$ of $\mathcal{F}^{u}$, $\restr{\sigma}{L_1}:L_1 \to L_2$ is holomorphic;
        \item $\sigma$ is $\mathbb C^*$-equivariant, that is:
        $\forall \lambda \in \mathbb C^*, \; \forall \overline{f}^1 \in M_1, \quad \sigma(\lambda\cdot \overline{f}^1 ) = (\lambda,0) \cdot\sigma(\overline{f}^1) \;; $
        \item $\forall t \in \mathbb R, \quad (\varphi^t)_2 \circ \sigma = \sigma\circ (\varphi^t)_1. $
    \end{enumerate}
    Moreover, for each leaf $L$ of $\mathcal{F}^{u}$, the complex affine structure induced by $\restr{\sigma}{L_1}$ is complete.

\end{theorem}
We will apply the fixed point theorem to a certain map from a certain complete metric space in order to guarantee the existence, the uniqueness and property $(iii)$ of such $\sigma$. $(ii)$ will be automatic from the construction of our space. However we will need to prove it is not empty. We'll also have to prove $(i)$ and the completeness of the affine structures thus constructed.
We split the proof of the previous theorem into several lemmas.\\
    Let's call $\mathcal{H}$ the set of continuous sections $\sigma$ of the $\mathbb C$-fiber bundle $M_2 \xrightarrow{\rho_2} M_1$ which are $\mathbb C^*$ equivariant with respect to the $\mathbb C^*$-actions introduced in Lemma \ref{lem:jetM}.
    \begin{lemma}
        $\mathcal{H}$ is in bijective correspondence with the set of continuous sections of a fiber bundle $E \xrightarrow{\pi} M$ whose fibers are contractible. In particular, $\mathcal{H}$ is not empty.
    \end{lemma}
    \begin{proof}
    For $x \in M$, let $E_x$ the set of continuous sections of $\rho_2^{-1} (\rho_1^{-1}(x)) \xrightarrow{\rho_2} \rho_1^{-1}(x)$ which are $\mathbb C^*$-equivariant.
    We know that $\rho_2^{-1} (\rho_1^{-1}(x))$ is homeomorphic to $\mathbb C^* \times \mathbb C$ and $\rho_1^{-1}(x)$ is homeomorphic to $\mathbb C^*$. With this identification, the $\mathbb C^*$-action on $\rho_2^{-1} (\rho_1^{-1}(x))$ defined in Lemma \ref{lem:jetM} corresponds to the action
         $\lambda \cdot (\mu,w)  =(\lambda\mu,\lambda^2w)$
    while the $\mathbb C^*$-action on $\rho_1^{-1}(x)$ corresponds to the trivial action. Therefore, $E_x$ is identified with the set of continuous maps $s: \mathbb C^* \to \mathbb C$ satisfying
    $s(\lambda \mu)=\lambda^2 s(\mu)$ for all $\lambda,\mu \in \mathbb C^*$, 
    that is the set of continuous maps $s: \mathbb C^* \to \mathbb C$ satisfying
    $s(\lambda)=\lambda^2 s(1)$ for all $\lambda,\mu \in \mathbb C^*$.
    So each $E_x$ is homeomorphic to $\mathbb C$.
    Now let $E:= \bigsqcup_{x \in M}E_x$ endowed with the disjoint union topology and let $\pi: E \to M$ the natural projection map. It is naturally a $\mathbb C$-fiber bundle over $M$. Denote by $\Gamma(E, \pi)$ the set of continuous sections of $E \xrightarrow{\pi} M$. 
    It is not difficult to see that there is a natural bijection between $\Gamma(E, \pi)$ and $\mathcal{H}$.
    A classical result in algebraic topology (see \cite{james_general_1984}) says that a fiber bundle over a paracompact Hausdorff space whose fibers are contractible admits necessarily a global continuous section. Therefore, since $\mathbb C$ is contractible, $\Gamma(E, \pi)$ and thus $\mathcal{H}$ is not empty.
    \end{proof}
    Let $\eta_i$ a local trivialization of $M_2 \xrightarrow{\rho_2} M_1$ over $\rho_1^{-1}(U_i)$. A global section $\sigma$ of $M_2 \xrightarrow{\rho_2} M_1$ can be expressed on $\rho_1^{-1}(U_i)$ as 
    $\sigma(\overline{\gamma}^1)=\eta_i^{-1}(\overline{\gamma}^1, w_i(\overline{\gamma}^1))$, where the map $w_i: \rho_1^{-1}(U_i) \to \mathbb C $ is continuous.
    One can show that $w_i$ and $w_j$ are linked by the \textit{global section equation} for $\overline{\gamma}^1 \in \rho_1^{-1}(U_i \cap U_j)$,
    \[\frac{w_i(\overline{\gamma}^1)-(\psi_i \circ \gamma)''(0)}{(\psi_i \circ \gamma)'(0)} = \frac{w_j(\overline{\gamma}^1)-(\psi_j \circ \gamma)''(0)}{(\psi_j \circ \gamma)'(0)}.\]
    Conversely a family $(w_i:\rho_1^{-1}(U_i) \to \mathbb C)_i$ of continuous maps satisfying the global section equation for each $(i,j)$ gives rise to a global continuous section $\sigma$ of $M_2 \xrightarrow{\rho_2} M_1$ defined on $\rho_1^{-1}(U_i)$ by $\sigma(\overline{\gamma}^1)=\eta_i^{-1}(\overline{\gamma}^1, w_i(\overline{\gamma}^1))$.
    Moreover, if $L$ is a strong unstable leaf, by the definition of the complex structure on $L_2$, then $\restr{\sigma}{L_1}:L_1 \to L_2$ is holomorphic if and only if for each $i$, $\restr{w_i}{\rho_1^{-1}(L \cap U_i)}: \rho_1^{-1}(L \cap U_i) \to \mathbb C$ is holomorphic. 
    Also, $\sigma$ is $\mathbb C^*$-equivariant if and only if for each $i$, for every $\lambda \in \mathbb C^*$ and $\overline{\gamma}^1 \in \rho_1^{-1}(U_i)$, $w_i(\lambda\cdot\overline{\gamma}^1) = \lambda^2 w_i(\overline{\gamma}^1)$.\\
    Now let $\sigma, \sigma^*$ global sections of $M_2 \xrightarrow{\rho_2} M_1$. Call respectively $(w_i)_i$ and $(w_i^*)_i$ the corresponding family of continuous maps to $\mathbb C$. The global section equations satisfied by $\sigma$ and $\sigma^*$ on each $\rho_1^{-1}(U_i \cap U_j)$ yield by subtracting one to the other:
    \[\frac{w_i(\overline{\gamma}^1)-w_i^*(\overline{\gamma}^1)}{(\psi_i \circ \gamma)'(0)}=\frac{w_j(\overline{\gamma}^1)-w_j^*(\overline{\gamma}^1)}{(\psi_j \circ \gamma)'(0)}.\]
    As a result, we can define the difference map $(\sigma - \sigma^*): M_1 \to \mathbb C$ by
    \[ (\sigma - \sigma^*): 
          \overline{\gamma}^1 \mapsto \frac{w_i(\overline{\gamma}^1)-w_i^*(\overline{\gamma}^1)}{(\psi_i \circ \gamma)'(0)} \quad \text{ if } \overline{\gamma}^1 \in \rho_1^{-1}(U_i).\]
    It is a continuous map which allows us to define a distance on $\mathcal{H}$ as follows. Fix a Riemannian metric $\| . \|$ on $M$, so that it defines a norm on $TM$ and thus on $M_1 \cong T\mathcal{F}^{u} \setminus 0$. Let $\sigma, \sigma^* \in \mathcal{H}$. Since for every $\lambda \in \mathbb C^*$ and $\overline{\gamma}^1 \in M_1$ (for example $\overline{\gamma}^1 \in\rho_1^{-1}(U_i)$), 
    \begin{align*}
        (\sigma-\sigma^*)(\lambda \cdot \overline{\gamma}^1)= \frac{\lambda^2 w_i(\overline{\gamma}^1)-\lambda^2w_i^*(\overline{\gamma}^1)}{\lambda(\psi_i \circ \gamma)'(0)}= \lambda \, (\sigma-\sigma^*)(\overline{\gamma}^1) \;,
    \end{align*}
    then for every $\overline{\gamma}^1 \in M_1$, 
    $\frac{(\sigma-\sigma^*)(\overline{\gamma}^1)}{\|\overline{\gamma}^1\|} = (\sigma-\sigma^*)(\frac{1}{\|\overline{\gamma}^1\|} \cdot \overline{\gamma}^1). $
    Denote by $S^1\mathcal{F}^{u}:= S^1M \cap T\mathcal{F}^{u}$ the unit tangent bundle to the strong unstable foliation with respect to the metric $\|.\|$. It is a smooth $\mathbb S^1$-fiber bundle over $M$, and is therefore compact.
    As a result, we can define the distance between $\sigma$ and $\sigma^*$:
    \[d(\sigma, \sigma^*):=\sup_{\overline{\gamma}^1 \in M_1} \frac{|(\sigma-\sigma^*)(\overline{\gamma}^1)|}{\|\overline{\gamma}^1\|} < \infty.\]
    It is straightforward to see that $d$ is indeed a metric on $\mathcal{H}$. Moreover,
    \begin{lemma}
        $(\mathcal{H},d)$ is a complete metric space.
    \end{lemma}
    \begin{proof}
        It is a standard completeness proof. 
    \end{proof}
    For $t \in \mathbb R$, one can show that $\overline{\varphi^t}: \left \{ \begin{array}{ccl}
              \mathcal{H} & \to & \mathcal{H} \\
              \sigma & \mapsto & (\varphi^t)_2 \circ \sigma \circ (\varphi^{-t})_1
        \end{array} \right. $
        is a well-defined map, satisfying moreover for $t,s \in \mathbb R$, $\overline{\varphi^t} \circ \overline{\varphi^s}= \overline{\varphi^{t+s}}$.
        \begin{lemma}
            For every $t_0 >0$, there exists $k \in \mathbb N$ such that $(\overline{\varphi^{-t_0}})^k$ is contracting for the metric $d$.
        \end{lemma}
    \begin{proof}
        Fix $t_0>0$ and let $k \in \mathbb Z$. Take $\overline{\gamma}^1 \in M_1$ and suppose $(\varphi^{-t_0k})_1(\overline{\gamma}^1)\in \rho_1^{-1}(U_{i_k})$. Then
        \begin{align*}
            \sigma((\varphi^{-t_0k})_1(\overline{\gamma}^1))&=\eta_{i_k}^{-1}\left(\overline{\varphi^{-t_0k}\circ\gamma}^1, w_{i_k}(\overline{\varphi^{-t_0k}\circ\gamma}^1) \right)\\
            &=\overline{\varphi^{-t_0k}\circ\gamma \left(z+ \frac{z^2}{2} \cdot \frac{w_{i_k}(\overline{\varphi^{-t_0k}\circ\gamma}^1)-(\psi_{i_k} \circ \varphi^{-t_0k}\circ\gamma)''(0)}{(\psi_{i_k} \circ \varphi^{-t_0k}\circ\gamma)'(0)} \right) }^2
        \end{align*}
        so
        \[\overline{\varphi^{-t_0k}}(\sigma)(\overline{\gamma}^1)=\overline{\gamma \left(z+ \frac{z^2}{2} \cdot \frac{w_{i_k}(\overline{\varphi^{-t_0k}\circ\gamma}^1)-(\psi_{i_k} \circ \varphi^{-t_0k}\circ\gamma)''(0)}{(\psi_{i_k} \circ \varphi^{-t_0k}\circ\gamma)'(0)} \right) }^2.\]
        Therefore, if we note, for $k \in \mathbb Z$, $(v_i^{(k)})_i$ the family of continuous $\mathbb C^*$-equivariant maps corresponding to $\overline{\varphi^{-t_0k}}(\sigma)$, then, for $\overline{\gamma}^1\in \rho_1^{-1}(U_{i_0})$, one can prove that 
        \begin{align*}
            v_{i_0}^{(k)}(\overline{\gamma}^1)=(\psi_{i_0}\circ \gamma)''(0)+(\psi_{i_0}\circ \gamma)'(0)\cdot \frac{w_{i_k}(\overline{\varphi^{-t_0k}\circ\gamma}^1)-(\psi_{i_k} \circ \varphi^{-t_0k}\circ\gamma)''(0)}{(\psi_{i_k} \circ \varphi^{-t_0k}\circ\gamma)'(0)}.
        \end{align*}
        Now, take $\sigma, \sigma^* \in \mathcal{H}$ and denote by $(w_i)_i$ and $(w_i^*)_i$ the family of continuous $\mathbb C^*$-equivariant maps to $\mathbb C$ corresponding to $\sigma$ and $\sigma^*$ respectively. \\
        We compute the distance $d(\overline{\varphi^{-t_0k}}(\sigma), \overline{\varphi^{-t_0k}}(\sigma^*))$: for $\overline{\gamma}^1 \in \rho_1^{-1}(U_{i_0})$, by keeping the above notations:
        \begin{align*}
            \frac{|(\overline{\varphi^{-t_0k}}(\sigma)-\overline{\varphi^{-t_0k}}(\sigma^*))(\overline{\gamma}^1)|}{\|\overline{\gamma}^1\|}&=\frac{1}{\|\overline{\gamma}^1\|}\cdot \left|\frac{v_{i_0}^{(k)}-v_{i_0}^{*(k)}}{(\psi_{i_0}\circ \gamma)'(0)}\right|
            =\frac{1}{\|\overline{\gamma}^1\|}\cdot\left |\frac{w_{i_k}(\overline{\varphi^{-t_0k}\circ\gamma}^1)-w_{i_k}^*(\overline{\varphi^{-t_0k}\circ\gamma}^1)}{(\psi_{i_k} \circ \varphi^{-t_0k}\circ\gamma)'(0)}\right|\\
            &=\frac{|(\sigma-\sigma^*)(\overline{\varphi^{-t_0k}\circ\gamma}^1)|}{\|\overline{\gamma}^1\|}=\frac{\|\overline{\varphi^{-t_0k}\circ\gamma}^1\|}{\|\overline{\gamma}^1\|} \cdot\frac{|(\sigma-\sigma^*)(\overline{\varphi^{-t_0k}\circ\gamma}^1)|}{\|\overline{\varphi^{-t_0k}\circ\gamma}^1\|}\\
            &\leq \frac{\|(\varphi^{-t_0k})_1(\overline{\gamma}^1)\|}{\|\overline{\gamma}^1\|} \cdot d(\sigma, \sigma^*) \leq \mu^{-t_0k} d(\sigma, \sigma^*).
        \end{align*}
        Therefore, for $t_0>0$ fixed and $k \in \mathbb N$ large enough so that $\mu^{-t_0k} \leq \frac{1}{2}$, we conclude.
        \end{proof}
    \begin{proof}[Proof of Theorem \ref{thm:AffStruAlternatif}]
        By the contracting mapping principle, for every $t_0>0$, there exists a unique element $\sigma_{t_0} \in \mathcal{H}$ such that $(\varphi^{-t_0})_2 \circ \sigma_{t_0} \circ (\varphi^{{t_0}})_1 = \sigma_{t_0}$.
        Let $\sigma:=\sigma_1$. We want to prove that for any $t_0>0$, $\sigma_{t_0} = \sigma$: since $((\varphi^{-{t_0}})_r)^{-1}=(\varphi^{{t_0}})_r$ for every $r \in \mathbb N$, it will follow that $(\varphi^{t_0})_2 \circ \sigma = \sigma \circ (\varphi^{{t_0}})_1$ for every $t_0 \in \mathbb R$, i.e. $(iii)$ of Theorem \ref{thm:compaffstru}. As a matter of fact, we want to show that for every $t_0>0$, $\sigma=\sigma_1$ is a fixed point of $\overline{\varphi^{-t_0}}$. But we know by definition of $\sigma_1$ that $\overline{\varphi^1}(\sigma_1) = \sigma_{1}.$
        So by composing this equality with $\overline{\varphi^{-t_0}}$, and using $(ii)$ of the previous claim, it comes $\overline{\varphi^1}(\overline{\varphi^{-t_0}}(\sigma_1)) = \overline{\varphi^{-t_0}}(\sigma_1).$
        Therefore, the element $\overline{\varphi^{-t_0}}(\sigma_1) \in \mathcal{H}$ is a fixed point of $\overline{\varphi^{1}}$, that is $\overline{\varphi^{-t_0}}(\sigma_1)=\sigma_1$, so $\sigma_1$ is a fixed point of $\overline{\varphi^{-t_0}}$ which gives the result.\\
        It remains to prove that for each strong unstable leaf $L$, $\restr{\sigma}{L_1}: L_1 \to L_2$ is holomorphic. This would imply the existence of a section satisfying $(i)$, $(ii)$ and $(iii)$ of Theorem \ref{thm:compaffstru}. The uniqueness of such section is immediate by the contraction mapping principle applied to $\overline{\varphi^{1}}$ on $\mathcal{H}$ for example.\\
        For every strong unstable leaf $L_\alpha$ and each $i \in I$, consider a connected component $(L_\alpha\cap U_i)_b$ of $L_\alpha\cap U_i$. We can define a complex affine structure on $(L_\alpha\cap U_i)_b$ by transporting the complex affine structure of $U_i^1 \subset \mathbb C$ via $\text{pr}_1\circ \restr{\psi_i}{(L_\alpha\cap U_i)_b}$, still denoted $\psi_i$ without any risk of confusion. This (unnatural) complex affine structure is holomorphically compatible with the initial complex structure. 
        Thus the space of $r$-jets of holomorphic diffeomorphisms (for the unnatural complex affine structure) from an open set of $0\in \mathbb C$ to an open set of $(L_\alpha\cap U_i)_b$ is holomorphically diffeomorphic to $((L_\alpha\cap U_i)_b)_r$ (which is $\rho_1^{-1}((L_\alpha\cap U_i)_b)$ if $r=1$ and $\rho_2^{-1}(\rho_1^{-1}((L_\alpha\cap U_i)_b))$ if $r=2$).
        By lemma \ref{lem:affstruL}, there exists a unique $\mathbb C^*$-equivariant continuous section $s_{\alpha,i,b}$ of $((L_\alpha\cap U_i)_b)_2 \xrightarrow{\rho_2} ((L_\alpha\cap U_i)_b)_1$ defining this unnatural complex affine structure on $(L_\alpha\cap U_i)_b$. We define, for $i \in I$ and $\overline{\gamma}^1\in \rho_1^{-1}((L_\alpha\cap U_i)_b)$, 
        $s_i(\overline{\gamma}^1)=s_{\alpha,i,b}(\overline{\gamma}^1)$.
        The map $s_i$ is a continuous $\mathbb C^*$-equivariant section of $\rho_2^{-1}(\rho_1^{-1}(U_i))\xrightarrow{\rho_2} \rho_1^{-1}(U_i)$.\\
        For $\overline{\gamma}^1 \in M_1$ and $k \in \mathbb N$, let $i_k=i(k, \overline{\gamma}^1) \in I$ such that $\varphi^{-k}(\gamma(0)) \in U_{i_k}$. 
        By definition of $s_i$, we can establish that 
            \[\left( (\varphi^k)_2\circ s_{i_k} \circ (\varphi^{-k})_1 \right)(\overline{\gamma}^1)= \overline{\gamma\left( z-\frac{z^2}{2}\cdot \frac{(\psi_{i_k} \circ \varphi^{-k} \circ \gamma)''(0) }{(\psi_{i_k} \circ \varphi^{-k} \circ \gamma))'(0)}\right)}^2.\]
    In order to prove that for each strong unstable leaf $L$, $\restr{\sigma}{L_1}: L_1 \to L_2$ is holomorphic, we consider $\overline{\gamma}^1 \in (L\cap U_{i_0})_1$.
    We write 
    \[\eta_{i_0}(\left( (\varphi^k)_2\circ s_{i_k} \circ (\varphi^{-k})_1 \right)(\overline{\gamma}^1))=(\overline{\gamma}^1, w_{i_0}^{(k)}(\overline{\gamma}^1)).\]
    By the same calculus as before, 
     \begin{align*}
            w_{i_0}^{(k)}(\overline{\gamma}^1)=(\psi_{i_0}\circ \gamma)''(0)-(\psi_{i_0}\circ \gamma)'(0)\cdot \frac{(\psi_{i_k} \circ \varphi^{-k}\circ\gamma)''(0)}{(\psi_{i_k} \circ \varphi^{-k}\circ\gamma)'(0)}.
        \end{align*}
    
    We still write $(w_i)_i$ the family of continuous $\mathbb C^*$-equivariant maps to $\mathbb C$ corresponding to $\sigma$. We want to show that the restriction of each $w_{i}$ to $L$ is holomorphic, or for any small open set $V_i$ of $L\cap U_i$, $\restr{w_i}{V_i}$ is holomorphic.
    First, it can be proven that the $\varphi$-invariance property $(iii)$ of $\sigma$ translates for $(w_i)$ as:
        \begin{align*}
             \frac{w_{i_0}(\overline{\gamma}^1)-(\psi_{i_0} \circ \gamma)''(0)}{(\psi_{i_k} 
             \circ\gamma)'(0)}=\frac{w_{i_k}(\overline{\varphi^{-k}\circ\gamma}^1)-(\psi_{i_k} \circ \varphi^{-k}\circ\gamma)''(0)}{(\psi_{i_k} \circ \varphi^{-k}\circ\gamma)'(0)} \qquad\text{ for }k \in \mathbb Z.
        \end{align*}
    Therefore, we obtain the following equality for $k \in \mathbb Z$ and $\overline{\gamma}^1 \in \rho_1^{-1}(U_{i_0})$,
    \begin{align*}
           \frac{w_{i_0}^{(k)}(\overline{\gamma}^1)-w_{i_0}(\overline{\gamma}^1)}{(\psi_{i_0} \circ\gamma)'(0)}=-\frac{w_{i_k}(\overline{\varphi^{-k}\circ\gamma}^1)}{(\psi_{i_k} \circ \varphi^{-k}\circ\gamma)'(0)}.
        \end{align*}
    Let $\overline{\gamma}^1 \in (L\cap U_{i_0})_1$ and $V_{i_0}$ a small neighborhood of $\overline{\gamma}^1$ in $(L\cap U_{i_0})_1$.
    Since the action of $(\varphi)_1$ on $M_1$ is identified to the action of the differential $d\varphi$ on $T\mathcal{F}^{u}\setminus 0$, then by compactness of $M$, if $V_{i_0}$ is taken sufficiently small, for every $k\in \mathbb N$ and $\overline{\alpha}^1 \in V_{i_0}$, $i(k,\overline{\alpha}^1)=i(k,\overline{\gamma}^1)$. We can assume $V_{i_0}$ compact also. \\
    For $\overline{\alpha}^1 \in M_1$, let \[f(\overline{\alpha}^1)=\sup_{i\in I}\left|(\psi_i \circ \alpha)'(0) \right| < \infty \quad \text{ and } \quad g(\overline{\alpha}^1)=\sup_{i\in I}|w_i(\overline{\alpha}^1)| < \infty. \] $f$ and $g$ are well defined continuous maps from $M_1$ to $\mathbb R^+$. Therefore, they are bounded on $S^1\mathcal{F}^{u}$ and $V_{i_0}$ and attain their bounds. Let 
    \[C_0=\sup(\restr{f}{V_{i_0}})<\infty, \, C_1'=\inf(\restr{f}{S^1\mathcal{F}^{u}})>0, \, C_2=\sup(\restr{g}{S^1\mathcal{F}^{u}})<\infty, \, C_3=\sup_{\overline{\alpha}^1 \in V_{i_0}} \|\overline{\alpha}^1\| <\infty .\]
    It comes for $k \in \mathbb N$ and $\overline{\alpha}^1 \in V_{i_0}$:
    \begin{align*}
        &|w_{i_0}^{(k)}(\overline{\alpha}^1)- w_{i_0}(\overline{\alpha}^1)| = \frac{|(\psi_{i_0}\circ\alpha)'(0)|}{|(\psi_{i_k} \circ \varphi^{-k}\circ\alpha)'(0) |} \cdot|w_{i_k}(\overline{\varphi^{-k}\circ\alpha}^1)|\\
        &=\frac{|(\psi_{i_0}\circ\alpha)'(0)|}{|(\psi_{i_k} \circ \varphi^{-k}\circ\alpha)'(0) |}  \|\overline{\varphi^{-k}\circ\alpha}^1 \|^2 \cdot \left|w_{i_k} \left(\frac{1}{\|\overline{\varphi^{-k}\circ\alpha}^1 \|}\cdot\overline{\varphi^{-k}\circ\alpha}^1 \right) \right|\\
        &=\frac{|(\psi_{i_0}\circ\alpha)'(0)|}{|(\psi_{i_k} \circ \left((\varphi^{-k}\circ\alpha)(\frac{1}{\|\overline{\varphi^{-k}\circ\alpha}^1 \|}\cdot) \right))'(0) |}  \|\overline{\varphi^{-k}\circ\alpha}^1 \| \cdot \left|w_{i_k} \left(\frac{1}{\|\overline{\varphi^{-k}\circ\alpha}^1 \|}\cdot\overline{\varphi^{-k}\circ\alpha}^1 \right) \right|\\
        &\leq \frac{C_0}{C'_1}\cdot C_2 \cdot \mu^{-k} \|\overline{\alpha}^1\| \leq \frac{C_0}{C'_1}\cdot C_2 \cdot C_3 \cdot \mu^{-k}.
    \end{align*}
    Therefore, $(w_{i_0}^{(k)})_{k\in \mathbb N}$ converges uniformly to $w_{i_0}$ on $V_{i_0}$. Also, each $w_i^{(k)}$ for $k\in \mathbb N$ is holomorphic on $\interior{V_{i_0}}$, so $w_{i_0}$ is holomorphic on $\interior{V_{i_0}}$. Since $V_{i_0} \subset L\cap U_{i_0}$ is arbitrary, $\restr{w_{i_0}}{(L\cap U_{i_0})_1}$ is holomorphic. Since $i_0$ is arbitrary, $\restr{\sigma}{L_1}:L_1 \to L_2$ is holomorphic, which proves $(ii)$ of Theorem \ref{thm:compaffstru}.\\
    It remains to prove that these complex affine structures are complete. 
    We will need the following corollary:
\begin{corollary}
\label{cor:contaffstru}
    Let $M$ a smooth compact manifold and $(\varphi^t)_{t\in \mathbb R}$ a transversely holomorphic Anosov flow on $M$. Suppose the unstable distribution $E^{u}$ is of complex dimension $1$. \\
    Then the strong unstable foliation is a continuous foliation with complex affine leaves.
    \end{corollary}
\begin{proof}
        The proof is identical to that of Corollary \ref{cor:conthololeaves} since we can now use complex affine charts for the strong unstable leaves.
    \end{proof}
\begin{corollary}
    Let $M$ a smooth compact manifold of dimension five and $(\varphi^t)_{t\in \mathbb R}$ a transversely holomorphic Anosov flow on $M$. \\
    Then the strong unstable and strong stable foliations are continuous foliations with complex affine leaves.
    \end{corollary}
    For every strong unstable leaf $L$ (which is simply connected), there exists a local diffeomorphism $D: L \to \mathbb C$ such that the above complex affine structure on $L$ is the pull back by $D$ of the natural complex affine structure of $\mathbb C$. We want to show that for every strong unstable leaf, $D$ is a diffeomorphism. 
    \begin{lemma}
    \leavevmode
        \begin{enumerate} 
            \item There exists $\epsilon_0>0$ such that for every strong unstable leaf $L$, for every developing map $D$ defining the above complex affine structure of $L$, for every distinct $x,y \in L$ satisfying $d(x,y)< \epsilon_0$, then $D(x) \neq D(y)$.
            \item There exists $\delta_0>0$ such that for every point $x\in M$, for every developing map $D$ defining the above complex affine structure of $\mathcal{F}^{u}_x$ satisfying $\|d_xD\|=1$, then $\text{Im}(D)\supset B(0,\delta_0)$.
        \end{enumerate}
    \end{lemma}
    \begin{proof}
         $(i):$ Consider a $C^0$ atlas of complex affine leaves $(U_i, \psi_i)_i$ for the strong unstable foliation given by Corollary \ref{cor:contaffstru}, which is finite (by compactness of $M$). By the definition and notations of the developing map, since the restriction of $\psi_i$ to each plaque is an affine chart for the corresponding strong unstable leaf, we define on $U_i$ a developing map for every strong unstable leaf $L$ intersecting $U_i$ by $D_L=\restr{\psi_i}{L}$.
        By compactness of $M$, there exists $\eta>0$ such that if two points in $M$ are of distance less than $\eta$, then they belong to the same foliation chart $U_i$.
        Therefore if $x,y \in L$ are distinct and of distance less than $\eta>0$, they belong to the same foliation chart given by the above, say $U_0$. Since $\psi_0$ is injective on $U_0$, necessarily $D_L(x)\neq D_L(y)$. Since any other developing map of $L$ differs by post composition by an invertible affine map, the same is true for any other developing map.\\
        $(ii):$ In the same spirit as $(i)$, we define on $U_i$ a developing map for every strong unstable leaf $L$ intersecting $U_i$ by $D_L=\restr{\psi_i}{L}$.
        The condition $\|d_xD_L\|=1$ for $x \in L\cap U_i$ is independent of $x\in U_i$ and $L$ strong unstable leaf intersecting $U_i$, and two such developing maps differ by post-composition by a rotation independent of $x \in U_i$ and $L$.
        Therefore, for each $i$, there exists $\delta_i>0$ such that for every point $x\in U_i$, for every developing map $D_{\mathcal{F}^{u}_x}$ of $\mathcal{F}^{u}_x$ satisfying $\|d_xD_{\mathcal{F}^{u}_x}\|=1$, then $\text{Im}(D_{\mathcal{F}^{u}_x})\supset B(0,\delta_i)$.
        Since there is only a finite number of charts, $\delta_*=\min_i\delta_i >0$ suits.
    \end{proof}
    These two results are mostly useful because of the following fact:
    \begin{lemma}
        Let $t\in \mathbb R$, $x \in M$ and $D$ a developing map defining the above complex affine structure of $\mathcal{F}^{u}_x$. \\
        Then $D\circ \varphi^t$ is a developing map defining the above complex affine structure of $\varphi^{-t}(\mathcal{F}^{u}_x)= \mathcal{F}^{u}_{\varphi^{-t}(x)}.$
    \end{lemma}
    \begin{proof}
        This result comes directly from the very definition of the developing map and the fact that each $\varphi^t$ is affine with respect to the above complex affine structures, so that if $(V_\alpha, l_\alpha)_\alpha$ is an atlas of complex affine structures for $\mathcal{F}^{u}_x$, then $(\varphi^{-t}(V_\alpha), l_\alpha\circ \varphi^t)_\alpha$ is an atlas of complex affine structures for $\mathcal{F}^{u}_{\varphi^{-t}(x)}$ with the same transition maps.
    \end{proof}     
    With that in mind, we conclude the proof of Theorem \ref{thm:compaffstru}.
    Let $L$ a strong unstable leaf and fix a developing map $D$ on $L$ defining the above complex affine structure. We show that $D$ is injective on $L$. 
    Let $x,y \in L$. Since $d(\varphi^{-k}(x), \varphi^{-k}(y))\xrightarrow[k \to +\infty]{} 0 $, there exists $k_0 \in \mathbb N$ such that $d(\varphi^{-k_0}(x), \varphi^{-k_0}(y)) < \epsilon_0$. By the previous claims, $(D\circ \varphi^{k_0})(\varphi^{-k_0}(x)) \neq (D\circ \varphi^{k_0})(\varphi^{-k_0}(y))$, i.e $D(x) \neq D(y)$.\\
    Let $x \in M$ and $D$ a developing map defining the above complex affine structure of $\mathcal{F}^{u}_x$ such that $D(x)=0$. Let $k \in \mathbb N$ and $g_k \in \mathbb C^*$ such that $|g_k|= \frac{1}{\|d_{\varphi^{-k}(x)}(D\circ \varphi^k)\|}$. By the previous claims, since $g_k \circ D \circ \varphi^k$ is a developing map defining the above complex affine structure of $\mathcal{F}^{u}_{\varphi^{-k}(x)}$, then 
    \[\text{Im}(g_k \circ D \circ \varphi^k) \supset B \left(0, \delta_0 \right).\]
    As a result, for every $k \in \mathbb N,\quad \text{Im}(D) \supset B\left(0, \frac{\delta_0}{g_k} \right),$
     which concludes the proof since $\frac{1}{g_k} \xrightarrow[k \to +\infty]{} +\infty$.
         \end{proof}

\subsection{Transverse projective structure for weak foliations}
\label{subsec:6.2}
Now that we have proven that the strong unstable leaves are complex affine manifolds in case the strong unstable distribution is of complex dimension one, a natural question that arises is wether the weak stable foliation is transversely affine, or at least transversely projective. 
We show that this is true in case the flow is topologically transitive. The main result is:
\begin{theorem}
\label{thm:FsTrProj}
    Let $(\varphi^t)_{t\in \mathbb R}$ a transversely holomorphic Anosov flow on a smooth compact manifold $M$. Suppose the strong unstable distribution $E^{u}$ is of complex dimension $1$. Suppose moreover that $(\varphi^t)_{t\in \mathbb R}$ is transitive.\\
    Then the weak stable foliation $\mathcal{F}^{ws}$ is transversely projective.
\end{theorem}
    As we will see, this hypothesis of transitivity is only used to ensure the existence of a Markov partition of arbitrarily small size for the flow $(\varphi^t)_{t\in \mathbb R}$ (see \cite{ratner_markov_1973}). It is not needed in the case of a holomorphic Anosov diffeomorphism on a compact manifold (see \cite{ghys_holomorphic_1995}).\\
    We recall that a flow $(\varphi^t)_{t\in \mathbb R}$ on $M$ is called (topologically) \textit{transitive} if it has an orbit dense in $M$. This is the same as saying: for every non-empty open sets $U$ and $V$, there exists $t>0$ such that $\varphi^{-t}(U)\cap V \neq \emptyset$.\\

Before going into the proof of Theorem \ref{thm:FsTrProj}, we give some reminders on diffeomorphims defined on complex projective manifolds which we will need for our study (see \cite{ghys_rigidite_1993}, \cite{ghys_deformations_1992}, \cite{ghys_holomorphic_1995}).
Additionally, we will discuss the notion of Markov partition for a transitive flow, as was studied in \cite{ratner_markov_1973}.

    A complex projective manifold $L$ of complex dimension one is a $(G,X)$-manifold, where $X=\mathbf{P}^1(\mathbb C)$ and $G=\mathrm{PSL}(2,\mathbb C)$, that is there exist local charts covering $L$ taking values in $\mathbf{P}^1(\mathbb C)$ such that the transition maps are homographies (elements of $\mathrm{PSL}(2,\mathbb C)$). For example, a complex affine manifold of complex dimension one is a complex projective manifold of complex dimension one.\\
    Let $L$ and $L'$ complex projective manifolds of dimension one and $f: L\to L'$ a holomorphic diffeomorphism. We can consider the following quadratic differential, defined in local projective coordinates, by the Schwarzian derivative:
    \[s(f)=\left ( \frac{f'''}{f'} -\frac{3}{2}\left (\frac{f''}{f'} \right )^2 \right ) dz\otimes dz.\]
    It is well defined as it does not depend on the local projective chart used. Moreover, it satisfies the following properties:
    \begin{enumerate} 
        \item $s(f)$ is identically zero if and only if $f$ is a projective diffeomorphism;
        \item If $f:L \to M$ and $g:M \to N$ are holomorphic diffeomorphisms between complex projective manifolds of complex dimension one, then
        \[s(g\circ f) = f^*s(g) + s(f);\]
        \item In particular, if $g_1:L \to L'$ and $g_2:M \to M'$ are projective diffeomorphisms between complex projective manifolds of complex dimension one, then
         \[s(g_2\circ f \circ g_1^{-1}) = g_1^*s(f).\]
    \end{enumerate}
In order to prove that the weak stable foliation $\mathcal{F}^s$ is transversely holomorphic with respect to the complex affine structures constructed before on the   unstable leaves, we follow the ideas mentioned in \cite{ghys_holomorphic_1995}. 
we begin with the same remark that we did in Theorem \ref{thm:Fstrholo}: we take two points $x$ and $y$ on the same weak stable leaf, such that the distance between $x$ and $y$ in the leaf is less than $\delta$. As mentioned before, if $W^{u}_x$ and $W^{u}_y$ are small (affine) open sets in $\mathcal{F}^{u}_x$ and $\mathcal{F}^{u}_y$ respectively, the holonomy of any path with respect to $W^{u}_x$ and $W^{u}_y$ does not depend on this very path. We call $h$ such holonomy.
In the general case, if $x$ and $y$ belong to the same weak stable leaf and are not necessarily close enough, we conclude in the same manner as before, that is we take intermediate points in the weak stable leaf of $x$ whose successive distances are bounded by $\delta$. Then, the holonomy of any path with respect to small open (affine) subsets of the   unstable leaves of $x$ and $y$ is the composition of projective maps and is therefore projective as well.

With that in mind, let $x_0 \in M$ and consider $x \in \mathcal{F}^{s}_{x_0}$. Consider small open neighborhoods $W^{u}_{x_0}$ in $\mathcal{F}^{u}_{x_0}$, and $W^{u}_{x}$ in $\mathcal{F}^{u}_{x}$, of $x_0$ and $x$ respectively, viewed as transversals to the weak stable leaf $\mathcal{F}^{s}_{x_0}$ at $x_0$ and $x$ respectively. Denote by $h_\alpha$ the $\mathcal{F}^s$-holonomy of the path $\alpha$ from $x_0$ to $x$ in $\mathcal{F}^{s}_{x_0}$ with respect to $W^{u}_{x_0}$ and $W^{u}_{x}$ respectively. 
As mentioned earlier, since the   unstable leaves are complex affine (thus projective) manifolds of complex dimension $1$, we can define the Schwarzian derivative of $h_\alpha$ which is a quadratic differential in $W^{u}_{x_0}$. Evaluated at $x_0$, we get a complex quadratic form $q_\alpha$ on the complex vector space $E_{x_0}^{u}$. 
We want to show that $h_\alpha$ is projective, which is the same as proving that it Schwarzian derivative vanishes identically. It is enough to prove that result for $x$ at a distance at most $\delta$ to $x_0$ since we can always take a finite family of intermediate points in order to compute the holonomy of $\alpha$.
It is enough to prove that $q_{x_0,x}=0$ by the following argument: if $z\in W^{u}_{x_0}$, then the $\mathcal{F}^s$-holonomy between $z$ and $h_{x_0,x}(z)$ with respect to small transversals $V\subset W^{u}_{x_0}$ and $V' \subset h_{x_0,x}(W^{u}_{x_0})$ is just the restriction to $V$ of $h_{x_0,x}$. Therefore, the Schwarzian derivative of $h_{x_0,x}$ evaluated at $z$ is just the Schwarzian derivative of $h_{z,h_{x_0,x}(z)}$ evaluated at $z$, i.e. $q_{z,h_{x_0,x}(z)}$.
As a result, for $v \in E_{x_0}^{u}$ fixed, we want to prove that for any $x\in \mathcal{F}^{s}_{x_0, \delta}$, $q_{x_0,x}(v)=0$, i.e. the continuous map
\[q_{x_0, \cdot}(v):\left \{\begin{array}{c}
     \mathcal{F}^{s}_{x_0, \delta} \to \mathbb C  \\
     x \mapsto q_{x_0,x}(v)
\end{array} \right.\]
is equal to zero. Therefore, since $\mathcal{F}^{s}_{x_0, \delta}$ is connected, we only need to show that this map is locally constant.

We will need the notion of Markov partition (\cite{ratner_markov_1973}).
Let $(\varphi^t)_{t\in \mathbb R}$ an Anosov flow on a smooth compact manifold $M$.\\
    A set $A$ included in one of the leaves of the strong, or weak, unstable, or stable, foliations is said to be \textit{admissible} if, with respect to the leaf metric, its interior is non-empty, the closure of its interior is equal to $A$, the measure of the boundary $\partial A$ is equal to $0$.\\
    Let $x_0 \in M$, $C \subset \mathcal{F}^{s}_{x_0, \epsilon}$ and $D \subset \mathcal{F}^{u}_{x_0, \epsilon}$ two admissible subsets.
    The set 
    \[P = [C,D]:= \left \{ [y,z], \; y \in C, \, z\in D  \right \}\]
    is called a \textit{parallelogram}. The \textit{sizes of $P$} are the diameters of $C$ and $D$. \\
    For $x=[y,z] \in P$, we note $$C(x):=[C,z] \subset \mathcal{F}^{ws}_{z, \epsilon} \quad \text{ and } \quad D(x):=[y,D] \subset \mathcal{F}^{u}_{y, \epsilon}\subset \mathcal{F}^{u}_{x, \epsilon} .$$
    Note also
    \begin{align*}
        \partial_cP=\bigcup_{z \in \partial D}[C,z],\quad 
        \partial_eP=\bigcup_{y \in \partial C}[y,D],\quad 
        \partial P=\partial_cP \cup \partial_eP.
    \end{align*}
    Let $P=[C_P, D_P]$ a parallelogram, and $\tau$ a continuous positive function on $C_P$. The set 
    \[V=\bigcup_{x\in D_P} \bigcup_{0 \leq t\leq \tau(x)} \varphi^t(C_P(x))\] is called a \textit{parallelepiped}. $P$ is called the \textit{lower face} of the parallelepiped $V$. The \textit{sizes of $V$} are the sizes of its lower face and the maximum of $\tau$.
    A finite family $(V_i)_{1\leq i \leq k}$ of parallelepipeds satisfying
    \[M=\bigcup_{i=1}^k V_i \qquad \text{and} \qquad \forall i,j \in \llbracket 1;k \rrbracket, \, i\neq j, \quad V_i\cap V_j = \partial V_i \cap \partial V_j\]
    is called a \textit{partition of $M$ by parallelepipeds}.\\
    A family $(P_i)_{i}$ of parallelepipeds is \textit{complete} if the positive orbit of every point $x \in M$, that is $(\varphi^t(x))_{t>0}$, intersect at least one parallelogram of the family. Since $M$ is supposed to be compact, then there exists a finite complete family of disjoint parallelograms for the Anosov flow $(\varphi^t)_{t\in \mathbb R}$.\\
    Let $(P_i)_{1\leq i \leq k}$ such a finite complete family of parallelograms ($k \in \mathbb N^*$), and $\Tau:= \bigsqcup_{i=1
    }^k P_i.$\\
    For $x\in \Tau$, let $l(x):= \min \left \{t>0, \, \varphi^t(x)\in \Tau \right \} $.
    Let also $f: \left \{
        \begin{array}{ccl}
        \Tau & \to & \Tau \\
        x & \mapsto & \varphi^{l(x)}(x)
    \end{array} \right. .$\\
    Since our flow is transversely holomorphic, $f$ is holomorphic.\\
    The system $(P_i)_{1\leq i \leq k}$, where $P_i=[C_i, D_i]$, is called \textit{Markovian} for the flow $(\varphi^t)_{t\in \mathbb R}$ if:
    \[\forall x \in \interior{P_i}\cap f^{-1}(\interior{P_j}), \quad 
    \begin{array}{l}
       f(\interior{\overbrace{C_i(x)}}) \subset C_j(f(x)) \\
        f(D_i(x)) \supset \interior{\overbrace{D_i(f(x))}}
    \end{array}. \]
    A partition $(V_i)_{1\leq i \leq k}$ of $M$ by parallelepipeds is a \textit{Markov partition} for the flow $(\varphi^t)_{t\in \mathbb R}$ if the system of its lower faces is Markovian for the flow $(\varphi^t)_{t\in \mathbb R}$. \\
    The main result proved in \cite{ratner_markov_1973} is:
    \begin{customthm}{2.1}[\cite{ratner_markov_1973}]
    \label{thm:rat}
        Let $(\varphi^t)_{t\in \mathbb R}$ a transitive Anosov flow on a smooth compact manifold $M$.
        Then for every $\epsilon>0$, there is a Markov partition of $M$ by parallelepipeds $(V_i)_{1\leq i \leq k}$ whose sizes are less than $\epsilon$ and whose boundaries are of measure $0$.
    \end{customthm}
    From now, for each parallelogram $P_i$, we fix a base point $x_i$ in the interior of $P_i$ and denote by $D_i$ the set $D_{P_i}(x_i)$. Sometimes, we will omit the index $_i$ if we do not care to which parallelogram a point of $\Tau$ belongs.\\
    The main properties of such a Markov partition that will be used are contained in the following proposition:
    \begin{lemma}
    \label{lem:pptesrat}
    \leavevmode
        \begin{enumerate} 
            \item Let $x\in P_i\cap f^{-1}(P_j)$.\\
            If $x\in \partial_c P_i$, then $f(x) \in \partial P_j$. If $x\in \partial_e P_i$, then $f^{-1}(x) \in \partial P_j$.
            \item $f$ is topologically mixing.
            \item For every sufficiently small open set $V$ of $\interior{D_i}$, there exists $r>0$ such that 
            \[\restr{f^{-1}}{V}=\restr{\varphi^{-r}}{V}.\]
        \end{enumerate}
    \end{lemma}
    \begin{proof}
        $(i)$ See Proposition $1.1$ of \cite{ratner_markov_1973}.\\
        $(ii)$ See Theorem $4.2$ of \cite{ratner_markov_1973}.\\
        $(iii)$ Let $x_0$ fixed, in $V$ and $y_0=f^{-1}(x_0)=\varphi^{-r}(x_0) \in \interior{P_j}$ for some $r>0$.
        By the previous point $(i)$, $y_0 \in \interior{\overbrace{D_j(y_0)}}$.
        Combining that with the Markov property, it comes $f^{-1}(V) \subset \interior{\overbrace{D_j(y_0)}}$.
        Therefore, the restriction of $f^{-1}$ to $V$ is just the restriction to $V$ of the $\mathcal{F}^{ws}$-holonomy of the path $\beta^{x_0}_{-r}$ (see Proposition \ref{prop:holoinvfol}) with respect to the transversals $V$ and $\interior{\overbrace{D_j(y_0)}}$, which by Corollary \ref{cor:holofloFsu} is precisely the restriction to $V$ of $\varphi^{-r}$.
    \end{proof}
Now in order to prove that the map $q_{x_0,.}(v)$ is locally constant on $\mathcal{F}_{x_0,\epsilon}^{s}$, we first remark by property $(ii)$ of the reminder that, if $x_0'\in \mathcal{F}^s_{x_0}$, $\alpha$ (respectively, $\alpha'$) is a path in the weak stable leaf of $x_0$ connecting $x_0$ (respectively, $x_0'$) and $x$, then
\[q_{\alpha'}(dh_{\overline{\alpha'}\alpha}(v))=q_{\alpha}(v)+q_{\overline{\alpha}\alpha'}(dh_{\overline{\alpha'}\alpha}(v)).\]
Therefore, if $x,x_0'\in \mathcal{F}^{s}_{x_0,\epsilon}$ and $U$ is a connected open neighborhood of $x$ in $\mathcal{F}_{x_0,\epsilon}^{s}$, then
\[\text{diam}\{q_{x_0,x}(v), \;x \in U\}=\text{diam}\{q_{x_0',x}(dh_{x_0,x'_0}(v)), \;x \in U\}.\]
Now recall that we want to prove that $q_{x_0,x}=0$ for $x\in \mathcal{F}^s_{x_0,\epsilon}$.
We remark that the proof of Theorem \ref{thm:rat} shows that for any lower face $P_i$ of a parallelepiped of the Markov partition, its edges $C_i$ and $D_i$ can be assumed to be connected.
Therefore, we can take $x$ belonging to a parallelepiped $V=\bigcup_{x^*\in D_P} \bigcup_{t=0}^{\tau(x^*)} \varphi^t(C_P(x^*))$ of a Markov partition of size small enough and change our base point $x_0$ to $x^* \in P$ such that $x\in \bigcup_{0 \leq t\leq \tau(x^*)}\varphi^t(C_P(x^*))$. 
As a result, we can assume $x^* \in P,\; x\in \bigcup_{t=0}^{\tau(x^*)}\varphi^t(C_P(x^*))$ and prove that $h_{x^*,x}$ is projective. But again, if we note $x' \in C_P(x^*)$ such that $\varphi^{t_0}(x')=x$, with $0\leq t_0 \leq \tau(x^*)$, since $\phi^{t_0
}$ is affine, by $(ii)$ of the reminder it is sufficient to prove that $h_{x^*,x'}$ is projective, which as we said earlier is the same as proving that $q_{x^*,x'}$ is zero. \\
From now on, every weak stable holonomy will be computed with respect to transversals of the form $D(x) \subset \mathcal{F}^{u}_x$. \\
Therefore, let $\mathcal{D}=\bigsqcup_{i=1}^kD_i$ be the disjoint union of the $D_i$'s. The aim is to prove that for any vector $v$ tangent to $\mathcal{D}$ at some point $x \in D_i\subset \mathcal{ D}$, seen as an element of $T_x\mathcal{F}^{u}$, $\delta(v):=\text{diam}\{q_{x,y}(v), \;y \in C_i(x)\}$
is equal to $0$.
Since the Schwarzian derivative is a complex quadratic differential, the map $\delta$ satisfies, for $v\in T\mathcal{D}$ and $\lambda \in \mathbb C$, $\delta(\lambda\, v)=|\lambda|^2 \delta(v).$
Let, for $p \in \mathcal{D}$ and $z=x+iy$ a holomorphic coordinate of $\mathcal{D}$ around $p$, 
\[\mu(p):=\frac{i}{2}\delta
(\frac{\partial}{\partial z}|_p)d_pz\wedge d_p\overline{z}=\delta(\frac{\partial}{\partial z}|_p)d_px\wedge d_py.\]
If $z'$ is another holomorphic coordinate of $\mathcal{D}$ around $p$,
we know that
$dz'\wedge d\overline{z'}=\left|\frac{\partial z'}{\partial z} \right|^2dz\wedge d\overline{z}.$
As a result, $\mu$ is a well-defined continuous $2$-form of type $(1,1)$ on $\mathcal{D}$. \\
For $x \in \mathcal{D}\cap f^{-1}(\interior{ P_i})$, we define $\phi(x)$ to be the projection of $f(x)$ onto $D_i$ (see also \cite{butterley_open_2020} p.25-28). $\phi$ is a well-defined expanding holomorphic transitive map from $\mathcal{D}\cap f^{-1}(\bigcup_{j=1}^k\interior{ P_j})$, which is an open dense subset of $\mathcal{D}$, to $\mathcal{D}$. \\
The map $\phi$ is not injective, as $\phi(x)=\phi(y)$ if and only if $f(x)$ and $f(y)$ belong to the same parallelogram and $C(f(x))=C(f(y))$. 
However:
\begin{lemma}
\label{lem:Psij}
    For every $i \in \llbracket 1, k \rrbracket$, there exist $N_i \in \llbracket 1, k \rrbracket$ and holomorphic injective contracting maps $\Psi_j^i: \interior {D_i} \to \interior{D_{i_j}}$ for $j\in \llbracket 1, N_i \rrbracket$ such that
    $$\phi \circ \Psi^i_j = I_d.$$
\end{lemma}
\begin{proof}
    Fix $i \in \llbracket 1, k \rrbracket$. 
    Let $f^{-1}(x_i) \in \interior{P_{i_1}}$ for some index $i_1$. Then by the Markov property and point $(i)$ of Lemma \ref{lem:pptesrat}, $f^{-1}\interior{(D_i)} \subset \interior{\overbrace{D_{i_1}(f^{-1}(x_i))}}$.
    By projecting on $D_{i_1}$, we obtain a map $\Psi^i_{i_1}: \interior D_i \to \interior D_{i_1}$ such that: for $x \in \interior D_i$, $\Psi^i_{i_1}(x) = \interior{\overbrace{C_{i_1}(f^{-1}(x))}} \cap D_{i_1}$ so by the Markov property, $f(\Psi^i_{i_1}(x))\in C_i(x)$ which means exactly that $\phi (\Psi^i_{i_1}(x)) = x$.\\
    Now assume $x$ and $y$ both belong to $D_i$ and denote by $N_x, N_y \in \llbracket 1, k \rrbracket$ the number of elements in $\phi^{-1}\{x\}$ and $\phi^{-1}\{y\}$ respectively.
    Fix $j \in \llbracket 1, N_x \rrbracket$ and consider $x'=f(\Psi^i_j(x))\in \interior{C_i(x)}.$
    Consider also $y'= D_i(x') \cap C_i(y)$ the projection of $y$ on $D_i(x')$. 
    Then $y' \in \interior{D_i(x')}$ so $f^{-1}(y')\in \interior{ \overbrace{D_j(\Psi^i_j(x))}}= \interior D_j$ satisfies $\phi(f^{-1}(y'))=y$.
    Since this works for every $j \in  \llbracket 1, N_x \rrbracket$, we have proved that $N_x\leq N_y$.
    By symmetry, $N_y\leq N_x$ i.e. $N_i:=N_x$ does not depend on $x \in \interior D_i$.\\
    The fact that $\Psi^i_j$ is injective and contracting is immediate by construction.
    Also, as the above shows, we have a formula for $\Psi^i_j(y)$ for $y \in \interior D_i$ from the value of $\Psi^i_j(x)$, where $x \in \interior D_i$, that is 
    $$\Psi^i_j(y)=f^{-1}\circ h_{x,f(\Psi^i_j(x))}.$$
    Since the weak stable foliation is transversely holomorphic (by Theorem \ref{thm:Fstrholo}), $h_{x,f(\Psi^i_j(x))}$ is holomorphic and thus $\Psi^i_j$ also by the above equality.
\end{proof}
We can take the Markov partition small enough so that for every $i \in \llbracket 1, k \rrbracket$ two different $j,j' \in \llbracket1,N_i \rrbracket$ are such that $i_j \neq i_{j'}$.
\begin{lemma}
    For $x \in \interior{D_i}$, $f^{-1}(C_i(x))=\bigcup_{j=1}^{N_i}C(\Psi^i_j(x))$
\end{lemma}
\begin{proof}
    Assume $y \in P_k$ is such that $f(y)\in C_i(x)$. Then $y \in \interior{\overbrace{D_k(y)}}$ because $x \in \interior{D_i}$. 
    Consider $y'$ the projection of $y$ on $D_k$. Then $y' \in \interior{D_k}$.
    Consider a sequence of elements $(y_n)_n$ in $\interior{\overbrace{C_k(y)}} \cap f^{-1}(\interior P_i)$ which converges to $y$.
    Then by the Markov property, $f(y') \in C_i(f(y_k))$. But since $l$ is constant on $C_k(y)$ and the flow is continuous on $M$, $(f(y_k))_k$ converges to $f(y)$, so $f(y') \in C_i(f(y))=C_i(x)$.
    This proves that $y'=\Psi^i_k(x)$, i.e. $y\in C(\Psi^i_k(x))$.\\
    Conversely, assume $y \in C(\Psi^i_j(x))$ for some $j \in \llbracket1, N_i \rrbracket$.
    Consider a sequence of elements $(y_n)_n$ in $\interior{\overbrace{C(y)}} \cap f^{-1}(\interior P_i)$ which converges to $y$.
    Then $f(y)$ is the limit of $(f(y_n))_n$ since $l$ is constant on $C(y)$.
    But we know by the Markov property that for every $n\in \mathbb N$, $f(y_n)\in C(\Psi^i_j(x))=C_i(x)$. 
    Therefore, $f(y) \in C_i(x)$ since the latter is closed, which concludes.
\end{proof}
In order to enlighten notations, we will omit to write the path with respect to which we compute the weak stable holonomy. If it is computed between $x$ and $y\in C(x)$, then we consider any path $\alpha$ in $C(x)$ as $h$ does not depend on the path between $x$ and $y$. 
If it is computed between $x$ and $f^{-1}(x)=\varphi^{-m}(x)$, then we consider the natural path $\beta^x_{-m}$ (see Proposition \ref{prop:holoinvfol}). 
If it is computed between $f^{-1}(x)=\varphi^{-m}(x)$ and $x$, then we consider the natural path $\overline{\beta^x_{-m}}$.
Therefore, it will be natural from the context to which path we refer to. For example, if $x \in \interior{D_i}$ and $y\in C_i(x)$, then $h_{f^{-1}(x), f^{-1}(y)}$ is the weak stable holonomy with respect to the path $\overline{\beta^x_{-m_x}} \alpha \beta^y_{-m_y}$, where $\alpha$ is a fixed path in $C_i(x)$ between $x$ and $y$.

\begin{lemma}
\label{lem:chgtptsdsC}
    Let $x \in \interior D_i$, $j\in \llbracket1, N_i \rrbracket$, $v \in T_{f^{-1}(x)}\mathcal{F}^{u}$, $z \in C_{j}(\Psi^i_j(x))$. Then:
    \[\emph{diam}\{q_{\Psi^i_j(x),z}(dh_{f^{-1}(x), \Psi^i_j(x)}(v)), \;z \in C_{j}(\Psi^i_j(x))\}=\emph{diam}\{q_{f^{-1}(x),z}(v), \;z \in C_{j}(\Psi^i_j(x))\}.\]
\end{lemma}
\begin{proof}
    By the above and the previous remark, we know that 
    \[q_{\Psi^i_j(x),z}(dh_{f^{-1}(x), \Psi^i_j(x)}(v))=q_{f^{-1}(x),z}(v)+q_{\Psi^i_j(x),f^{-1}(x)}(dh_{f^{-1}(x),\Psi^i_j(x)}(v)).\]
    The result is proven since the second term of the right hand side of this equality does not depend on $z \in C_{j}(\Psi^i_j(x))$.
\end{proof}

\begin{lemma}
\label{lem:holof-1(x)}
    Let $x \in \interior{D_i}$ and $y\in C_i(x)$.\\
    Then, if $V$ is a sufficiently small open neighborhood of $x$ in $\interior D_i$,
    \[\restr{f^{-1} \circ h_{x,y}}{V}=\restr{h_{f^{-1}(x), f^{-1}(y)} \circ f^{-1}}{V}.\]
    Therefore, for every $v \in T_x\mathcal{D}$,
    \[q_{x,y}(v)=q_{f^{-1}(x),f^{-1}(y)}(d_xf^{-1}(v)).\]
\end{lemma}
\begin{proof}
    The first claim is a simple verification by taking $V$ small, since the restriction of $f^{-1}$ to such $V$ is a weak stable holonomy between $V$ and $\interior{\overbrace{D_j(f^{-1}(x))}}$ (see also point $(iii)$ of Lemma \ref{lem:pptesrat}).
    The second claim comes therefore from $(ii)$ of the reminder and the fact that $\varphi^t$ acts affinely on strong unstable leaves by Theorem \ref{thm:compaffstru}.
\end{proof}
With all that in mind, we are ready to prove that $\delta$ vanishes identically. 
\begin{lemma}
    The measure $\mu$ on $\mathcal{D}$ is invariant by $\phi$.
\end{lemma}
\begin{proof}
We prove that for every Borel set $\mathcal{B} \subset \mathcal{D}$,
$\mu(\mathcal{B}) \leq \mu(\phi^{-1}(\mathcal{B})).$
This will imply the result since
\begin{align*}
    \mu(\mathcal{D})=\mu(\mathcal{B}) + \mu(\mathcal{D} \setminus\mathcal{B}) \leq \mu(\phi^{-1}(\mathcal{B})) + \mu(\phi^{-1}(\mathcal{D} \setminus\mathcal{B}))=\mu(\phi^{-1}(\mathcal{D})) \leq \mu(\mathcal{D}),
\end{align*}
Since $\partial\mathcal{D}$ is of Lebesgue measure zero, and because $\mu$ is absolutely continuous with respect to the Lebesgue measure, we prove that for every Borel set $\mathcal{B} \subset \interior{\mathcal{D}}$,
$\mu(\mathcal{B}) \leq \mu(\phi^{-1}(\mathcal{B})).$
By a result on Borel measures (the outer regular property, see \cite{pages_analyse_2018}), it is enough to verify 
$\mu(\Omega) \leq \mu(\phi^{-1}(\Omega))$
for every open set $\Omega \subset \interior{\mathcal{D}}$ and thus for every open set $\Omega \subset \interior{D_i}$ for every $i \in \llbracket 1, k \rrbracket$. We fix therefore such $i$ and such open set $\Omega \subset \interior{D_i}$.
By definition,
$\phi^{-1}(\Omega)= \bigsqcup_{j=1}^{N_i} \Psi^i_j(\Omega).$
Thus, we prove that $\mu(\Omega) \leq \sum_{j=1}^{N_i} \mu(\Psi^i_j(\Omega)).$
Let $v$ a vector tangent to $D_i$ at $x\in \interior{D_i}$. Recall that the diameter of a connected finite union of sets is at most the sum of the diameters of these sets.
Therefore,
\begin{align*}
    \delta(v)&=\text{diam}\{q_{x,y}(v), \;y \in C_i(x)\}
    =\text{diam}\{q_{f^{-1}(x),z}(\restr{df^{-1}}{E^{u}_x}(v)), \;z \in f^{-1}(C_i(x))\}\\
    &\leq \sum_{j=1}^{N_i}\text{diam}\{q_{f^{-1}(x),z}(\restr{df^{-1}}{E^{u}_x}(v)), \;z \in C_j(\Psi^i_j(x))\}\\
    &=\sum_{j=1}^{N_i}\text{diam}\{q_{\Psi^i_j(x),z}(dh_{f^{-1}(x),\Psi^i_j(x)}(\restr{df^{-1}}{E^{u}_x}(v))), \;z \in C_j(\Psi^i_j(x))\}\\
    &=\sum_{j=1}^{N_i}\text{diam}\{q_{\Psi^i_j(x),z}(d\Psi_j^i(v)), \;z \in C_j(\Psi^i_j(x))\}= \sum_{j=1}^{N_i}\delta(d\Psi_j^i(v)).
\end{align*}
The penultimate equality comes from Lemma \ref{lem:holof-1(x)} and the remark at the end of the proof of Lemma \ref{lem:Psij} which gives a formula for $\Psi^i_j$ in a neighborhood of $x$.
Moreover, 
\[\delta(d\Psi_j^i(\frac{\partial}{\partial z}|_p))=\delta \left(\frac{\partial\Psi^i_j}{\partial z}|_p \cdot \frac{\partial}{\partial z}|_{\Psi^i_j(p)} \right) = \left|\frac{\partial\Psi^i_j}{\partial z}|_p \right |^2 \delta(\frac{\partial}{\partial z}|_{\Psi^i_j(p)}) = \det(d_p\Psi^i_j)\, \delta(\frac{\partial}{\partial z}|_{\Psi^i_j(p)})\]
since $\Psi^i_j$ is holomorphic.
Now, we compute:
\begin{align*}
    \mu(\Omega)=\int_\Omega\delta(\frac{\partial}{\partial z}|_p)\,dxdy &\leq \sum_{j=1}^{N_i}\int_\Omega\delta(d\Psi_j^i(\frac{\partial}{\partial z}|_p))\,dxdy=\sum_{j=1}^{N_i}\int_\Omega\det(d_p\Psi^i_j)\, \delta(\frac{\partial}{\partial z}|_{\Psi^i_j(p)})) \,dxdy\\
    &= \sum_{j=1}^{N_i}\int_{\Psi^i_j(\Omega)}\delta(\frac{\partial}{\partial z}|_p)\, dxdy =\sum_{j=1}^{N_i} \mu(\Psi^i_j(\Omega))
\end{align*}
by the change of variable formula since each $\Psi^i_j$ is holomorphic thus orientation preserving.
\end{proof}
\begin{lemma}
    $\delta$ vanishes identically.
\end{lemma}
\begin{proof}
    Otherwise, $\delta$ is positive on an non-empty open set and is thus positive everywhere on $\mathcal{D}$ by transitivity of $\phi$ and $\phi$-invariance of $\mu$.\\
    Fix $x \in \interior{D_i}$ and call $v_0=\frac{\partial}{\partial z}|_x\in E^{u}_x \setminus 0$. 
    All previous inequalities are equalities,
    i.e.
    \[\text{diam}\{q_{x,y}(v_0), \;y \in C_i(x)\} =\sum_{x'\in \phi^{-1}(x)}\text{diam}\{q_{f^{-1}(x),z}\left(\restr{df^{-1}}{E^{u}_x}
    (v_0) \right), \;z \in C(x')\}.\]
    By iterating, it comes for every $N \in \mathbb N$:
    \begin{align*}
        \text{diam}\{q_{x,y}(v_0), \;y \in C_i(x)\} = \sum_{x'\in \phi^{-N}(x)}\text{diam}\{q_{f^{-N}(x),z}\left((\restr{df^{-1}}{E^{u}_x}
    )^N(v_0) \right), \;z \in C(x')\}.
    \end{align*}
    We now use the following two facts in order to come to a contradiction and to prove therefore that $\delta$ indeed vanishes identically so that the proof of Theorem \ref{thm:FsTrProj} is complete (see Facts 1 and 2 of \cite{ghys_holomorphic_1995} p.596). We prove them in the appendix.
    \begin{lemma}
    \label{lem:fact1}
        Let $K$ a subset of $\mathbb R^p$ $(p\geq 1)$ which is the union of finitely many non-empty distinct compact sets $(K_j)_{1\leq j \leq q}$ $(q \geq 3)$, each non reduced to a point, satisfying for $1 \leq i \leq q-1$,
        \[K_i\cap K_{i+1} \neq \emptyset\]
        and such that
        \[\text{\emph{diam}}(K)=\sum_{j=1}^q \text{\emph{diam}}(K_j).\]
        Then no point of $K$ belongs to three distinct $K_j$.
    \end{lemma}
    
    \begin{lemma}
    \label{lem:fact2}
        Let $p\geq 1$ and $C$ a compact subset of $\mathbb R^p$ such that for every $\alpha>0$, there exists $N_\alpha \in \mathbb N^*$ and $N_\alpha$ compact subsets $C_i^\alpha$ of $\mathbb R^p$ covering $C$ whose diameters are less than $\alpha$ and satisfying: for every $\alpha'< \alpha$, and $i\in \llbracket1, N_{\alpha'} \rrbracket$, there exists $j\in \llbracket1, N_{\alpha} \rrbracket$ such that $C_i^{\alpha'} \subset C_j^\alpha$. \\
        Then there exists $\alpha >0$ and a point of $C$  belonging to $p+1$ distinct $C_i^\alpha$. 
    \end{lemma}
    We conclude the proof of $\delta=0$.
    Call $K:=\{q_{x,y}(v_0), \;y \in C_i(x)\}$
    and for every $N \in \mathbb N,\; x'\in \phi^{-N}(x) $, 
    \[K^{(N,x')}:=\{q_{f^{-N}(x),z}\left((\restr{df^{-1}}{E^{u}_x}
    )^N(v_0) \right), \;z \in C(x')\}.\]
   They are compact and connected. Therefore, we can arrange the indices so that the sets $(K^{(N,x')})_{N,x'}$ satisfy the intersection assumption of Fact (1).
    Moreover, each $K^{(N,x')}$ (which is non-empty) is non reduced to a point because its diameter is in the image of $\delta$ and $\delta$ is positive everywhere on $\mathcal{D}$ by a previous claim. 
    More precisely, the diameter of $K^{(N,x')}$ is the image by $\delta$ of the image of $v_0$ by the differential of a composition of $N$ maps of the form $\Psi^i_j$. 
    Therefore, since $\delta$ is continuous on $\overline{B^1}(\mathcal{D}):= T\mathcal{D} \cap \overline{B^1}(\mathcal{F}^{u})$ where $\overline{B^1}(\mathcal{F}^{u})$ is the set of tangent vectors to the strong unstable foliation whose norm is less or equal to $1$, and because $\overline{B^1}(\mathcal{D})$ is compact, $\delta$ is uniformly continuous on $\overline{B^1}(\mathcal{D})$ so for any $\alpha >0$, there exists $N_\alpha \in \mathbb N$ large enough so that for any $N \in \mathbb N$ and $x' \in \phi^{-N}(x)$, the diameter of $K^{(N,x')}$ does not exceed $\alpha$.
    By Fact (2), there exists $N^* \in \mathbb N$ large enough and a point of $K$ belonging to three distinct $K^{(N^*,x')}$. 
    This is absurd by Fact (1).
The proof of Theorem \ref{thm:FsTrProj} is now complete.
\end{proof}
\begin{corollary}
\label{cor:FsFuTrProj}
     Let $(\varphi^t)_{t\in \mathbb R}$ a transversely holomorphic Anosov flow on a five dimensional smooth compact manifold $M$. Suppose it is topologically transitive.\\
     Then the weak stable and weak unstable foliations $\mathcal{F}^{ws}$ ans $\mathcal{F}^{wu}$ are transversely projective.
\end{corollary}
\begin{corollary}
\label{cor:flotrproj}
    Let $(\varphi^t)_{t\in \mathbb R}$ a transversely holomorphic Anosov flow on a five dimensional smooth compact manifold $M$. Suppose it is topologically transitive.\\
    Then $(\varphi^t)_{t\in \mathbb R}$ admits a transverse $(\mathrm{PSL}(2, \mathbb C) \times \mathrm{PSL}(2, \mathbb C), \mathbf{P}^1(\mathbb C) \times \mathbf{P}^1(\mathbb C))$-structure.
\end{corollary}
\begin{proof}
    By the previous corollary, the weak stable and the weak unstable foliations each admits a transverse $(\mathrm{PSL}(2, \mathbb C),\mathbf{P}^1(\mathbb C))$-structure. 
    We proceed as in Corollary \ref{cor:flotrholocarac} since both foliations are smooth:
    there exist an atlas of submersions $(U_i, f_i^s)_{i\in I}$ onto $\mathbf{P}^1(\mathbb C)$ and elements $(\gamma^s_{ij})_{ij}$ of $\mathrm{PSL}(2, \mathbb C)$ defining the transverse $(\mathrm{PSL}(2, \mathbb C),\mathbf{P}^1(\mathbb C))$-structure of the weak stable foliation. There also exist an atlas of submersions $(V_i, f_i^u)_{i\in J}$ onto $\mathbf{P}^1(\mathbb C)$ and elements $(\gamma^u_{ij})_{ij}$ of $\mathrm{PSL}(2, \mathbb C)$ defining the transverse $(\mathrm{PSL}(2, \mathbb C),\mathbf{P}^1(\mathbb C))$-structure of the weak unstable foliation.
    By taking a refinement if necessary, we can assume $I=J$ and $V_i=U_i$ for every $i\in I$.
    Again we define an atlas of submersions $(U_i, f_i)$ onto $\mathbf{P}^1(\mathbb C) \times \mathbf{P}^1(\mathbb C)$ and elements $\gamma_{ij}$ of $\mathrm{PSL}(2, \mathbb C) \times \mathrm{PSL}(2, \mathbb C)$ by:
    \begin{align*}
        f_i &= (f_i^u, f_i^s)\\
        \gamma_{ij}&=(\gamma_{ij}^u \, ,\, \gamma_{ij}^s).
    \end{align*}
    This atlas is compatible with the one defining the orbit foliation of $(\varphi^t)_{t\in \mathbb R}$ so the proof is complete.
    \end{proof}
\section{Classification of transversely holomorphic transitive Anosov flows on five-dimensional compact manifolds}
\label{sec:7}
The main result of this article is:
\begin{theorem}
\label{thm:classif}
    Let $(\varphi^t)_{t \in \mathbb R}$ a transversely holomorphic Anosov flow on a smooth compact manifold $M$ of dimension five. Suppose $(\varphi^t)_{t \in \mathbb R}$ topologically transitive.\\
    Then $(\varphi^t)$ is either $C^\infty$-orbit equivalent to the suspension of a hyperbolic automorphism of a complex torus, or, up to finite covers, $C^\infty$-orbit equivalent to the geodesic flow on the unit tangent bundle of a compact hyperbolic three-dimensional manifold.
\end{theorem}

So from now on, we suppose $(\varphi^t)_{t\in \mathbb R}$ is a transversely holomorphic Anosov flow on a five dimensional smooth compact manifold $M$, which is in addition topologically transitive.
We cannot conclude as in \cite{ghys_holomorphic_1995} since the Bowen-Margulis measure is not invariant by weak stable holonomy (see \cite{pollicott_margulis_1987}, \cite{hamenstadt_new_1989}).\\
We have shown in Corollary \ref{cor:FsFuTrProj} that the weak stable and weak unstable foliations $\mathcal{F}^{ws}$ and $\mathcal{F}^{wu}$ admit transverse projective structures.
As mentioned before, this implies that the lifted foliation $\widetilde{\mathcal{F}^{wu}}$ of $\mathcal{F}^{wu}$ in the universal covering space $\widetilde{M}$ of $M$ is defined by a global submersion: $D_1: \widetilde{M} \to \mathbf{P}^1(\mathbb C).$
Moreover, there is a homomorphism $H_1: \pi_1(M) \to \mathrm{PSL}(2,\mathbb C)$
such that:
\[ \forall \widetilde{x} \in \widetilde{M}, \; \forall \gamma \in \pi_1(M), \quad D_1(\gamma \cdot \widetilde{x})= H_1(\gamma)D_1(\widetilde{x}).\]
Define in the same way $D_2$ and $H_2$ for $\widetilde{\mathcal{F}^{ws}}$.
Let:
\[ D: \left \{\begin{array}{ccl}
        \widetilde{M} & \to & \mathbf{P}^1(\mathbb C) \times \mathbf{P}^1(\mathbb C) \\
        \widetilde{x} & \mapsto & (D_1(\widetilde{x}), D_2(\widetilde{x}))
    \end{array} \right. \quad \text{ and } \quad H: \left \{\begin{array}{ccl}
        \pi_1(M) & \to & \mathrm{PSL}(2,\mathbb C) \times \mathrm{PSL}(2,\mathbb C) \\
        \gamma & \mapsto & (H_1(\gamma), H_2(\gamma))
    \end{array} \right. .\] 
Naturally, we still have  for $\widetilde{x} \in \widetilde{M}$ and $\gamma \in \pi_1(M)$,
$D(\gamma \cdot \widetilde{x})= H(\gamma)D(\widetilde{x}),$
where the second action is the action component by component.
In fact $D$ is a developing map for the transverse $(\mathrm{PSL}(2, \mathbb C) \times \mathrm{PSL}(2, \mathbb C), \mathbf{P}^1(\mathbb C) \times \mathbf{P}^1(\mathbb C))$
structure of $(\varphi^t)_{t}$, and $H$ its corresponding holonomy representation (see the proof of Corollary \ref{cor:flotrproj}).\\
We first prove the following proposition which is at the basis of the classification of transversely holomorphic Anosov transitive flows on smooth compact five-dimensional manifolds:
\begin{proposition}
\label{prop:Dbundle}
\leavevmode
    \begin{enumerate}[label=(\roman*)]
        \item Every orbit of the lifted flow $(\widetilde{\varphi}^t)_{t \in \mathbb R}$ is diffeomorphic to $\mathbb R$ ;
        \item The orbits of the lifted flow are precisely the fibers of $D$ ;
        \item $\widetilde{M} \xrightarrow{D} D(\widetilde{M})$ is a smooth fiber bundle with fiber $\mathbb R$ ;
        \item The image $D(\widetilde{M})$ of $D$ is equal to one of the following sets:
        \begin{enumerate} 
            \item $(\mathbf{P}^1(\mathbb C) \setminus \{a\}) \times (\mathbf{P}^1(\mathbb C)  \setminus \{b\})$, where $a,b\in \mathbf{P}^1(\mathbb C)$ ;
            \item $(\mathbf{P}^1(\mathbb C)\times \mathbf{P}^1(\mathbb C)) \setminus \text{gr}^H(u)$, where $u: \mathbf{P}^1(\mathbb C) \to \mathbf{P}^1(\mathbb C)$ is a homeomorphism and $\text{gr}^H(u):=\{(x,u(x)), \; x\in \mathbf{P}^1(\mathbb C)\}$ is the horizontal graph of $u$.
        \end{enumerate}
    \end{enumerate}
\end{proposition}
We will need several lemmas in order to prove it:
\begin{lemma}
\label{lem:liftFssaff}
 For every $\widetilde{x} \in \widetilde{M}$:
\begin{enumerate}[label=(\roman*)]
    \item $\widetilde{\mathcal{F}^{u}_{\widetilde{x}}}$ and $\widetilde{\mathcal{F}^{s}_{\widetilde{x}}}$ are naturally equipped with a complete complex affine structure of complex dimension $1$.
    \item There exist elements $w_1(\widetilde{x})$ and $w_2(\widetilde{x})$ of $\mathbf{P}^1(\mathbb C)$ such that
    \[\restr{D_2}{\widetilde{\mathcal{F}^{u}_{\widetilde{x}}}}: \widetilde{\mathcal{F}^{u}_{\widetilde{x}}} \to \mathbf{P}^1(\mathbb C) \setminus \{w_2(\widetilde{x})\} \quad \text{ and } \quad \restr{D_1}{\widetilde{\mathcal{F}^{s}_{\widetilde{x}}}}: \widetilde{\mathcal{F}^{s}_{\widetilde{x}}} \to \mathbf{P}^1(\mathbb C) \setminus \{w_1(\widetilde{x})\}\]
    are well-defined diffeomorphisms.
\end{enumerate}
\end{lemma}
\begin{proof}
    $(i):$ Denote by $x=p(\widetilde{x})$. Since, $\mathcal{F}^{u}_x$ has a complete complex affine structure of complex dimension $1$, it is enough to prove that $\restr{p}{\widetilde{\mathcal{F}^{u}_{\widetilde{x}}}}: \widetilde{\mathcal{F}^{u}_{\widetilde{x}}} \to \mathcal{F}^{u}_{x}$  is a well-defined diffeomorphism. The same can therefore be said for strong stable leaves.\\
    By definition of the lifted foliation, $\restr{p}{\widetilde{\mathcal{F}^{u}_{\widetilde{x}}}}: \widetilde{\mathcal{F}^{u}_{\widetilde{x}}} \to \mathcal{F}^{u}_{x}$  is well-defined. We first show it is surjective. Let $y \in \mathcal{F}^{u}_x$. Since $\mathcal{F}^{u}_x$ is connected, there exists a path $\gamma$ in $\mathcal{F}^{u}_x$ connecting $x$ and $y$. By the uniqueness lifting property, there exists a (unique) lift $\widetilde{\gamma}$ of $\gamma$ in $\widetilde{M}$ which starts at $\widetilde{x}$. We note $\widetilde{y}:=\widetilde{\gamma}(1)$. We prove that $\widetilde{\mathcal{F}^{u}_{\widetilde{y}}} = \widetilde{\mathcal{F}^{u}_{\widetilde{x}}}$ so that $y=p(\widetilde{y})$ is in the image of $\restr{p}{\widetilde{\mathcal{F}^{u}_{\widetilde{x}}}}$.
    We recall that for $\widetilde{z} \in \widetilde{M}$,
    $\widetilde{\mathcal{F}^{u}_{\widetilde{z}}}$ is the connected component of $\widetilde{z}$ in $p^{-1}(\mathcal{F}^{u}_{p(\widetilde{z})})$. Therefore, it sufficient to prove that there is path in $p^{-1}(\mathcal{F}^{u}_{x})$ connecting $\widetilde{\mathcal{F}^{u}_{\widetilde{x}}}$ and $\widetilde{\mathcal{F}^{u}_{\widetilde{y}}}$: $\widetilde{\gamma}$ fits since $p\circ \widetilde{\gamma}=\gamma$ is a path in $\mathcal{F}^{u}_x$.\\
    Moreover, if $\widetilde{y}, \widetilde{z}$ are elements of $\widetilde{\mathcal{F}^{u}_{\widetilde{x}}}$ satisfying $p(\widetilde{y})=p(\widetilde{z})=x$, then since $\widetilde{\mathcal{F}^{u}_{\widetilde{x}}}$ is connected (by definition of a leaf), let $\Gamma$ a path in $\widetilde{\mathcal{F}^{u}_{\widetilde{x}}}$ connecting $\widetilde{y}$ and $\widetilde{z}$. The path $p\circ \Gamma$ is therefore a loop in $\mathcal{F}^{u}_{x}$ at $x$. It is necessarily homotopic to the constant path at $x$ since $\mathcal{F}^{u}_{x}$ is simply connected. By a result of algebraic topology (\cite{hatcher_algebraic_2002}, \cite{paulin_introduction_nodate}), the constant path at $\widetilde{x}$ and $\Gamma$ have the same endpoint, i.e. $\widetilde{x}=\Gamma(1)=\widetilde{y}$.\\
    $(ii):$ The proof of Corollary \ref{cor:flotrproj} and the definition of the developing map corresponding to the transverse $( \mathrm{PSL}(2, \mathbb C), \mathbf{P}^1(\mathbb C))$ of the weak stable foliation imply that $\restr{D_2 \circ p}{\mathcal{F}^{u}_x}$ is a developing map corresponding to the complete complex affine structure of complex dimension one of the strong unstable leaf $\mathcal{F}^{u}_x$. Therefore, it is a diffeomorphism from $\mathcal{F}^{u}_x$ to an open subset of $\mathbf{P}^1(\mathbb C)$ diffeomorphic to $\mathbb C$, that is a set of the form $\mathbf{P}^1(\mathbb C) \setminus \{w_2(\widetilde{x})\}$ where $w_2(\widetilde{x})$ is a point of $\mathbf{P}^1(\mathbb C)$. The same can be said for the weak unstable foliation and strong stable leaves.
\end{proof}
\begin{corollary}
The intersection of a leaf of $\widetilde{\mathcal{F}^{u}}$ (respectively, $\widetilde{\mathcal{F}^{s}}$) with a leaf of $\widetilde{\mathcal{F}^{ws}}$ (respectively, $\widetilde{\mathcal{F}^{wu}}$) is either empty or a point.
\end{corollary}
\begin{proof}
    If $\widetilde{\mathcal{F}^{u}_{\widetilde{x}}} \cap \widetilde{\mathcal{F}^{s}_{\widetilde{y}}}$ is not empty, then we can suppose $\widetilde{y}=\widetilde{x}$ and take $\widetilde{x}' \in \widetilde{\mathcal{F}^{u}_{\widetilde{x}}} \cap \widetilde{\mathcal{F}^{s}_{\widetilde{x}}}$. Since $D_2(\widetilde{x}')=D_2(\widetilde{x})$, the previous lemma gives immediately $\widetilde{x}'=\widetilde{x}$.
\end{proof}
\begin{corollary}
The intersection of a leaf of $\widetilde{\mathcal{F}^{ws}}$ with a leaf of $\widetilde{\mathcal{F}^{wu}}$ is either empty or an orbit of the lifted flow on $\widetilde{M}$.
\end{corollary}
\begin{proof}
    This is just a consequence of
    \[\widetilde{\mathcal{F}^{wu}_{\widetilde{x}}} \cap \widetilde{\mathcal{F}^{ws}_{\widetilde{y}}}=\bigcup_{t \in \mathbb R}\widetilde{\mathcal{F}^{u}_{\widetilde{\varphi}^t(\widetilde{x})}} \cap \widetilde{\mathcal{F}^{ws}_{\widetilde{y}}}=\bigcup_{t \in \mathbb R}\widetilde{\varphi}^t\left (\widetilde{\mathcal{F}^{u}_{\widetilde{x}}}\cap \widetilde{\mathcal{F}^{ws}_{\widetilde{y}}} \right )\]
    and the previous corollary.
\end{proof}
\begin{corollary}
\label{cor:QHausd}
    The orbit space $Q_{\widetilde{\Phi}}$ of the lifted flow $(\widetilde{\varphi}^t)_{t \in \mathbb R}$ is Hausdorff.
\end{corollary}
\begin{proof}
    The proof is identical to that of \cite{barbot_caracterisation_1995} p.254.
\end{proof}
\begin{lemma}
\label{lem:U_x}
    For $\widetilde{x} \in \widetilde{M}$, we denote by 
    $U_{\widetilde{x}}=\bigcup_{\widetilde{y}'\in \widetilde{\mathcal{F}^{ws}_{\widetilde{x}}}} \widetilde{\mathcal{F}^{u}_{\widetilde{y}'}}.$
    Then:
    \begin{enumerate}[label=(\roman*)]
        \item For every $\widetilde{x}$ in $\widetilde{M}$, $U_{\widetilde{x}}$ is a non-empty connected open subset of $\widetilde{M}$, stable by the action of the lifted flow $(\widetilde{\varphi}^t)_{t \in \mathbb R}$ ;
        \item If we note 
        $T_{\widetilde{x}}=\bigcup_{\widetilde{y}'\in \widetilde{\mathcal{F}^{s}}(\widetilde{x})} \widetilde{\mathcal{F}}^{u}_{\widetilde{y}'}$,
        then $\restr{D}{T_{\widetilde{x}}}: T_{\widetilde{x}} \to D(U_{\widetilde{x}})$ is a well-defined homeomorphism.
        \item There exists a countable family $(\widetilde{x_i})_{i \in \mathbb N^*}$ of points of $\widetilde{M}$ such that:
        \begin{enumerate} 
            \item $\bigcup_{i \in \mathbb N^*} U_{\widetilde{x_i}} = \widetilde{M}$ ;
            \item For every $k \in \mathbb N^*$, the open set $\Omega_k:= \bigcup_{i=1}^k U_{\widetilde{x_i}}$ is connected.
        \end{enumerate}
    \end{enumerate}
\end{lemma}
\begin{proof}
    $(i):$ By definition, we can write 
    \[U_{\widetilde{x}}=\bigcup_{\widetilde{y}'\in \widetilde{\mathcal{F}^{ws}}(\widetilde{x})} \widetilde{\mathcal{F}}^{u}_{\widetilde{y}'}=\bigcup_{\widetilde{y}'\in \widetilde{\mathcal{F}^{s}}(\widetilde{x})} \widetilde{\mathcal{F}}^{wu}_{\widetilde{y}'}\]
    so $U_{\widetilde{x}}$ is connected.
    Moreover, recall that for any $ \widetilde{z} \in \widetilde{M}$ and $t \in \mathbb R$,
    \begin{align*}
        &\widetilde{\varphi_t} (\widetilde{\mathcal{F}^{u}(\widetilde{z})})=\widetilde{\mathcal{F}^{u}}(\widetilde{\varphi_t}(\widetilde{z}));\\
        &\widetilde{\mathcal{F}^{ws}}(\widetilde{z})=\bigcup_{t \in \mathbb R}\widetilde{\varphi_t} (\widetilde{\mathcal{F}^{s}}(\widetilde{z})).
        \end{align*}
    so for $ \widetilde{x} \in \widetilde{M}$ and $t \in \mathbb R$,
    \[\widetilde{\varphi_t}(U_{\widetilde{x}})= U_{\widetilde{x}}.\]
    To see that $U_{\widetilde{x}}$ is open in $\widetilde{M}$, consider $\widetilde{y} \in U_{\widetilde{x}}$ and a small open neighborhood $V$ of $\widetilde{y}$ in $\widetilde{\mathcal{F}^{s}_{\widetilde{y}}}$, seen as a small transversal to the lifted   foliation, on which the (path-independent) $\widetilde{\mathcal{F}^{u}}$-holonomy between $\widetilde{y}$ and $\widetilde{y'}=\widetilde{\mathcal{F}^{u}_{\widetilde{y}}} \cap \widetilde{\mathcal{F}^{ws}_{\widetilde{x}}}$ is well-defined. Now consider a foliated chart, for the lifted   foliation, containing $V$. We restrict this chart so that every leaf of the lifted   unstable foliation intersect $V$. 
    This gives an open neighborhood of $\widetilde{y}$ in $U_{\widetilde{x}}$.\\
    $(ii):$ $\restr{D}{T_{\widetilde{x}}}$ is well-defined. Let $\widetilde{z} \in U_{\widetilde{x}}$. By the second formula of $U_{\widetilde{x}}$, there exist $\widetilde{y}' \in \widetilde{\mathcal{F}^{s}_{\widetilde{x}}}$ and $\widetilde{z}' \in \widetilde{\mathcal{F}^{u}_{\widetilde{y}'}}$ such that $\widetilde{z}$ and $\widetilde{z}'$ belong to the same orbit of the lifted flow. Therefore, since $\widetilde{z}' \in T_{\widetilde{x}}$ and $D(\widetilde{z})=D(\widetilde{z}')$, the map is surjective. Let $\widetilde{z}', \widetilde{Z}'$ in $T_{\widetilde{x}}$ satisfy $D(\widetilde{z}')=D(\widetilde{Z}')$. Note $\widetilde{y}', \widetilde{Y}' \in \widetilde{\mathcal{F}^{s}_{\widetilde{x}}}$ such that $\widetilde{z}' \in \widetilde{\mathcal{F}^{u}_{\widetilde{y}'}}$ and $\widetilde{Z}' \in \widetilde{\mathcal{F}^{u}_{\widetilde{Y}'}}$. By definition of $D_1$ and $D_2$, we have $D_1(\widetilde{z}')=D_1(\widetilde{y}') \neq w_1(\widetilde{x})$ and $D_2(\widetilde{z}')\neq w_2(\widetilde{y}')$. In the same manner, $D_1(\widetilde{Z}')=D_1(\widetilde{Y}') \neq w_1(\widetilde{x})$ and $D_2(\widetilde{Z}')\neq w_2(\widetilde{Y}')$.
    
    By Lemma \ref{lem:liftFssaff}, since $\widetilde{y}' $ and $ \widetilde{Y}'$  both belong to $\widetilde{\mathcal{F}^{s}_{\widetilde{x}}}$ and verify $D_1(\widetilde{Y}')=D_1(\widetilde{y}')$, it comes $\widetilde{y}'= \widetilde{Y}$. Therefore, by the same Lemma, since $\widetilde{z}'$ and $\widetilde{Z}'$ both belong to $\widetilde{\mathcal{F}^{u}_{\widetilde{y}'}}$ and satisfy $D_2(\widetilde{Z}')=D_2(\widetilde{z}')$, it follows $\widetilde{Z}'=\widetilde{z}'$. It is straightforward to see that $\restr{D}{T_{\widetilde{x}}}$ is a homeomorphism. 
    
    $(iii):$ As in the proof of Theorem \ref{thm:FsTrProj}, take a Markov partition  $(V_i)_{1\leq i \leq k}$ (switch the roles of $E^{s}$ and $E^{u}$) and consider for each $i$ a point $y_i$ in the interior of $P_i$. 
    Consider, for $i \in \llbracket 1, k \rrbracket$, the (countable) set of lifts $(\widetilde{y_i^\alpha})_{\alpha \in A_i}$ of $y_i$ on $\widetilde{M}$.
    Then $(U_{\widetilde{y_i^\alpha}})_{i\in \llbracket 1, k \rrbracket, \alpha \in A_i}$ covers $\widetilde{M}$.
    Write the sequence $S:=(\widetilde{y_i^\alpha})_{i\in \llbracket 1, k \rrbracket, \alpha \in A_i}$ as $(\widetilde{z_n})_{n \in \mathbb N^*}$.
    Fix a first element $\widetilde{x_1}$ of this sequence.
    Then $\Omega_1=U_{\widetilde{x_1}}$ is connected.
    Since $\widetilde{M}$ is connected and each $U_{\widetilde{z_n}}$ is open, there exists an element $\widetilde{x_2} \in S$ such that 
    \[\Omega_1 \cap U_{\widetilde{x_2}} \neq \emptyset.\]
    Therefore, $\Omega_2=\Omega_1 \cup U_{\widetilde{x_2}}$ is connected. \\
    By induction, for every $n \geq 2$, there exists an element $\widetilde{x_{n}}$ satisfying
    \[\Omega_{n-1}\cap U_{\widetilde{x_n}} \neq \emptyset,\]
    so $\Omega_n=\Omega_{n-1} \cup U_{\widetilde{x_n}}$ is connected.
    The result is proven.
\end{proof}
From now, we note $U_i:=U_{\widetilde{x_i}}$ for $i \in \mathbb N^*$.
\begin{proof}[Proof of Proposition \ref{prop:Dbundle}]
    $(i):$ Suppose $\widetilde{z}$ is periodic for the lifted flow. By the previous lemma, let $\widetilde{x_i}$ in the above family such that $\widetilde{z} \in U_{\widetilde{x_i}}$. Let also $\widetilde{y}' \in \widetilde{\mathcal{F}^{s}_{\widetilde{x_i}}}$ and $\widetilde{z}' \in \widetilde{\mathcal{F}^{u}_{\widetilde{y}'}}$ such that $\widetilde{z}$ and $\widetilde{z}'$ belong to the same orbit of the lifted flow. As a result, there exists $t_0>0$ such that $\widetilde{\varphi}^{-t_0}(\widetilde{z}')=\widetilde{z}'$. Since
    \[d(\widetilde{\varphi}^{-kt_0}(\widetilde{y}'),\widetilde{\varphi}^{-kt_0}(\widetilde{z}')) \xrightarrow[k\to +\infty]{} 0,\]
    it comes that $D(\widetilde{y}') = D(\widetilde{z}')$ by continuity of $D$ at $\widetilde{z}'$ and the fact that elements of the same orbit for the lifted flow have the same image by $D$,
    and thus $\widetilde{y}'=\widetilde{z}'$ by Lemma \ref{lem:liftFssaff}.
    In the same manner, we prove that $\widetilde{y}'=\widetilde{x_i}$. In particular, $\widetilde{x_i}$ is periodic for the lifted flow which is absurd.\\
    We will prove the following lemma (see \cite{ghys_holomorphic_1995}, \cite{fang_rigidity_2007}):
\begin{lemma}
\label{lem:DbundleRec}
    With the notations of Lemma \ref{lem:liftFssaff}:
    \begin{enumerate}
        \item For every $\widetilde{x} \in \widetilde{M}$, the fibers of $\restr{D}{U_{\widetilde{x}}}:U_{\widetilde{x}} \to D(U_{\widetilde{x}})$ are precisely the orbits in $U_{\widetilde{x}}$ of the lifted flow.
        Also, $\restr{D}{U_{\widetilde{x}}}:U_{\widetilde{x}} \to D(U_{\widetilde{x}})$ is a trivial $\mathbb R$-fiber bundle and $D(U_{\widetilde{x}})$ is the complement in $\mathbf{P}^1(\mathbb C) \times \mathbf{P}^1(\mathbb C)$ of the union of $\{ w_1(\widetilde{x}) \}\times\mathbf{P}^1(\mathbb C)$ and the horizontal graph of a continuous map from $\mathbf{P}^1(\mathbb C)\setminus\{w_1(\widetilde{x})\}$ to $\mathbf{P}^1(\mathbb C)$.
    \item For every $k \in \mathbb N^*$:
    \begin{enumerate}[label=(\roman*)]
        \item The fibers of $\restr{D}{\Omega_k}$ are precisely the orbits in $\Omega_k$ of the lifted flow .
        \item $\Omega_k \xrightarrow{\restr{D}{\Omega_k}} D(\Omega_k)$ is a smooth $\mathbb R$-fiber bundle.
        \item $D(\Omega_k)$ is either of these two forms:
        \begin{enumerate}[label=(\arabic*)]
            \item $\mathbf{P}^1(\mathbb C)\times \mathbf{P}^1(\mathbb C) \setminus \text{gr}^H(u_k)$, where $u_k: \mathbf{P}^1(\mathbb C) \to \mathbf{P}^1(\mathbb C)$ is a continuous map;
           \item $\mathbf{P}^1(\mathbb C)\times \mathbf{P}^1(\mathbb C) \setminus \left((\{ a_k\} \times \mathbf{P}^1(\mathbb C) )\cup  
            \text{gr}^H(u_k) \right)$,\\
            where $a_k \in \mathbf{P}^1(\mathbb C)$ and $u_k: \mathbf{P}^1(\mathbb C)\setminus \{a_k\} \to \mathbf{P}^1(\mathbb C)$ is a continuous map.
        \end{enumerate}
    \end{enumerate}
    \end{enumerate}
\end{lemma}
\begin{proof}[Proof of Lemma \ref{lem:DbundleRec}]
    $(1):$ Let $\widetilde{z} \in U_{\widetilde{x}}$ and call, as before, $\widetilde{y}' \in \widetilde{\mathcal{F}^{s}_{\widetilde{x}}}$ and $\widetilde{z}' \in \widetilde{\mathcal{F}^{u}_{\widetilde{y}'}}$ such that $\widetilde{z}$ and $\widetilde{z}'$ belong to the same orbit of the lifted flow. We write $D(\widetilde{z}')=(z_1,z_2)$ where $z_1 =D_1(\widetilde{z}')=D_1(\widetilde{y}') \neq w_1(\widetilde{x})$ and $z_2 =D_2(\widetilde{z}')\neq w_2(\widetilde{y}')$. Then as we already said, the orbit of $\widetilde{z}$ lies in $(\restr{D}{U_{\widetilde{x}}})^{-1}(z_1,z_2)$. Conversely, if an element $\widetilde{Z}$ in $U_{\widetilde{x}}$ is such that $D(\widetilde{Z})=(z_1,z_2)$, then if we note $\widetilde{Y}' \in \widetilde{\mathcal{F}^{s}_{\widetilde{x}}}$ and $\widetilde{Z}' \in \widetilde{\mathcal{F}^{u}_{\widetilde{Y}'}}$ such that $\widetilde{Z}$ and $\widetilde{Z}'$ belong to the same orbit of the lifted flow, it comes that $\widetilde{Y}'=\widetilde{y}'$ and $\widetilde{Z}'=\widetilde{z}'$
    since $D(\widetilde{Z}')=D(\widetilde{z}')$. Thus $\widetilde{Z}$ belongs to the orbit of $\widetilde{z}'$.\\
    Now, the map
    \[\left \{\begin{array}{ccl}
        D(U_{\widetilde{x}}) \times \mathbb R & \to & U_{\widetilde{x}}   \\
          ((z_1, z_2), t) & \mapsto & \widetilde{\varphi}^t\left ((\restr{D}{T_{\widetilde{x}}})^{-1}(z_1,z_2) \right )
    \end{array}
     \right. \]
    is well defined, surjective and injective by $(i)$ of Proposition $\ref{prop:Dbundle}$.
    It is straightforward to see that it is an homeomorphism.
    We know look at the set $D(U_{\widetilde{x}})$.
    We know by Lemma \ref{lem:liftFssaff} that $\restr{D_1}{\widetilde{\mathcal{F}^{s}}_{\widetilde{x}}}$ is a diffeomorphism from $\widetilde{\mathcal{F}^{s}}_{\widetilde{x}}$ to $\mathbf{P}^1(\mathbb C)\setminus \{w_1(\widetilde{x})\}$. Therefore, in order to enlighten notation, we will identify $y\in \mathbf{P}^1(\mathbb C)\setminus \{w_1(\widetilde{x})\}$ with $\widetilde{y}=(\restr{D_1}{\widetilde{\mathcal{F}^{s}}_{\widetilde{x}}})^{-1}(y) \in \widetilde{\mathcal{F}^{s}}_{\widetilde{x}}$.
    Again, $\restr{D_2}{\widetilde{\mathcal{F}^{u}_{y}}}$ is a diffeomorphism from $\widetilde{\mathcal{F}^{u}_{y}}$ to $\mathbf{P}^1(\mathbb C)\setminus \{w_2(y)\}$.
    Therefore, if we call $a=w_1(\widetilde{x})$, it is straightforward to see that 
     \[D(U_{\widetilde{x}})= \left (\mathbf{P}^1(\mathbb C)\setminus \{a\} \times \mathbf{P}^1(\mathbb C)\right )\setminus \text{gr}^H(w_2).\]
    By the following elementary fact, since $U_{\widetilde{x}}$ is open and $D$ is a submersion thus an open map, $w_2$ is continuous, and the proof of $(1)$ is complete because 
    \[\left (\mathbf{P}^1(\mathbb C)\setminus \{a\} \times \mathbf{P}^1(\mathbb C)\right )\setminus \text{gr}^H(w_2)=\left (\mathbf{P}^1(\mathbb C) \times \mathbf{P}^1(\mathbb C)\right )\setminus \left ( (\{a\} \times \mathbf{P}^1(\mathbb C))\cup \text{gr}^H(w_2) \right ).\]
    \begin{lemma}
        Let $X,Y$ metric spaces such that $Y$ is compact, and a map $f:X \to Y$.
        Then: $f$ is continuous on $X$ if and only if $\text{gr}^H(f)$ is a closed subset of $X \times Y$.
    \end{lemma}
    \noindent$(2):$ We prove the result by induction on $k \in \mathbb N^*$. The case $k=1$ has already been established in $(1)$. 
    Assume $(i), (ii)$ and $(iii)$ true for a fixed $k \in \mathbb N^*$.
    First we prove $(i)$. $(ii)$ will be immediate. And we will eventually prove $(iii)$.\\
    $(i):$ Again, $D$ is constant on each orbit of the lifted flow. Let $\widetilde{x}, \widetilde{y} \in \Omega_{k+1}$ and assume $D(\widetilde{x})=D(\widetilde{y})$. The goal is to prove that $\widetilde{x}$ and $\widetilde{y}$ belong to the same orbit of the lifted flow. By $(1)$ and the inductive hypothesis, we can assume $\widetilde{x} \in \Omega_k$ and $\widetilde{y} \in U_{k+1}$. By construction,
    $\Omega_k \cap U_{k+1}$ is not empty.
    We start by the following fact:
    \begin{lemma}
    Let $M$ a smooth connected manifold of dimension $n \geq 2$.
        \begin{enumerate}[label=(\roman*)]
            \item If $N$ is a smooth submanifold of $M$ of codimension $q \geq 2$, then $M\setminus N$ is connected.
            \item If $N_1, \cdots , N_k$ are smooth submanifolds of $M$, each closed and of codimension greater or equal than $2$, then $M\setminus (\bigcup_{i=1}^k N_i)$ is connected.
        \end{enumerate}
    \end{lemma}
    Here, by the induction hypothesis $(iii)$, $D(\Omega_k)$ is the complement in $\mathbf{P}^1(\mathbb C) \times \mathbf{P}^1(\mathbb C)$ of the graph of a continuous map $u_k: \mathbf{P}^1(\mathbb C) \to \mathbf{P}^1(\mathbb C) $ or the complement in $\mathbf{P}^1(\mathbb C) \times \mathbf{P}^1(\mathbb C)$ of the union of a vertical $\{a_k\} \times \mathbf{P}^1(\mathbb C) $ and the graph of a continuous map $u_k: \mathbf{P}^1(\mathbb C)\setminus\{a_k\} \to \mathbf{P}^1(\mathbb C)$. Also, by $(1)$, $D(U_{k+1})$ is the complement in $\mathbf{P}^1(\mathbb C) \times \mathbf{P}^1(\mathbb C)$ of the union of a vertical $\{a'_{k+1}\} \times \mathbf{P}^1(\mathbb C) $ and the graph of a continuous map $u'_{k+1}: \mathbf{P}^1(\mathbb C)\setminus\{a'_{k+1}\} \to \mathbf{P}^1(\mathbb C)$. 
    Therefore, $D(\Omega_k)\cap D(U_{k+1})$ is the complement in $\mathbf{P}^1(\mathbb C) \times \mathbf{P}^1(\mathbb C)$ of a finite union of smooth closed submanifolds of $\mathbf{P}^1(\mathbb C)\times \mathbf{P}^1(\mathbb C)$ of codimension $2$. 
    By the previous fact, $D(\Omega_k)\cap D(U_{k+1})$ is a (non-empty) connected open subset of $\mathbf{P}^1(\mathbb C) \times \mathbf{P}^1(\mathbb C)$. It is thus path-connected.
    As a result, let $\widetilde{z} \in \Omega_k \cap U_{k+1}$ and $\gamma$ a continuous curve in $D(\Omega_k)\cap D(U_{k+1})$ which starts at $D(\widetilde{z})$ and ends at $D(\widetilde{x})=D(\widetilde{y})$.
    Since $\restr{D}{\Omega_k}: \Omega_k \to D(\Omega_k)$ and $\restr{D}{U_{k+1}}: U_{k+1} \to D(U_{k+1})$ are both fiber bundles by the induction hypothesis $(ii)$ and $(1)$, they satisfy the homotopy lifting property. There exist therefore a continuous curve $\widetilde{\gamma}_1$ in $\Omega_k$ which starts at $\widetilde{z}$ and a continuous curve $\widetilde{\gamma}_2$ in $U_{k+1}$ which starts at $\widetilde{z}$ such that $D \circ \widetilde{\gamma}_1=D \circ \widetilde{\gamma}_2= \gamma.$
    Since $D(\widetilde{\gamma}_1(1))=\gamma(1)=D(\widetilde{x})$ and both $\widetilde{\gamma}_1(1) $ and $\widetilde{x}$ are in $\Omega_k$, then by the induction hypothesis $(ii)$ , $\widetilde{\gamma}_1(1)$ and $\widetilde{x}$ belong to the same orbit of the lifted flow.
    By $(1)$, the same is true for $\widetilde{\gamma}_2(1)$ and $\widetilde{y}$ since they both live in $U_{k+1}$ and have the same image by $D$.
   We therefore want to show that $\widetilde{\gamma}_1(1)$ and $\widetilde{\gamma}_2(1)$ belong to the same orbit of the lifted flow.\\
   Let $\Lambda=\{t\in [0,1], \; \widetilde{\gamma}_1(t) \text{ and } \widetilde{\gamma}_2(t) \text{ belong to the same orbit of the lifted flow}\}.$
   It is not empty because $\widetilde{\gamma}_1(0)=\widetilde{z}=\widetilde{\gamma}_2(0)$.
   Let $t_0 \in \Lambda$, i.e. $\widetilde{\gamma}_1(t_0)$ and $\widetilde{\gamma}_2(t_0)$ belong to the same orbit of the lifted flow. By the invariance of  $U_{k+1}$ by the action of the lifted flow, and because $\widetilde{\gamma}_2(t_0) \in U_{k+1}$, it follows that $\widetilde{\gamma}_1(t_0) \in U_{k+1}$.
   Since $U_{k+1}$ is open and each curve $\widetilde{\gamma}_1$ and $\widetilde{\gamma}_2$ is continuous, there exists a small $\epsilon>0$ such that 
       $\widetilde{\gamma}_1 \left([t_0-\epsilon, t_0 +\epsilon] \cap [0,1] \right)\subset U_{k+1}$ and $\widetilde{\gamma}_2 \left([t_0-\epsilon, t_0 +\epsilon] \cap [0,1] \right) \subset U_{k+1}$
   Therefore, for $t \in [t_0-\epsilon, t_0 +\epsilon] \cap [0,1]$,
       $D(\widetilde{\gamma}_1(t))= \gamma(t) = D(\widetilde{\gamma}_2(t))$ and 
       $\widetilde{\gamma}_1(t), \widetilde{\gamma}_2(t) \in U_{k+1}$,
    thus $\widetilde{\gamma}_1(t)$ and $\widetilde{\gamma}_2(t)$ belong to the same orbit of the lifted flow by $(1)$. This shows that $\Lambda$ is open in $[0,1]$.\\
    Let $(t_n)_{n\in \mathbb N}$ a sequence of $\Lambda$ converging to $t^* \in [0,1]$. Denote by $\widetilde{\Phi_1}$ the lifted flow orbit of $\widetilde{\gamma}_1(t^*)$ and by $\widetilde{\Phi_2}$ that of $\widetilde{\gamma}_2(t^*)$.
    Let $O_1$ and $O_2$ $\widetilde{\Phi}$-saturated open neighborhoods of $\widetilde{\Phi_1}$ and $\widetilde{\Phi_2}$ respectively.
    For $n \in \mathbb N$ large enough,  $\widetilde{\gamma}_1(t_n) \in O_1$ and $\widetilde{\gamma}_2(t_n) \in O_2$. Since $\widetilde{\gamma}_1(t_n)$ and $\widetilde{\gamma}_2(t_n)$ belong to the same orbit of the lifted flow, their orbit belong to $O_1$ and $O_2$.
    Therefore, by corollary \ref{cor:QHausd}, $\widetilde{\Phi_1} =\widetilde{\Phi_2}$, i.e. $\widetilde{\gamma}_1(t^*)$ and $\widetilde{\gamma}_2(t^*)$ belong to the same orbit of the lifted flow.\\
    By connectedness of $[0,1]$, $\Lambda=[0,1]$ so $\widetilde{\gamma}_1(1)$ and $\widetilde{\gamma}_2(1)$ belong to the same orbit of the lifted flow, which ends the proof of $(i)$ for the inductive step $k+1$.\\
    We now prove $(ii)$ for the inductive step $k+1$. We want to prove that any point $z$ of $D(\Omega_{k+1})=D(\Omega_k) \cup D(U_{k+1})$ is included in an open subset $U$ of $D(\Omega_{k+1})$ such that 
    $\Omega_{k+1}\cap D^{-1}(U)\xrightarrow[]{D} U$
    is isomorphic to 
    $ U\times \mathbb R\xrightarrow[]{\text{pr}_1} U.$
    If $z \in D(\Omega_k)$, then by the inductive hypothesis $(ii)$, there exists an open neighborhood $U$ of $z$ in $D(\Omega_{k})$ (and thus in $D(\Omega_{k+1})$ since $D(\Omega_{k})$ is an open subset of $D(\Omega_{k+1})$) such that 
    $\Omega_{k}\cap D^{-1}(U)\xrightarrow[]{D} U$
    is isomorphic to 
    $U\times \mathbb R\xrightarrow[]{\text{pr}_1} U.$
    We prove in fact that $\Omega_{k}\cap D^{-1}(U)=\Omega_{k+1}\cap D^{-1}(U)$. The first inclusion is immediate. If $\widetilde{x}\in U_{k+1}\cap D^{-1}(U)$, then there exists $\widetilde{y}\in \Omega_k$ such that $D(\widetilde{x})=D(\widetilde{y})$. By $(i)$ of the inductive step $k+1$, $\widetilde{x}$ and $\widetilde{y}$ belong to the same orbit of the lifted flow. In particular, since $\Omega_k$ is invariant by the action of the lifted flow, it follows that $\widetilde{x} \in \Omega_k$.\\
    In the same manner, if $z \in D(U_{k+1})$, by $(1)$, there exists an open neighborhood $U$ of $z$ in $D(U_{k+1})$ (and thus in $D(\Omega_{k+1})$) such that 
    $U_{k+1}\cap D^{-1}(U)\xrightarrow[]{D} U$ is isomorphic to 
    $U\times \mathbb R\xrightarrow[]{\text{pr}_1} U.$
    Again it appears that $U_{k+1}\cap D^{-1}(U)=\Omega_{k+1}\cap D^{-1}(U)$ by the same reasoning as the previous one.
    This proves $(ii)$ of the inductive step $k+1$.\\
    We eventually prove $(iii)$ of the inductive step $k+1$. First, by the inductive hypothesis $(iii)$, assume that $D(\Omega_k)$ is the complement in $\mathbf{P}^1(\mathbb C) \times \mathbf{P}^1(\mathbb C)$ of the graph of a continuous map $u_k: \mathbf{P}^1(\mathbb C) \to \mathbf{P}^1(\mathbb C)$.
    \begin{lemma}
        For any $p\neq a'_{k+1}$, $u'_{k+1}(p)=u_k(p)$.
    \end{lemma}
    \begin{proof}
        Suppose there exists $p_0\neq a'_{k+1}$ such that $u'_{k+1}(p_0)\neq u_k(p_0)$. Then 
        \[(p_0, u'_{k+1}(p_0))\in (\mathbf{P}^1(\mathbb C)\times \mathbf{P}^1(\mathbb C)) \setminus \text{gr}^H(u_k)=D(\Omega_k).\]
        Therefore, there exists $i_0\in\llbracket 1, k \rrbracket$ and $\widetilde{y_1} \in \widetilde{\mathcal{F}^{s}_{\widetilde{x_{i_0}}}}$ such that $(p_0,u'_{k+1}(p_0)) \in D(\widetilde{\mathcal{F}^{u}_{\widetilde{y_1}}})$. We also know that the set
        $D(\widetilde{\mathcal{F}^{u}_{\widetilde{y_1}}})$ is of the form $\{b\}\times (\mathbf{P}^1(\mathbb C) \setminus \{c\})$. So since it is included in $D(\Omega_k)=(\mathbf{P}^1(\mathbb C)\times \mathbf{P}^1(\mathbb C)) \setminus \text{gr}^H(u_k)$ and $b=p_0$, then necessarily $c=u_k(p_0)$.
        Also, by recalling the definition of $a'_{k+1} \in \mathbf{P}^1(\mathbb C)$ and $u'_{k+1}: \mathbf{P}^1(\mathbb C)\setminus \{a'_{k+1}\} \to \mathbf{P}^1(\mathbb C)$ in the proof of $(1)$, it comes
        \[\restr{D_1}{\widetilde{\mathcal{F}^{s}_{\widetilde{x_{k+1}}}}}: \widetilde{\mathcal{F}^{s}_{\widetilde{x_{k+1}}}} \to \mathbf{P}^1(\mathbb C) \setminus \{a'_{k+1}\}\]
        is a diffeomorphism so if we note $\widetilde{y}_2 = (\restr{D_1}{\widetilde{\mathcal{F}^{s}_{\widetilde{x_{k+1}}}}})^{-1}(p_0) \in \widetilde{\mathcal{F}^{s}_{\widetilde{x_{k+1}}}}$, then
        \[\restr{D}{\widetilde{\mathcal{F}^{u}_{\widetilde{y_2}}}}: \widetilde{\mathcal{F}^{u}_{\widetilde{y_2}}} \to \{p_0\} \times (\mathbf{P}^1(\mathbb C) \setminus \{u'_{k+1}(p_0)\})\]
        is a diffeomorphism. As a result,
        \[D(\widetilde{\mathcal{F}^{u}_{\widetilde{y_1}}})\cap D(\widetilde{\mathcal{F}^{u}_{\widetilde{y_2}}})=\left (\{p_0\}\times (\mathbf{P}^1(\mathbb C) \setminus \{u_k(p_0)\}) \right) \cap \left (\{p_0\}\times (\mathbf{P}^1(\mathbb C) \setminus \{u'_{k+1}(p_0)\}) \right )\neq \emptyset.\]
        Let therefore $\widetilde{x}_1 \in \widetilde{\mathcal{F}^{u}_{\widetilde{y_1}}} \subset \Omega_k$ and $\widetilde{x}_2 \in \widetilde{\mathcal{F}^{u}_{\widetilde{y_2}}} \subset U_{k+1}$ such that
        $D(\widetilde{x_1})=D(\widetilde{x_2}).$
        By $(i)$ of the inductive step $k+1$, $\widetilde{x_1}$ and $\widetilde{x_2}$ belong to the same orbit of the lifted flow, that is there exists $t_0 \in \mathbb R$ such that $\widetilde{\varphi}^{t_0}(\widetilde{x_2})=\widetilde{x_1}$.
        Eventually,
        \begin{align*}
            \{p_0\}\times (\mathbf{P}^1(\mathbb C) \setminus \{u'_{k+1}(p_0)\})&=D(\widetilde{\mathcal{F}^{u}_{\widetilde{y_2}}})=D(\widetilde{\mathcal{F}^{u}_{\widetilde{x_2}}})=D(\widetilde{\varphi}^{t_0}(\widetilde{\mathcal{F}^{u}_{\widetilde{x_2}}}))=D(\widetilde{\mathcal{F}^{u}_{\widetilde{\varphi}^{t_0}(\widetilde{x_2})}})\\
            &=D(\widetilde{\mathcal{F}^{u}_{\widetilde{x_1}}})=D(\widetilde{\mathcal{F}^{u}_{\widetilde{y_1}}})
            =\{p_0\}\times (\mathbf{P}^1(\mathbb C) \setminus \{u_k(p_0)\})
        \end{align*}
        which is absurd since $u'_{k+1}(p_0)\neq u_k(p_0)$.
    \end{proof}
    \begin{proof}[End of proof of Lemma \ref{lem:DbundleRec}]
    To conclude the first case of $(iii)$ of the inductive step $k+1$, we thus just write
    \begin{align*}
       (\mathbf{P}^1(\mathbb C)\times \mathbf{P}^1(\mathbb C)) \setminus \text{gr}^H(u_k)\supset (\mathbf{P}^1(\mathbb C)\times \mathbf{P}^1(\mathbb C)) \setminus \left((\{ a'_{k+1}\} \times \mathbf{P}^1(\mathbb C) )\cup  
            \text{gr}^H(u'_{k+1}) \right)
    \end{align*}
    which is immediate by the previous claim and by considering the complements. Therefore
    \[D(\Omega_{k+1})=D(\Omega_{k})\cup D(U_{k+1})=D(\Omega_{k})=(\mathbf{P}^1(\mathbb C)\times \mathbf{P}^1(\mathbb C)) \setminus \text{gr}^H(u_k).\]
    Eventually, by the inductive hypothesis $(iii)$, assume that $D(\Omega_k)$ is the complement in $\mathbf{P}^1(\mathbb C) \times \mathbf{P}^1(\mathbb C)$ of the union of a vertical $\{a_k\} \times \mathbf{P}^1(\mathbb C) $ and the graph of a continuous map $u_k: \mathbf{P}^1(\mathbb C)\setminus\{a_k\} \to \mathbf{P}^1(\mathbb C)$. 
    By the exact same reasoning as the one in the proof of the previous claim, it comes that for any $p \in \mathbf{P}^1(\mathbb C)$ different from $a_k$ and $a'_{k+1}$, $u'_{k+1}(p)=u_k(p)$.\\
    If $a_k \neq a'_{k+1}$, then, by what we have just said, we can define the map $u: \mathbf{P}^1(\mathbb C) \to \mathbf{P}^1(\mathbb C)$ by $u(p)=u_k(p)$ if $p \neq a_k$ and $u(p)=u'_{k+1}(p)$ if $p \neq a'_{k+1}$. $u$ is continuous since $u_k$ and $u'_{k+1}$ are.
    It is straightforward to verify, by considering the complements, that
    $D(\Omega_{k+1})= (\mathbf{P}^1(\mathbb C)\times \mathbf{P}^1(\mathbb C)) \setminus \text{gr}^H(u).$\\
    If $a_k = a'_{k+1}$, then $u_k=u'_{k+1}$ so 
    $D(\Omega_{k+1})=D(\Omega_{k})=(\mathbf{P}^1(\mathbb C)\times \mathbf{P}^1(\mathbb C)) \setminus \left((\{ a_{k}\} \times \mathbf{P}^1(\mathbb C) )\cup  
            \text{gr}^H(u_{k}) \right),$
    which concludes the proof of Lemma \ref{lem:DbundleRec}
\end{proof} 
    We are now ready to prove $(ii), (iii), (iv)$ of Proposition \ref{prop:Dbundle}.\\
    $(ii):$ Let $\widetilde{x}, \widetilde{y} \in \widetilde{M}$. Then there exists $k\in \mathbb N^*$ large enough so that $\widetilde{x}, \widetilde{y} \in \Omega_k$. By $(2,i)$ of the previous lemma, $\widetilde{x}$ and $\widetilde{y}$ belong to the same orbit of the lifted flow. Again, the converse is straightforward.\\
    $(iii):$ As in the proof of Lemma \ref{lem:DbundleRec} $(2,ii)$, we want to prove that any point $z$ of $D(M)=\bigcup_{k \in \mathbb N^*} D(\Omega_k)$ is included in an open subset $U$ of $D(M)$ such that $D^{-1}(U)\xrightarrow[]{D} U $ is isomorphic to 
    $U\times \mathbb R\xrightarrow[]{\text{pr}_1} U.$
    If $z \in D(\Omega_{k_0})$ for some $k_0 \in \mathbb N^*$, then by $(2,ii)$ of that lemma, there exists an open neighborhood $U$ of $z$ in $D(\Omega_{k_0})$ (and thus in $D(M)$ since $D(\Omega_{k_0})$ is an open subset of $D(M)$) such that 
    $\Omega_{k_0}\cap D^{-1}(U)\xrightarrow[]{D} U$ is isomorphic to 
    $U\times \mathbb R\xrightarrow[]{\text{pr}_1} U.$
    In fact $\Omega_{k_0}\cap D^{-1}(U)= D^{-1}(U)$ by the same argument as before, which concludes.\\
    $(iv):$ The proof in $(iii)$ of the inductive step in fact shows the following.
    For $k \in \mathbb N^*$, if $D(\Omega_k)$ is the complement in $\mathbf{P}^1(\mathbb C) \times \mathbf{P}^1(\mathbb C)$ of the graph of a continuous map $u_k: \mathbf{P}^1(\mathbb C) \to \mathbf{P}^1(\mathbb C)$, then $D(\Omega_{k+1})=D(\Omega_k)$. If $D(\Omega_k)$ is the complement in $\mathbf{P}^1(\mathbb C) \times \mathbf{P}^1(\mathbb C)$ of the union of a vertical $\{a_k\} \times \mathbf{P}^1(\mathbb C) $ and the graph of a continuous map $u_k: \mathbf{P}^1(\mathbb C)\setminus\{a_k\} \to \mathbf{P}^1(\mathbb C)$, then knowing that $D(U_{k+1})$ is the complement in $\mathbf{P}^1(\mathbb C) \times \mathbf{P}^1(\mathbb C)$ of the union of a vertical $\{a'_{k+1}\} \times \mathbf{P}^1(\mathbb C) $ and the graph of a continuous map $u'_{k+1}: \mathbf{P}^1(\mathbb C)\setminus\{a'_{k+1}\} \to \mathbf{P}^1(\mathbb C)$, it comes:
    \begin{enumerate}
        \item either $u_k = u'_{k+1}$, and therefore $D(\Omega_{k+1})=D(\Omega_k)$ ;
        \item either $u_k \neq u'_{k+1}$, and there exists a continuous map $u:\mathbf{P}^1(\mathbb C)\to \mathbf{P}^1(\mathbb C)$ such that 
        \[D(\Omega_{k+1})= (\mathbf{P}^1(\mathbb C)\times \mathbf{P}^1(\mathbb C)) \setminus \text{gr}^H(u).\]
    \end{enumerate}
    As a result, if we note for $i \in \mathbb N^*$, $D(U_i)$ as the complement in $\mathbf{P}^1(\mathbb C) \times \mathbf{P}^1(\mathbb C)$ of the union of a vertical $\{a'_{i}\} \times \mathbf{P}^1(\mathbb C) $ and the graph of a continuous map $u'_{i}: \mathbf{P}^1(\mathbb C)\setminus\{a'_{i}\} \to \mathbf{P}^1(\mathbb C)$, then as the induction starts by $D(U_1)$, it follows:
    \begin{enumerate}
        \item either for all $i \in \mathbb N^*$, $u'_i=u'_1$, thus $D(\widetilde{M})=D(U_1)$ is the complement in $\mathbf{P}^1(\mathbb C) \times \mathbf{P}^1(\mathbb C)$ of the union of $\{a\} \times \mathbf{P}^1(\mathbb C) $ and the graph of the continuous map $u': \mathbf{P}^1(\mathbb C)\setminus\{a\} \to \mathbf{P}^1(\mathbb C)$, where $a:=a'_1$ and $u':=u'_1$ ;
        \item or there exists $k_0\geq 2$ (taken minimal) such that $u'_{k_0}\neq u'_1$, thus if we note $u:\mathbf{P}^1(\mathbb C) \to \mathbf{P}^1(\mathbb C) $ the continuous map extending $u'_{k_0}$ and $u'_{1}$ (see the proof of $(iii)$ of the inductive step), then $D(\widetilde{M})$ is the complement in $\mathbf{P}^1(\mathbb C) \times \mathbf{P}^1(\mathbb C)$ of the horizontal graph of $u$.
    \end{enumerate}
    Moreover, we can go over Lemma \ref{lem:U_x} and exchange the role of the strong unstable leaves by the strong stable to come to the analog conclusion we've just given:
     \begin{enumerate}[label=(\arabic*)']
        \item either $D(\widetilde{M})$ is the complement in $\mathbf{P}^1(\mathbb C) \times \mathbf{P}^1(\mathbb C)$ of the union of $\mathbf{P}^1(\mathbb C) \times \{b\}$ and the graph of a continuous map $w': \mathbf{P}^1(\mathbb C)\setminus\{b\} \to \mathbf{P}^1(\mathbb C)$ ;
        \item or $D(\widetilde{M})$ is the complement in $\mathbf{P}^1(\mathbb C) \times \mathbf{P}^1(\mathbb C)$ of the vertical graph of a continuous map $w:\mathbf{P}^1(\mathbb C) \to \mathbf{P}^1(\mathbb C) $.
    \end{enumerate}
    By combining these possibilities, the only options left are the one given by $(iii)$ of Proposition \ref{prop:Dbundle}.
    This concludes the proof of Proposition \ref{prop:Dbundle}.
\end{proof}
\end{proof}
We are now able to start the proof of Theorem \ref{thm:classif}.
\begin{proof}[Proof of Theorem \ref{thm:Fstrholo}]
    First assume that there exists $a,b \in \mathbf{P}^1(\mathbb C)$ such that
    $D(\widetilde{M})=(\mathbf{P}^1(\mathbb C) \setminus \{a\}) \times (\mathbf{P}^1(\mathbb C)  \setminus \{b\}).$
    By post-composing $D$ by a certain couple of Möbius transformations, we can assume
    $D(\widetilde{M})=\mathbb C \times \mathbb C.$
    Recall also that for $\gamma \in \pi_1(M)$ and $\widetilde{x} \in \widetilde{M}$, 
    $D(\gamma \cdot \widetilde{x})= H(\gamma)D(\widetilde{x}).$
    As a result, for $\gamma \in \pi_1(M)$, the Möbius transformations $H_1(\gamma)$ and $H_2(\gamma)$ both sends $\mathbb C$ to $\mathbb C$. Thus they must be affine maps of $\mathbb C$. 
    Therefore, by the following lemma, the weak unstable foliation $\mathcal{F}^u$ of our transversely holomorphic Anosov flow $(\varphi^t)_t$ admits a transverse $(\text{Aff}(\mathbb R^4), \mathbb R^4)$-structure. 
    \begin{lemma}
    \label{lem:trG'X'stru}
       Let $X$ a smooth manifold on which acts analytically a group $G$. Let also $X'$ an open subset of $X$ and $G'$ a subgroup of $G$.
       Suppose $\mathcal{F}$ is a foliation, on a smooth manifold $M$, with a transverse $(G,X)$-structure. Suppose also that there exists a developing map $D: \widetilde{M} \to X$ and a corresponding holonomy representation $H: \pi_1(M) \to G$ for this transverse $(G,X)$-structure which satisfy
       \[D(\widetilde{M}) \subset X' \quad \text{ and } \quad H(\pi_1(M)) \subset G'.\]
       Then $\mathcal{F}$ admits a transverse $(G', X')$-structure, compatible with the initial transverse $(G,X)$-structure. Moreover, $D: \widetilde{M} \to X'$ is a developing map and $H$ is the corresponding holonomy representation $H: \pi_1(M) \to G'$ for this transverse $(G', X')$-structure .
    \end{lemma}
    \begin{proof}
        Consider a family of submersions $(U_i,s_i)$ defining the transverse $(G,X)$-structure of $\mathcal{F}$ and which allows to compute the given developing map $D: \widetilde{M} \to X'$ (see Definition \ref{def:devmap}). 
        Let $x_0 \in U_0$ the corresponding base point.
        Let $i,j$ such that $U_i\cap U_j \neq \emptyset$ and consider a loop $\beta$ that starts in $U_i$ and ends in $U_j$. 
        The transition map between $U_i$ and $U_j$ is the restriction to $s_j(U_i\cap U_j)$ of an element $g$ of $G$.
        However we know by definition that $H(\beta)=g$  belongs to $G'$.
        Now let $x \in U_0$ and $\alpha$ a path in $U_0$ from $x_0$ to $x$.
        Then $D(\tilde{\alpha})=s_0(x)$ by definition, which belongs to $X'$.
        Eventually, for $i\neq 0$ fixed and $x \in U_i$, consider a path $\alpha$ between $x_0$ and $x$, going through $U_0, \cdots, U_{i-1}, U_i$.
        Then with the notations of Definition \ref{def:devmap}, $D(\tilde{\alpha})=g_0 \cdots g_{i-1} \cdot s_i(x)$ belongs to $X'$, thus $s_i(x) \in X'$ since $g_0 \cdots g_{i-1} \in G'$ by the above.
        This proves that the initial chosen family of submersion atlas defines in fact a transverse $(G',X')$-structure for $\mathcal{F}$ obviously compatible with the transverse $(G,X)$-structure, and $D$ and $H$ are obviously respectively developing map and corresponding holonomy map for this transverse $(G',X')$-structure.
    \end{proof}
    \noindent By \cite{plante_anosov_1981}, $(\varphi^t)_{t\in \mathbb R}$ is $C^\infty$-orbit equivalent to the suspension of an Anosov diffeomorphism $f$. We note $\theta: M \to F_f$ the orbit conjugacy between $M$ and the suspension manifold of $f$. Still denote by $F$ the embedded global transverse section of the suspension flow and $f:F \to F$ the corresponding Anosov diffeomorphism. $F$ is connected and without boundary. It is also compact since $M$ is and thus $F_f$ also.
    Call $\Sigma:=\theta^{-1}(F)$
    and $\phi: \Sigma \to \Sigma$
    the Poincaré map of $\Sigma$ for the flow $(\varphi^t)_{t \in \mathbb R}$. It is an Anosov diffeomorphism of $\Sigma$.
    The map $\restr{\theta}{\Sigma}:\Sigma \to F$ smoothly conjugates $\phi: \Sigma \to \Sigma$ and $f: F \to F$.
    Since $(\varphi^t)_{t \in \mathbb R}$ is transversely holomorphic, by Proposition \ref{prop:caractrholo}, $\Sigma$ is a complex surface and $\phi: \Sigma \to \Sigma$ is holomorphic with respect to this complex structure. 
    Therefore, by transporting the complex structure of $\Sigma$ to $F$ via $\restr{\theta}{\Sigma}$, $F$ is a compact connected complex surface and $f:F\to F$ is holomorphic with respect to this complex structure. 
    By Theorem A of \cite{ghys_holomorphic_1995}, $F$ is biholomorphic to a complex torus $\mathbb C^2\setminus\Lambda$ and, under this correspondence, $f$ is holomorphically conjugate to a hyperbolic toral automorphism $\widehat{A}$ of $\mathbb C^2\setminus\Lambda$.
    As a result, the suspension flow of $f$ on $F_f$ and the suspension flow of $\widehat{A}$ on $\mathbb C^2\setminus\Lambda$ are $C^\infty$-flow equivalent. 
    Thus $(\varphi^t)_{t \in \mathbb R}$ is $C^\infty$-orbit equivalent to the suspension of a hyperbolic toral automorphism.\\

    Eventually, assume that there exists a homeomorphism $u: \mathbf P^1(\mathbb C) \to \mathbf P^1(\mathbb C)$ such that
    \[D(\widetilde{M})=(\mathbf P^1(\mathbb C)\times \mathbf P^1(\mathbb C)) \setminus \text{gr}^H(u).\]
    Let $\Gamma:=\pi_1(M)$ denote the fundamental group of $M$. 
    \begin{lemma}
        $H:\Gamma \to \mathrm{PSL}(2,\mathbb C)\times\mathrm{PSL}(2,\mathbb C)$ is injective.
    \end{lemma}
    
    \begin{proof}
        Assume on the contrary that there exists a non-trivial element $\gamma \in \Gamma$ such that $H(\gamma)=I_{\mathrm{PSL(2,\mathbb C)} \times \mathrm{PSL(2,\mathbb C)}}$. Then for every $\widetilde{x} \in \widetilde{M}$, since
        $D(\gamma \cdot \widetilde{x})=H(\gamma)\cdot D(\widetilde{x})=D(\widetilde{x}),$
        then by $(ii)$ of \ref{prop:Dbundle}, $\widetilde{x}$ and $\gamma \cdot \widetilde{x}$ belong to the same orbit of the lifted flow, that is there exists $t_0 \in \mathbb R$ such that 
        $\gamma \cdot \widetilde{x}=\widetilde{\varphi}^{t_0}(\widetilde{x}).$
        Recall that an automorphism of $\widetilde{M} \xrightarrow[]{p}M$ is trivial if and only if it has a fixed point. 
        Therefore, $t_0\neq 0$. By applying $p$ to the previous equality, it follows that
        $x=\varphi^{t_0}(x)$,
        where $x=p(\widetilde{x})$. Thus, every point of $M$ is periodic, which is absurd if we consider for example a point whose orbit is dense in $M$ (which exists by topological transitivity).
    \end{proof}

    \begin{lemma}
    For $\gamma \in \Gamma$,
    $u \circ H_1(\gamma)\circ u^{-1}=H_2(\gamma).$\\
    In particular, $H_1: \Gamma \to \mathrm{PSL(2, \mathbb C)}$ and $H_2: \Gamma \to \mathrm{PSL(2, \mathbb C)}$ are injective.
    \end{lemma}
    \begin{proof}
        First, $D_1: \widetilde{M} \to \mathbf{P}^1(\mathbb C)$ is surjective. Indeed, let $z \in \mathbf{P}^1(\mathbb C)$ and $y\neq u(z)$. Then 
        \[(z,y) \in (\mathbf P^1(\mathbb C)\times \mathbf P^1(\mathbb C)) \setminus \text{gr}^H(u)=D(\widetilde{M}).\]
        Suppose on the contrary that there exists $\gamma_0 \in \Gamma$ such that 
        $u \circ H_1(\gamma_0) \neq H_2(\gamma_0)\circ u.$
       This means, by the above, that there exists $\widetilde{x_0} \in \widetilde{M}$ such that 
       $u(H_1(\gamma_0) \cdot D_1(\widetilde{x_0}))\neq H_2(\gamma_0)\cdot u(D_1(\widetilde{x_0})).$
       Therefore, 
       \[\left (H_1(\gamma_0) \cdot D_1(\widetilde{x_0}),H_2(\gamma_0)\cdot u(D_1(\widetilde{x_0}))\right ) \in (\mathbf P^1(\mathbb C)\times \mathbf P^1(\mathbb C)) \setminus \text{gr}^H(u)=D(\widetilde{M}),\]
       i.e. there exists $\widetilde{y} \in \widetilde{M}$ such that
          $ H_1(\gamma)\cdot D_1(\widetilde{x_0})=D_1(\widetilde{y})$ and 
          $ H_2(\gamma)\cdot u(D_1(\widetilde{x_0}))=D_2(\widetilde{y}).$
       As a result, since $H_1: \Gamma \to \mathrm{PSL}(2,\mathbb C)$ and $H_2: \Gamma \to \mathrm{PSL}(2,\mathbb C)$ are homomorphisms, and by invariance of $D$ with respect to the natural actions of $\Gamma$ on $\widetilde{M}$ and $H(\Gamma)$ on $\mathbf{P}^1(\mathbb C) \times \mathbf{P}^1(\mathbb C)$, it comes
           $D_1(\widetilde{x_0})=D_1(\gamma^{-1}\cdot \widetilde{y})$ and 
           $u(D_1(\widetilde{x_0}))=D_2(\gamma^{-1}\cdot \widetilde{y})$
       which is a contradiction since $\text{gr}^H(u)\cap D(\widetilde{M})=\emptyset$.\\
       Now if $\gamma \in \Gamma$ is such that $H_1(\gamma)=I_{\mathrm{PSL}(2,\mathbb C)}$, then by the above $H_2(\gamma)=I_{\mathrm{PSL}(2,\mathbb C)}$, thus $H(\gamma)=I_{\mathrm{PSL}(2,\mathbb C)\times \mathrm{PSL}(2,\mathbb C)}$ so $\gamma=0$ since $H$ is injective.
    The same can be said for $H_2$.
    \end{proof}
    In order to show the orbit equivalence with a geodesic flow on a hyperbolic manifold, we must exhibit first a discrete torsion-free cocompact subgroup of $\mathrm{PSL}(2,\mathbb C)$.
    
    \begin{lemma}
        $H_1(\Gamma)$ is a discrete subgroup of $\mathrm{PSL}(2,\mathbb C)$.
    \end{lemma}
    
    \begin{proof}
        We use the following general lemma, which has been proved in \cite{fang_rigidity_2007} p.1786.
        \begin{lemma}
        \label{lem:Hdiscret}
            Let $\mathcal{F}$ a foliation, on a smooth connected manifold $M$, with a transverse $(G,X)$-structure. 
            Denote by $D:\widetilde{M} \to X$ a developing map, and $H: \pi_1(M) \to G$ the corresponding holonomy representation, for this transverse $(G,X)$-structure. 
            Suppose that the union of the embedded leaves of $\mathcal{F}$ is dense in $M$ and that the fibers of $D$ are connected.\\
            Then $H(\pi_1(M))$ is a discrete subgroup of $G$.
        \end{lemma}
        In our case, as proved in Corollary \ref{cor:flotrproj}, the flow $(\varphi^t)_{t\in \mathbb R}$ admits a  transverse $$(\mathrm{PSL}(2, \mathbb C) \times \mathrm{PSL}(2, \mathbb C), (\mathbf{P}^1(\mathbb C) \times \mathbf{P}^1(\mathbb C))\setminus \text{gr}^H(u))$$structure such that the map $D$ is a developing map, and $H$ is the corresponding holonomy representation, for this structure.
        Since $(\varphi^t)_{t\in \mathbb R}$ is topologically transitive, the union of periodic orbits of $(\varphi^t)_{t\in \mathbb R}$ is dense in $M$. Since a periodic orbit is closed in $M$ thus embedded (see \cite{camacho_geometric_1985}), and by the fact that the fibers of $D$ are exactly the orbits of the lifted flow, then by the previous lemma, $H(\Gamma)$ is a discrete subgroup of $\mathrm{PSL}(2, \mathbb C) \times \mathrm{PSL}(2, \mathbb C)$.
        Therefore, $H_1(\Gamma)$ is discrete since, by the previous claim, it is homeomorphic to $H(\Gamma)$ via the homeomorphism 
        $\left \{ \begin{array}{ccl}
             H_1(\Gamma)& \to & H(\Gamma) \\
             h& \mapsto & (h,u \circ h \circ u^{-1})
        \end{array} \right..$
    \end{proof}
    \begin{lemma}
        $H_1(\Gamma)$ is a cocompact subgroup of $\mathrm{PSL}(2,\mathbb C)$.
    \end{lemma}
    \begin{proof}
        The proof is identical to that of Lemma 2.8 of \cite{fang_rigidity_2007}. See \cite{hatcher_algebraic_2002}, \cite{paulin_introduction_nodate} for references of algebraic topology.
    \end{proof}
    \begin{lemma}
    We can choose $D$ so that
    $D(\widetilde{M})=(\mathbf P^1(\mathbb C)\times \mathbf P^1(\mathbb C)) \setminus \Delta,$
    where $\Delta:=\{(x,x), \; x \in \mathbf{P}^1(\mathbb C)\}$ is the diagonal of $\mathbf{P}^1(\mathbb C)$, and that $H_1=H_2$.
    \end{lemma}
    \begin{proof}
        Since $u$ conjugates $H_1(\Gamma)$ and $H_2(\Gamma)$, by Mostow's rigidity theorem (see \cite{mostow_quasi-conformal_1968}), $u: \mathbf{P}^1(\mathbb C) \to \mathbf{P}^1(\mathbb C)$ is an element of $\mathrm{PSL}(2, \mathbb C)$.
        Consider now 
        \[D^*:=D_1^*\otimes D_2^*:=(I_d \otimes u^{-1})\circ D=D_1 \otimes (u^{-1}\circ D_2).\]
        Then $D_1^*$, $D_2^*$ and $D^*$ are respectively developing maps for the transverse $(\mathrm{PSL}(2, \mathbb C), \mathbf{P}^1(\mathbb C))$ structure of $\mathcal{F}^u$, the transverse $(\mathrm{PSL}(2, \mathbb C), \mathbf{P}^1(\mathbb C))$ structure of $\mathcal{F}^s$, and the transverse \[(\mathrm{PSL}(2, \mathbb C) \times \mathrm{PSL}(2, \mathbb C), \mathbf{P}^1(\mathbb C) \times \mathbf{P}^1(\mathbb C))\]
        structure for the flow $(\varphi^t)_{t \in \mathbb R}$.
        The holonomy representation $H_2^*$ corresponding to $D_2^*$ is 
        $u^{-1} \circ H_2 \circ u=H_1.$
        Moreover, 
            $D^*(\widetilde{M})=(I_d \otimes u^{-1})(D(\widetilde{M}))=(\mathbf{P}^1(\mathbb C)\times \mathbf{P}^1(\mathbb C)) \setminus\Delta.$
    \end{proof}
    \begin{lemma}
        There exists a finite covering $\overline{M} \xrightarrow[]{\overline{p}} M$ over $M$ whose fundamental group is discrete, torsion-free and cocompact.
    \end{lemma}
    \begin{proof}
        Selberg's lemma (see \cite{nica_three_nodate}) implies that if $\Gamma$ is a finitely generated subgroup of a linear group of characteristic zero, then $\Gamma$ admits a torsion-free subgroup of finite index.
        Also, by the Galois correspondence (see \cite{puttick_galois_2012}), if $M$ is a topological manifold and $H$ is a subgroup of $\pi_1(M)$ of finite index, then there exists a finite covering over $M$ whose fundamental group is $H$ and whose covering degree is the index $[\pi_1(M): H]$.
        In our case, since $M$ is compact, $\pi_1(M)=\Gamma$ is finitely generated, and because $\Gamma \cong H_1(\Gamma)$ is a subgroup of $\mathrm{PSL}(2,\mathbb C)$, then by the above, there exists a finite covering $\overline{M} \xrightarrow[]{p} M$ over $M$ whose fundamental group $\Gamma'$ is a finite index torsion-free subgroup of $\Gamma$.
        Therefore, $\Gamma'$ is also discrete. Since $\Gamma$ is cocompact and $\Gamma'$ is a discrete finite-index torsion-free subgroup of $\Gamma$, then the cohomological dimension of $\Gamma'$ is $3$ and $\Gamma'$ is thus cocompact.
    \end{proof}
     We conclude the proof of Theorem \ref{thm:classif}]
     The lift of the flow on $\overline{M}$ has also a transverse $(\mathrm{PSL}(2, \mathbb C) \times \mathrm{PSL}(2, \mathbb C), (\mathbf{P}^1(\mathbb C) \times \mathbf{P}^1(\mathbb C))\setminus \Delta)$-structure, whose developing map is $D$ and whose corresponding holonomy representation is $\restr{H}{\Gamma'}$. It is also Anosov and its strong/weak stable/unstable foliations are the lifts of the corresponding ones on $M$.
    Therefore, by lifting to a finite cover, we can assume $\Gamma$ is torsion free.
    We define 
    $V=\mathbb H^3 \diagup H_1(\Gamma).$
    Then $V$ is a compact hyperbolic smooth manifold. Consider the geodesic flow $(\psi_t)_{t\in \mathbb R}$ on $S^1V$. 
    In order to prove that $(\varphi^t)_{t\in \mathbb R}$ is $C^\infty$-orbit equivalent to $(\psi_t)_{t\in \mathbb R}$, we will need the following Proposition established in \cite{haefliger_groupoides_1984}.
    \begin{proposition}
    \label{lem:barbo}
        Let $(M_1, \mathcal{F}_1)$ and $(M_2, \mathcal{F}_2)$ two $C^\infty$-foliated connected manifolds each admitting a complete transverse $(G,X)$-structure. Assume:
        \begin{enumerate}[label=(\alph*)]
            \item The developing maps of these structures have connected fibers ;
            \item The $\pi_1(M_1)$-action on the leaf space $Q_{\widetilde{\mathcal{F_1}}}$ of the lifted foliation $\widetilde{\mathcal{F}_1}$ on the universal cover of $M_1$ is $C^\infty$-conjugate to the $\pi_1(M_2)$-action on the leaf space $Q_{\widetilde{\mathcal{F_2}}}$ of the lifted foliation $\widetilde{\mathcal{F}_2}$ on the universal cover of $M_2$ ;
            \item The holonomy covering of each leaf of $\mathcal{F}_1$ and $\mathcal{F}_2$ is contractible.
        \end{enumerate}
        Then there exists a smooth map $h: M_1 \to M_2$ satisfying:
        \begin{enumerate}[label=(\roman*)] 
            \item $h$ is a homotopy equivalence ;
            \item $h$ sends each leaf of $\mathcal{F}_1$ onto a leaf of $\mathcal{F}_2$ ;
            \item $h$ sends two different leaves of $\mathcal{F}_1$ to two different leaves of $\mathcal{F}_2$ ;
            \item $h$ is transversally a $C^\infty$-diffeomorphism conjugating the transverse $(G,X)$-structure of $\mathcal{F_1}$ and that of $\mathcal{F}_2$.
        \end{enumerate}
    \end{proposition}
    \noindent The $H_1(\Gamma)$-action on $(\mathbf{P}^1(\mathbb C) \times \mathbf{P}^1(\mathbb C))\setminus \Delta$ induced by the isomorphism $H_1: \Gamma \to H_1(\Gamma)$ and coming from the identification of the leaf space of the lifted flow of $(\varphi^t)_t$ with $(\mathbf{P}^1(\mathbb C) \times \mathbf{P}^1(\mathbb C))\setminus \Delta$ is given, for $\gamma \in \Gamma$ and $z \in (\mathbf{P}^1(\mathbb C) \times \mathbf{P}^1(\mathbb C))\setminus \Delta$, by:
    \[H_1(\gamma) \cdot z:=\overline{D_M}(\gamma\cdot {D_M}^{-1}(z))=\overline{D_M}({D_M}^{-1}(H(\gamma) z))=H(\gamma)z\]
   so it coincides with the $H_1(M)$-action on $(\mathbf{P}^1(\mathbb C) \times \mathbf{P}^1(\mathbb C))\setminus \Delta$ coming from the identification of the leaf space of the lifted geodesic flow on $S^1V$ with $(\mathbf{P}^1(\mathbb C) \times \mathbf{P}^1(\mathbb C))\setminus \Delta$.
    Now we prove that the holonomy covering of each orbit of $(\varphi^t)_{t \in \mathbb R}$ is homeomorphic to $\mathbb R$ and is thus contractible. This will prove the corresponding statement for the geodesic flow on $S^1V$. First, a non-periodic orbit of $(\varphi^t)_{t \in \mathbb R}$ is diffeomorphic to $\mathbb R$. Therefore, its fundamental group is $0$ so its holonomy covering is $\mathbb R$. If we consider now a periodic orbit $L$, and a loop $\gamma$ in $L$ non-homotopic to a constant path, then its holonomy is non-trivial by Corollary \ref{cor:holofloFsu} since the flow $(\varphi^t)_{t \in \mathbb R}$ is Anosov. Therefore, its holonomy covering is homeomorphic to the universal covering space of $L$, that is $\mathbb R$. \\
    Let $h: M \to S^1V$ given by Proposition \ref{lem:barbo}.
    The problem if we want to conclude is that the restriction of $h$ to an orbit of $(\phi^t)_{t \in \mathbb R}$ is not necessarily injective, so $h$ does not give a priori a $C^\infty$-orbit equivalence between the flow $(\varphi^t)_{t\in \mathbb R}$ and the geodesic flow $(\psi_t)_{t \in \mathbb R}$ on $S^1V$.
    As mentioned in \cite{barbot_caracterisation_1995} and \cite{ghys_deformations_1992}, let a continuous function $u:\mathbb R \times M \to \mathbb R$ defined by the functional equation:
    \[\forall t \in \mathbb R, \; \forall x \in M, \quad h(\varphi^t(x))=\psi^{u(t,x)}(h(x)).\]
    Then by a technical lemma (see \cite{barbot_caracterisation_1995}) there exists $T \in \mathbb R$ such that for every $x \in M$, $u(T,x) \neq 0$.
    Therefore, if we let, for $x \in M$, 
    \[u_T(x):=\frac{1}{T}\int_0^T u(s,x)ds,\]
    then the map $h_T: M\to S^1V$ given, for $x \in M$, by:
    \[h_T(x)=\psi^{u_T(x)}(h(x)),\]
    is a $C^\infty$-diffeomorphism satisfying the same properties as $h$ and is thus a $C^\infty$-orbit equivalence between $(\varphi_t)_{t \in \mathbb R}$ and the geodesic flow on $S^1V$.
    We have thus proved that $(\varphi_t)_{t \in \mathbb R}$ is, after lifting to a certain finite cover over $M$, $C^\infty$-orbit equivalent to the geodesic flow on the unit tangent bundle of a smooth compact hyperbolic $3$-manifold.
\end{proof}
\begin{corollary}
    Let $(\varphi^t)_{t \in \mathbb R}$ a transversely holomorphic Anosov flow on a smooth compact manifold $M$ of dimension five. Suppose $(\varphi^t)_{t \in \mathbb R}$ topologically transitive.\\
    Then the strong stable and strong unstable distribution are $C^\infty$.
\end{corollary}
\begin{proof}
    By Theorem \ref{thm:classif}, we can lift the flow to a finite covering space of $M$ such that its orbit, strong stable and strong unstable foliations correspond to that of a suspension of a holomorphic automorphism of a complex torus, or that of a geodesic flow on a compact hyperbolic $3$-manifold. 
    Since their strong distributions are $C^\infty$, the same can be said for $(\varphi^t)_{t \in \mathbb R}$
\end{proof}

\appendix
\section{ }
We give the proofs of Lemma \ref{lem:fact1} and Lemma \ref{lem:fact2} mentioned in subsection \ref{subsec:6.2} during the proof of Theorem \ref{thm:FsTrProj}. They are just topological facts concerning $\mathbb R^p$.
\begin{proof}[Proof of Lemma \ref{lem:fact1}]
        Let $|\cdot|$ be the euclidean norm on $\mathbb R^p$.
        Denote, for $i \in \llbracket1,q \rrbracket$, $d_i=\text{diam}(K_i)$. Since $K$ is compact, by continuity there exist $x_1, x_q \in K$ such that $\text{diam}(K)=|x_1-x_q|$. Let 
        \[r=\min\{i \in \llbracket1,q \rrbracket, \; x_1 \in K_i\}, \quad s=\max\{i \in \llbracket1,q \rrbracket, \; x_q \in K_i\}.\]
        We can assume $r\leq s$ (otherwise, $r$ is defined by the maximum, $s$ by the minimum and necessarily $r \geq s$).
        Then 
        \[ \sum_{i=1}^q d_i =|x_1-x_q|\leq \text{diam}(\bigcup_{i=r}^s K_i)\leq \sum_{i=r}^sd_i.\]
        Since each $K_i$ is not reduced to a point, it comes $r=1$ and $s=q$.
        Now, denote, for $i \in \llbracket1,q-1 \rrbracket$, by $y_{i,i+1}$ an arbitrary element of $K_i \cap K_{i+1}$.
        It follows 
        \[\sum_{i=1}^q d_i =|x_1-x_q|\leq |x_1 - y_{1,2}|+ \sum_{i=1}^{q-2} |y_{i,i+1}-y_{i+1,i+2}| +|y_{q-1,q} - x_{q}|,\]
        so
        \begin{align*}
            \forall y_{1,2} \in K_1 \cap K_2,& \quad |x_1-y_{1,2}|=d_1,\\
            \forall i\in \llbracket 1,q-2 \rrbracket, \; \forall x \in K_i \cap K_{i+1}, \;\forall y \in K_{i+1} \cap K_{i+2},& \quad |x-y|=d_{i+1},\\
            \forall y_{q-1,q} \in K_{q-1} \cap K_q, &\quad |y_{q-1,q}-x_q|=d_q.
        \end{align*}
        Therefore, the case of equality in the triangular inequality gives that for any such points, $x_1, y_{1,2}, y_{2,3}, \cdots, y_{q-1,q}, x_q$ are aligned in that order. This imply that for every $i \in \llbracket1, q-1 \rrbracket$, $K_i\cap K_{i+1}$ is a point denoted by $x_{i, i+1}$. We also note $x_{0,1}:=x_1$ and $x_{q,q+1}:=x_q$. Since each $K_j$ is not a reduced to point, $(x_{i,i+1})_{0\leq i \leq q}$ is a sequence of distinct points.\\
        We now prove that for every $i\in \llbracket 1,q-2 \rrbracket$ and $l\in \llbracket 2,q-i \rrbracket$, $K_i \cap K_{i+l}=\emptyset$, which concludes thanks to the previous claim. 
        Otherwise, let $i\in \llbracket 1,q-2 \rrbracket, \;l\in \llbracket 2,q-i \rrbracket$ and $z \in K_i \cap K_{i+l}$.
        Then 
        \begin{align*}
            \sum_{j=i}^{i+l} d_{j} = \sum_{j=i}^{i+l} |x_{j-1,j}-x_{j,j+1}|&=\left |\sum_{j=i}^{i+l} x_{j-1,j}-x_{j,j+1}\right|=|x_{i-1,i}-x_{i+l,i+l+1}|\\
            &\leq |x_{i-1,i}-z| + |z-x_{i+l,i+l+1}|\leq d_i + d_{i+l}.
        \end{align*}
        Therefore $d_{i+1}=\cdots=d_{i+l-1}=0$ which is absurd.
    \end{proof}
    \begin{proof}[Proof of Lemma \ref{lem:fact2}]
        For $\epsilon>0$, and $A \subset \mathbb R^p$, let $A(\epsilon)=\{x \in \mathbb R^p, d(x,A)<\epsilon\}$. By assumption, for every $\epsilon >0$ and $\alpha >0$, 
        $\mathcal{C}^\alpha(\epsilon)=(C_i^\alpha(\epsilon))_{1\leq i\leq N_\alpha}$ is an open covering of $C(\epsilon)$. Since $C(\epsilon)$ is an open subset of $\mathbb R^p$, its topological dimension is $p$. 
        Therefore, there exists an open covering $\mathcal{U_\epsilon}=(U_\epsilon^{(i)})_{1\leq i\leq M_\epsilon}$ of $C(\epsilon)$ such that for every finite refinement of $\mathcal{U}_\epsilon$, there exists a point of $\mathbb R^p$ belonging to $p+1$ distinct open subsets of the refinement.
        Fix $\epsilon >0$ and take $\alpha >0$ small enough so that $\mathcal{C}^\alpha_\epsilon$ is a (finite) refinement of $\mathcal{U}_\epsilon$. As a result, there exists a point $x_\epsilon$ belonging to $p+1$ distinct elements of $\mathcal{C}^\alpha(\epsilon)$, which we note $C^\alpha_{i^{(\epsilon)}_1}(\epsilon), \cdots , C^\alpha_{i^{(\epsilon)}_{p+1}}(\epsilon)$ where each $i^{(\epsilon)}_j$ for $j\in \llbracket1,p+1 \rrbracket$ is in $\llbracket1, N_\alpha\rrbracket$.
        If $\epsilon=\frac{1}{n}>0$, then by compactness of $\llbracket1, N_\alpha\rrbracket$, after extraction, for $n$ large enough, we can assume that each sequence $(i^{(\frac{1}{n})}_j)_n$ for $j\in \llbracket1,p+1 \rrbracket$ is constant and equals $i_j \in \llbracket1, N_\alpha\rrbracket$. After extraction, since $C$ is compact, we can also assume that $(x_{\frac{1}{n}})_n$ converges to some $x \in C(1)$. By definition, for $n$ large enough and $j \in \llbracket1, p+1 \rrbracket$, 
        \[ d(x_{\frac{1}{n}}, C^\alpha_j)\leq \frac{1}{n}.\]
        As a matter of fact, since each $ C^\alpha_j$ is closed and the map $d(\cdot,  C^\alpha_j)$ is continuous, $x$ belongs to $C^\alpha_1, \cdots ,  C^\alpha_{p+1}$ which is the result.
    \end{proof}

\bibliographystyle{abbrv}
\bibliography{article_1}
\end{document}